\documentclass{amsart}

\usepackage{amsmath}
\usepackage{amsthm}
\usepackage{amssymb}

\usepackage{bm}
\usepackage{color}
\usepackage{enumerate}
\usepackage{graphicx}

\def\BB{{\mathcal B}}

\def\EE{{\mathcal E}}

\def\LL{{\mathcal L}}

\def\NN{{\mathcal N}}

\def\PP{{\mathcal P}}
\def\QQ{{\mathcal Q}}

\def\UU{{\mathcal U}}
\def\VV{{\mathcal V}}
\def\bVV{{\boldsymbol{\mathcal V}}}

\def\bUU{{\boldsymbol{\mathcal U}}}
\def\WW{{\mathcal W}}
\def\XX{{\mathcal X}}
\def\YY{{\mathcal Y}}
\def\ZZ{{\mathcal Z}}
\def\M{{\mathbb M}}

\newcommand{\ds}{\displaystyle}

\newtheorem{thm}{Theorem}[section]
\newtheorem{lem}[thm]{Lemma}

\newtheorem{cor}[thm]{Corollary}
\newtheorem{defi}[thm]{Definition}
\newtheorem{prop}[thm]{Proposition}
\newtheorem{rk}[thm]{Remark}
\newtheorem{nota}[thm]{Notation}

\renewcommand{\theequation}{\arabic{section}.\arabic{equation}}

\newcommand{\beqn}{\begin{equation}}
\newcommand{\eeqn}{\end{equation}}
\newcommand{\bear}{\begin{eqnarray}}
\newcommand{\eear}{\end{eqnarray}}
\newcommand{\bean}{\begin{eqnarray*}}
\newcommand{\eean}{\end{eqnarray*}}

\newcommand{\rr}{{\mathbb{R}}}
\newcommand{\rd}{{\rr^3}}
\newcommand{\R}{{\mathbb{R}}}

\newcommand{\nn}{{\mathbb{N}}}

\newcommand{\e}{\varepsilon}
\newcommand{\en}{{\varepsilon_n}}
\newcommand{\eps}{\varepsilon}
\newcommand{\vip}{\vskip.2cm}
\newcommand{\indiq}{\hbox{\rm 1}{\hskip -2.8 pt}\hbox{\rm I}}
\newcommand{\E}{\mathbb{E}}

\newcommand{\intot}{\int_0^t }

\newcommand{\intrd}{\int_{\rr^3}}
\newcommand{\cM}{{\M}}
\newcommand{\cL}{{\mathcal L}}
\newcommand{\cP}{{\mathcal P}}
\newcommand{\cS}{{\mathcal S}}
\newcommand{\cK}{{\mathcal K}}
\newcommand{\cB}{{\mathcal B}}

\newcommand{\cF}{{\mathcal F}}
\newcommand{\cI}{{{\mathcal I}}}

\newcommand{\tg}{{{\tilde \beta}}}

\newcommand{\tX}{{{\tilde X}}}
\newcommand{\tY}{{{\tilde Y}}}
\newcommand{\bX}{{{\hat X}}}
\newcommand{\bY}{{{\hat Y}}}

\newcommand{\Psym}{{\bf P}_{\! sym}}

\newcommand{\lnm}{\left< \left<}
\newcommand{\rnm}{\right> \right>}

\DeclareMathOperator{\diver}{div}
\DeclareMathOperator{\Tr}{Tr}

\newcommand{\Black}{\color{black}}

\newcommand{\poubelle}[1]{}

\begin{document}

\title[On the Landau equation]
{Propagation of chaos for the Landau equation with moderately soft potentials}

\author{Nicolas Fournier}
\author{Maxime Hauray}

\address{N. Fournier: Laboratoire de Probabilit\'es et Mod\`eles al\'eatoires, UMR 7599, UPMC, Case 188,
4 pl. Jussieu, F-75252 Paris Cedex 5, France.}

\email{nicolas.fournier@upmc.fr}

\address{M. Hauray: Institut de Math\'ematiques de Marseille UMR 7353, Universit\'e d'Aix-Marseille, 39 rue F. Joliot Curie, F-13453 Marseille Cedex 13.}

\email{maxime.hauray@univ-amu.fr}



\begin{abstract}
We consider the 3D Landau equation for moderately soft potentials ($\gamma\in(-2,0)$ with the usual notation)
as well as a stochastic system of $N$ particles approximating it.
We first establish some strong/weak stability estimates for the Landau equation, which are satisfying only 
when $\gamma \in [-1,0)$.
We next prove, under some appropriate conditions on the initial data, the so-called 
propagation of molecular chaos, i.e. that the empirical measure of the particle system
converges to the unique solution of the Landau equation. 
The main difficulty
is the presence of a singularity in the equation.
When $\gamma \in (-1,0)$, the strong-weak uniqueness estimate allows us to use a coupling argument and to obtain
a rate of convergence.
When $\gamma \in (-2,-1]$, we use the classical martingale method introduced by McKean. 
To control the singularity, we have to take advantage of the regularity provided by the entropy dissipation.
Unfortunately, this dissipation is too weak for some (very rare) aligned configurations.
We thus introduce a perturbed system with an additional noise, show the propagation 
of chaos for that perturbed system and finally prove that the additional noise is almost never used in the limit.
\end{abstract}

\maketitle

%
%
%
%
%
%
%
%

\section{Introduction and main results}
\setcounter{equation}{0}
\label{sec:intro}

\subsection{The Landau equation}

The 3D homogeneous Landau equation for moderately soft potentials writes
\begin{align} \label{HL3D}
\partial_t f_t(v) = \frac 1 2 
\diver_v \Big( \int_{\rr^3} a(v-v_*)[ f_t(v_*) \nabla f_t(v) - f_t(v) \nabla f_t(v_*)
]\,dv_* \Big),
\end{align}
where the initial distribution $f_0 : \R^3 \to \R$ is given. The unknown $f_t:\R^3\mapsto\R$ stands 
for the velocity-distribution in a plasma. The matrix $a:\R^3\mapsto 
\cM_{3 \times 3}(\rr)$ is symmetric nonnegative and given by 
\[
a(v) = |v|^{2 + \gamma} \Bigl( I - \frac{v \otimes v}{|v|^2} \Bigr)
\]
for some $\gamma \in (-2,0)$. We will also use the notation
\[
b(v)=\diver a(v)=-2 |v|^\gamma v.
\]

This equation, with $\gamma=-3$, replaces the Boltzmann equation when particles are subjected to a Coulomb
interaction. It was derived by Landau in 1936. Physically, only the case $\gamma=-3$ is really interesting:
it is explained in~\cite{bps} that the case $\gamma=-3$ is the only one that you can obtain from a 
particle system in a suitable {\it weak coupling} limit, even if the interaction potential has a finite range 
(a fact that was already discovered by Bogolyubov).

When $\gamma \in (-3,1]$, the Landau equation can be seen as an approximation of the corresponding Boltzmann 
equation in the asymptotic of {\it grazing collisions}.
There is a huge literature on that model. See Villani \cite{v:nc} for the existence theory when 
$\gamma \in [-3,1]$ and links with the Boltzmann equation, Arsen'ev-Peskov \cite{ap} for a local existence
result when $\gamma=-3$, see \cite{fou} for a local uniqueness result when $\gamma=-3$, \cite{fg} 
for a global well-posedness result when $\gamma \in (-2,0)$. Well-posedness, regularity and large-time
behavior are studied when $\gamma \in (0,1]$ by Desvillettes-Villani \cite{dv,dv2}. A probabilistic
interpretation was introduced by Funaki in \cite{f}. We also refer to the book of Villani \cite{v:h}
for a long review on kinetic models, including this one.
Let us only mention a few important properties: the Landau equation preserves mass, 
momentum, kinetic energy and dissipates entropy.

\subsection{Some notation}

We denote by $\cP(\R^3)$, the set of probability measures on $\R^3$.
When $f\in\cP(\R^3)$ has a density, we also denote by $f\in L^1(\rd)$ this density.
For $q>0$, $\cP_q(\R^3)$ stands for the set of all $f\in\cP(\R^3)$ such that $m_q(f)=\intrd |v|^q f(dv)<\infty$.

\vip
For $f\in\cP(\R^3)$, we introduce the notation
\begin{gather}\label{bfaf}
b(f,v):=\intrd b(v-v_*)f(dv_*),\quad a(f,v):=\intrd a(v-v_*)f(dv_*), \quad 
\sigma(f,v):=\Big(a(f,v)\Big)^{\frac12}.
\end{gather}
Observe that $a(f,v)$ is symmetric nonnegative (since $a(v)$ is for all $v \in \R^3)$, so that it indeed 
admits a unique nonnegative 
symmetric square root. The Landau equation \eqref{HL3D} can be rewritten
\begin{align}\label{alt}
\partial_t f_t(v) = \frac 1 2 \diver_v \bigl(a(f_t,v)\nabla_v f_t(v) - b(f_t,v) f_t(v)\bigr).
\end{align}

We will denote by $H$ the entropy functional: for $f\in \cP_q(\R^3)$, for some $q >0$ we define 
$H(f)=\intrd f(v) \log f(v) dv \in  (-\infty, + \infty]$ if $f$ has a density and  $H(f)=\infty$ else. 
Under the moment assumption, that entropy is always well defined, see for instance~\cite[Lemma 3.1]{hm}.

\vip
We shall also use similar notations for probability distributions of systems of $N$ particles. 
Precisely, for $N \geq 1$, $\cP((\rd)^N)$ stands for the set of probability measures on $(\rd)^N$.
When $F^N \in\cP((\rd)^N)$ has a density, we also denote by $F^N \in L^1((\rd)^N)$ this density.
We write $\cP_{sym}((\rd)^N)$ for the set of all exchangeable elements of $\cP((\rd)^N)$.
For $F^N \in \cP_{sym}((\rd)^{N})$, we introduce $m_q(F^N):= \int_{(\rd)^N} |v_1|^q F^N(dv_1,\ldots,dv_N)$
and define $\cP_q((\rd)^N)=\{ F^N \in \cP_{sym}((\rd)^N) \; : \; m_q(F^N)<\infty\}$.
Finally, we introduce the entropy  of $F^N \in \cP_q((\rd)^N)$ for some $q>0$ by setting 
\[
H(F^N) := 
\begin{cases}
N^{-1} \int_{(\rd)^N} F^N(v_1,\ldots,v_N)\log F^N(v_1,\ldots,v_N) dv_1\ldots dv_N & \text{if $F^N$ has a density},\\
\ds + \infty & \text{otherwise}.
\end{cases}
\]

\vip
We will use the MKW (Monge-Kantorovich-Wasserstein) distance, see Villani 
\cite{v:t} for many details: for $f,g\in\cP_2(\rd)$,
$$
W_2(f,g)= \inf \Big\{ \Big(\int_{\rd\times\rd} |v-w|^2 R(dv,dw)\Big)^{1/2} \; : \; R \in \Pi(f,g)   \Big\},
$$
where $\Pi(f,g) = \bigl\{    R \in \cP(\rr^3\times\rr^3) \; : \; R  \text{ has marginals }  
f \text{ and } g \bigr\}$.

\vip
We now recall what is usually called {\it propagation of molecular chaos}.

\begin{defi}\label{dfc}
Let $\YY$ be a random variable taking values in a Polish space $E$ with law $g$ 
and, for each $N\geq 2$,
a family $(\YY_1^N,\dots,\YY^N_N)$ of exchangeable $E$-valued random variables.
The sequence $(\YY_1^N,\dots,\YY^N_N)$ is said to
be $\YY$-chaotic (or $g$-chaotic) if one of the three equivalent conditions is satisfied:

\vip

(i) $(\YY^N_1,\YY^N_2)$ goes in law to $g\otimes g$
as $N\to\infty$;

\vip

(ii) for all $j \ge 1$, $(\YY^N_1, \dots, \YY^N_j)$ goes in law to $g^{\otimes j}$
as $N\to\infty$;

\vip

(iii) $N^{-1} \sum_1^N \delta_{\YY_i^N}$ goes in probability to $g$ as $N\to\infty$.
\end{defi}

We refer for instance to Sznitman \cite{SSF} for the equivalence of the three conditions 
and to~\cite{hm} for quantitative versions of that equivalence. Propagation of chaos holds for a particle system 
(towards the solution $(f_t)_{t\geq 0}$ of its limit equation) if starting with initial data $f_0$-chaotic, 
the particles are $f_t$-chaotic for all times $t$.
And \emph{trajectorial} propagation of chaos holds when the trajectories of the particles are $Y$-chaotic, 
for a suitable process $Y$ associated to the limit equation. \Black

\vip

We will use the generic notation $C$ for all positive constant appearing in the sequel. 
When needed, we will indicate in subscript the quantities on which it depends.

\subsection{Well-posedness and strong/weak stability}

We first recall that for 
$\alpha \in (-3,0)$ and $f\in \cP(\R^3)\cap L^p(\R^3)$ with $p>3/(3+\alpha)$, it holds that
\begin{align}\label{weakb2}
\sup_{v\in\R^3} \intrd |v-v_*|^\alpha f(v_*)dv_* \leq 1+ C_{\alpha,p}\|f\|_{L^p}.
\end{align}
This is easily checked: write that $\intrd |v-v_*|^\alpha f(v_*)dv_* \leq 1+ \int_{v_*\in B(v,1)} |v-v_*|^\alpha 
f(v_*)dv_*$ and use the H\"older inequality.

\vip

We next define the (classical) notion of weak solutions we use.

\begin{defi} Let $\gamma \in (-2,0)$. We say that $f$ is a weak solution to \eqref{HL3D} if it 
satisfies:

\vip

(i) $f \in L^\infty_{loc}([0,\infty),\cP_2(\rd))$,

\vip

(ii) if $\gamma \in (-2,-1)$, $f \in L^1_{loc}([0,\infty),L^p(\rd))$ for some $p>3/(4+\gamma)$,

\vip

(iii) for all $\varphi\in C^2_b(\rr^3)$, all $t\geq 0$, 
\begin{align}\label{wl}
\intrd \varphi(v)f_t(dv) = \intrd \varphi(v)f_0(dv) + \intot \intrd \intrd L\varphi(v,v_*) 
f_s(dv)f_s(dv_*) ds,
\end{align}
where
$$
L\varphi(v,v_*):= \frac 1 2 \sum_{k,l=1}^3 a_{kl}(v-v_*)\partial^2_{kl}\varphi(v)+ \sum_{k=1}^3
b_{k}(v-v_*)\partial_{k}\varphi(v).
$$
\end{defi}

Remark that every term is well-defined in \eqref{wl} under our assumptions on $f$ and $\varphi$, since
\begin{equation} \label{weakb1}
|L\varphi(v,v_*)|\leq C_\varphi (|v-v_*|^{\gamma+1}+|v-v_*|^{\gamma+2}).
\end{equation}
If $\gamma \in [-1,0)$, we have $|L\varphi(v,v_*)| \leq C_\varphi(1+|v|^2+|v_*|^2)$ so that condition (i) 
is sufficient, 
while if $\gamma \in (-2,-1)$, we have $|L\varphi(v,v_*)| \leq C_\varphi(1+|v|^2+|v_*|^2+|v-v_*|^{\gamma+1})$,
so that conditions (i) and (ii) are enough: use \eqref{weakb2} with $\alpha=\gamma+1$.

\vip

For $\gamma \in (-2,0)$ we set 
\beqn \label{def:pq}
q(\gamma):=\frac{\gamma^2}{2+\gamma}, \quad
p_1(\gamma) := \frac3{3+\gamma} \quad \text{and, for }  q > q(\gamma), \quad
p_2(\gamma,q) := \frac{3q-3\gamma}{q-3\gamma}.
\eeqn
It can be checked that $1 < p_1(\gamma)<  p_2(\gamma,q) <3$. 
Let us recall the well-posedness result of \cite{fg}.

\begin{thm}[Corollary 1.4 in  \cite{fg}]  \label{thfg}
Let $\gamma \in (-2,0)$ and $f_0 \in \cP_2(\rr^3)$ satisfy $H(f_0)<\infty$ and $m_q(f_0)<\infty$
for some $q >q(\gamma)$.
Then \eqref{HL3D} has a unique weak solution 
$f \in L^\infty_{loc}([0,\infty),\cP_2(\rr^3))\cap  L^1_{loc}([0,\infty),L^p(\rr^3))$ for some $p>p_1(\gamma)$.
Moreover, this unique solution satisfies 
$m_2(f_t)=m_2(f_0)$ and $H(f_t)\leq H(f_0)$ for all $t\geq 0$ and 
belongs to $L^1_{loc}([0,\infty),L^p(\rr^3))$ for all $p \in (p_1(\gamma), p_2(\gamma,q))$.
\end{thm}

We first state some weak/strong stability estimates and improve the above uniqueness result.

\begin{thm}\label{wp}
Let $\gamma \in (-2,0)$ and $f_0 \in \cP_2(\rr^3)$ satisfy also $H(f_0)<\infty$ and $m_q(f_0)<\infty$ 
for some $q >q(\gamma)$. Let $f$ be the unique weak solution of \eqref{HL3D} built in Theorem \ref{thfg}.

\vip

(i) Assume that $\gamma \in (-1,0]$. For any other weak solution
$g \in L^\infty_{loc}([0,\infty),\cP_2(\rr^3))$ to \eqref{HL3D} starting from $g_0 \in \cP_2(\rr^3)$, any 
$p \in (p_1(\gamma), p_2(\gamma,q))$, any $t\geq 0$,
\[
W_2^2(f_t,g_t) \leq W_2^2(f_0,g_0)\exp\Big(C_{\gamma,p} \intot (1+\|f_s\|_{L^p}) ds\Big).
\]
We thus have uniqueness for \eqref{HL3D} in the class
of all measure solutions in $L^\infty_{loc}([0,\infty),\cP_2(\rr^3))$.

\vip

(ii) Assume that $\gamma \in (-2,-1]$. For any 
$p \in (p_1(\gamma), p_2(\gamma,q))$, any
\begin{equation}\label{rr}
r> \frac{2p-3}{(3+\gamma)(p-1)-1}
\end{equation}
and any other weak solution 
$g \in L^\infty_{loc}([0,\infty),\cP_2(\rr^3))\cap L^1_{loc}([0,\infty),L^r(\rr^3))$ to \eqref{HL3D},
\[
W_2^2(f_t,g_t) \leq W_2^2(f_0,g_0) \exp\Big(C_{\gamma,p,r} \intot (1+\|f_s\|_{L^p}+\|g_s\|_{L^r}) ds\Big).
\]
In particular, for any $r>[3(q+ \gamma)]/[5q + 2 \gamma q + 3 \gamma]$, 
we have uniqueness for \eqref{HL3D} in the class
$L^\infty_{loc}([0,\infty),\cP_2(\rr^3)) \cap L_{loc}^1([0,\infty), L^r(\rr^3))$
(it suffices to let $p\uparrow p_2(\gamma,q)$ in \eqref{rr}).
\end{thm}

When $\gamma \in (-1,0)$, 
we thus prove the uniqueness in the class
of all measure solutions in $L^\infty_{loc}([0,\infty),\cP_2(\rr^3))$.
This is quite satisfying and interesting for the well-posedness theory, 
but there is another important consequence:
we will be able to apply (up to some fluctuations) the stability result
to the empirical measure of an associated particle system. Of course, 
such an empirical measure has no chance to belong to some $L^p$ space with $p>1$,
but we will use an approximated version of this stability principle involving discrete $L^p$
norms.

\vip

When $\gamma \in (-2,-1]$, the strong/weak estimate is of course less satisfying, 
since we do not manage to completely get rid of the
regularity assumptions on $g$. The uniqueness we deduce is slightly better than that stated in
\cite[Corollary 1.4]{fg}, but the stability result cannot be applied
to the empirical measure of a particle system.

\subsection{Entropy dissipation and a priori bounds.}

The fact that {\it smooth} solutions to the Landau equation~\eqref{HL3D} belongs to 
$L^1_{loc}([0,\infty),L^p(\rd))$ is a consequence of the entropy dissipation. We sketch here the 
argument for the sake of completeness and also because we will use, in the proof of Theorem~\ref{mr},
a similar strategy to obtain some regularity estimate on the particle system. 
Precisely, the entropy dissipation for a solution $f_t$ of the Landau equation~\eqref{HL3D} writes
\[
\frac d {dt} H(f_t) = - \intrd\intrd  \bigl( \nabla \log f_t(v) - \nabla \log f_t(v_*) \bigr)^*
a(v-v_*) \bigl( \nabla \log f_t(v) - \nabla \log f_t(v_*) \bigr) 
 \, f_t(dv) f_t(dv_*)
\]
(where the superscript $*$ stand for the transposition), which is certainly non-negative thanks to the 
non-negativity of $a$. Using the drift and diffusion introduced in~\eqref{bfaf}, the dissipation of 
entropy may be rewritten
\[
\frac d {dt} H(f_t) = - \intrd  \nabla \log f_t(v) a(f_t,v) \nabla \log f_t(v)   \, f_t(dv)
- \intrd \diver b(f_t,v) \, f_t(dv).
\]
Then since the entropy of $f_t$ is decreasing and its second moment constant, Lemma~\ref{ellip1}
(taken from~\cite[Proposition 4]{dv}) shows that the first term in the r.h.s. controls a weighted Fisher 
information $I_\gamma(f_t) := \intrd |\nabla \log f_t|^2 (1+ |v|)^\gamma \,f_t(dv)$. This in turn allows us to control 
the $L^p$-norm of $f$ (for some values of $p$), 
thanks to Lemma~\ref{lem:FishInteg1}, provided we have sufficiently many moments.
Finally, a $L^p$-norm of $f_t$ with $p$ large enough is sufficient to bound $\| \diver b(f_t) \|_\infty$.
All in all, if $f_0$ has a finite entropy and moment of order $q$,
the entropy dissipation leads, for any $p \in \bigl(p_1(\gamma), p_2(\gamma,q) \bigr)$, 
to
\[
\int_0^T \| f_t\|_{L^p} \,dt \le C_{\gamma,p,q,T}.
\]

\subsection{The particle system}
We now consider the following toy model: we have $N$ particles characterized by their velocities $\VV_i^N$,
solving the following system of $\rr^3$-valued SDEs
\begin{equation}\label{ps}
\forall i=1,\ldots,N,\quad  
\VV^N_i(t) =\VV^N_i(0)+  \intot b\bigl(\tilde\mu^N_s,\VV^N_i(s) \bigr)ds 
+  \intot \sigma\bigl(\tilde\mu^N_s,\VV^N_i(s) \bigr) d\BB_i(s),
\end{equation}
with the notation \eqref{bfaf}.
Here $(\BB_i(t))_{i=1,\dots,N,t\geq 0}$ is an independent family of $3D$ standard Brownian motions
independent of $(\VV^N_i(0))_{i=1,\dots,N}$ and finally, for some $\eta_N\in(0,1)$,
\begin{equation*}
\tilde \mu^N_t =\mu^N_t \star \phi_{\eta_N} \quad \hbox{where} \quad \mu^N_t= \frac 1 N \sum_1^N \delta_{\VV^N_i(t)} 
\quad \hbox{and} \quad \phi_\eta(x) = \frac{3}{4 \pi \eta^3}\indiq_{\{|x|<\eta\}}.
\end{equation*}

We could probably also study the same system without the smoothing by convolution with $\phi_{\eta_N}$.
But without this smoothing, the particle system is not clearly well-posed. Since the paper is already 
technical enough, we decided not to study this (possibly difficult) issue.
However, this is not really a limitation, {\it since in all the results below, we allow $\eta_N$
to tend to $0$ as fast as one wants.}

\begin{prop}\label{wpps}
For any $N\geq 1$, any initial condition $(\VV^N_i(0))_{i=1,\dots,N}$, any $\eta_N \in (0,1)$, \eqref{ps} has a 
pathwise unique strong solution $(\VV^N_i(t))_{i=1,\dots,N, t\geq 0}$.
\end{prop}

The main topic of this paper is to show that, provided $\eta_N\to 0$ and 
under suitable conditions on $(\VV^N_i(0))_{i=1,\dots,N}$,
the empirical measure $\mu^N_t$ converges, as $N\to \infty$, 
to the weak solution $f_t$ of \eqref{HL3D} built in Theorem \ref{thfg}.
When $\gamma=0$, the coefficients $a,b$ are smooth and such a convergence has been proved by
Fontbona-Gu\'erin-M\'el\'eard \cite{fgm}, Carrapatoso \cite{c}, see also \cite{fou1}.
In a work in preparation \cite{bfg}, Bolley-Guillin-Fournier obtain some results
when $\gamma\in (0,1)$.
In \cite{c} and \cite{bfg}, a slightly more natural particle system, which a.s. preserves momentum
and kinetic energy, is considered. We have not been able to study this system with the present
technique, although the difference of its behavior seems very light.

\vip

When $\gamma<0$, the situation is rather more complicated, due to the singularity of the coefficients
$a,b$. To our knowledge, there are no available results in that context, except the one of 
Miot-Pulvirenti-Saffirio \cite{mps}, where a partial result has been obtained when $\gamma=-3$:
they prove the convergence to the Landau hierarchy, which unfortunately does not allow one to conclude.
Here we propose two methods.

\vip

When $\gamma \in (-1,0)$, we handle a direct computation, mimicking the strong/weak stability study, and
we get a result with a rate of convergence. We have to use a blob approximation of the empirical
measure, in the spirit of \cite{HaurayJabin}.

\vip

When $\gamma \in (-2,0)$, we use that the dissipation of entropy of the particle system
provides some regularity enough to control the singularity. By this way, we obtain a convergence result,
without rate, by using a tightness/uniqueness principle. We follow here \cite{fhm1}, where we studied
a similar problem for the 2D Navier-Stokes equation. However, a major difficulty appears: the 
entropy dissipation is actually not sufficient, due to the lack of ellipticity of the matrix $a$.
The diffusion coefficients in \eqref{ps} may degenerate for almost aligned configurations.
We thus have to show such configurations almost never appear.

\subsection{A convergence result with rate for $\gamma \in (-1,0)$}

In that case the singularity is not too large, and we can use in spirit our
strong/weak uniqueness result to study the propagation of chaos.

\begin{thm}\label{mrrate}
Assume that $\gamma \in (-1,0)$ and let $f_0 \in \cP_q(\rd)$ for some $q \ge 8$ with a finite entropy.
Consider $f$ the unique weak solution to \eqref{HL3D} defined in Theorem \ref{thfg}.
For each $N\geq 2$, consider an exchangeable family $(\VV^N_i(0))_{i=1,\dots,N}$ with 
$\sup_{N\geq 2} \E[|\VV^N_1(0)|^4]<\infty$
and the corresponding unique solution  $(\VV^N_i(t))_{i=1,\dots,N, t\geq 0}$ to
\eqref{ps} with some $\eta_N \in (0,N^{-1/3})$. Denote by $\mu^N_t=N^{-1} \sum_1^N \delta_{\VV^N_i(t)}$ 
the associated empirical measure. Then for $\alpha = (1 - 6/q)(2+2\gamma)/3$, for all $T>0$,
\[
\sup_{[0,T]} \E \bigl[W_2^2(\mu^N_t,f_t)\bigr] \leq C_{T,q} \Bigl(N^{-\alpha}+N^{-1/2} 
+ \E \bigl[W_2^2(\mu^N_0,f_0)\bigr] \Bigr).
\]
In particular, the propagation of molecular chaos holds.
\end{thm}

Recall, see \cite{fgui}, that the best general rate we can hope for the expected squared $2$-Wasserstein 
distance between an empirical measure of some i.i.d. $\rd$-valued random variables and their common law is in 
$N^{-1/2}$. Hence for $\gamma \in (-1/4,0)$, and when the particles are initially i.i.d. and $f_0$-distributed, 
the rate of convergence is optimal if $f_0$ has a finite moment of sufficiently high order. 
Of course, it is likely that $N^{-1/2}$ is the true rate for all values of $\gamma$, while our rate deteriorates 
considerably when approaching $\gamma = -1$. 
However, we are quite satisfied, since to our knowledge, there are very few quantitative results
of propagation of chaos for systems with a singular interaction.

\subsection{Trajectories} Our second method will prove a slightly stronger result
(although without rate): the convergence at the level of trajectories.
We thus need to introduce the stochastic paths associated to the Landau equation.
These paths will furthermore be used to prove the strong/weak estimates.

\vip

Given a Brownian motion $(\BB(t))_{t\geq 0}$, independent of an initial condition $\VV(0)$ with law $f_0$, 
we are interested in a continuous adapted $\rr^3$-valued process $(\VV(t))_{t\geq 0}$ solution to
\begin{align}\label{nsde}
\VV(t)=\VV(0) + \intot b(f_s,\VV(s))ds + \intot \sigma(f_s,\VV(s))d\BB(s),
\end{align}
where $f_t \in \cP(\rr^3)$ is the law of $\VV(t)$ and using notation \eqref{bfaf}.

\vip

The process $(\VV(t))_{t\geq 0}$ represents the time evolution of the velocity of a {\it typical}
particle in a plasma whose velocity distribution solves the Landau equation. Such a probabilistic
interpretation (in the case of the Boltzmann equation) was initiated by Tanaka \cite{t}. See
Funaki \cite{f} for the case of the Landau equation.
The following results of existence and uniqueness are
proved in \cite{fg}, but with another formulation involving a white noise.
We will shortly prove them again, since we need to extend them to another nonlinear
SDE that we will introduce later.

\begin{prop}\label{nsdewp}
Let $\gamma \in (-2,0)$ and 
$f_0 \in \cP_2(\rr^3)$ satisfy also $H(f_0)<\infty$ and $m_q(f_0)<\infty$ for some $q>q(\gamma)$, 
recall \eqref{def:pq}. 

\vip
(i) There exists a pathwise unique
continuous adapted solution $(\VV(t))_{t\geq 0}$ to~\eqref{nsde}
such that $f=(f_t)_{t\geq 0} \in L^\infty_{loc}([0,\infty),\cP_2(\rd))\cap 
L^1_{loc}([0,\infty),L^p(\rr^3))$ for some $p\in (p_1(\gamma),p_2(\gamma,q))$.

\vip
(ii) Furthermore, $f$ is the weak solution to the Landau equation~\eqref{HL3D} 
given by Theorem~\ref{thfg}.
\end{prop}

\subsection{A convergence result without rate}

We will assume the following hypothesis on the initial conditions of \eqref{ps} and \eqref{HL3D}.

\begin{equation}\label{condichaos}
\left\{\begin{array}{l} 
\ds \text{(i) } f_0 \in \cP_2(\rr^3) \cap \cP_q(\rr^3) \hbox{ for some } q>q(\gamma)
\text{ and } H(f_0)<\infty; \\[1em]
\ds  \text{(ii) } \text{the sequence } (\VV^N_1(0),\dots,\VV^N_N(0)) 
\text{ with law $F^N_0$ is exchangeable and $f_0$-chaotic;}\\[.5em]
\ds \text{(iii) } \sup_{N\geq 2} \E[|\VV^N_1(0)|^2+|\VV^N_1(0)|^q] <\infty \quad \text{and }  
\sup_{N\geq 2} H(F^N_0)  < \infty.
\end{array}\right.
\end{equation}

All these conditions hold true if $f_0$ satisfies point (i) and if $(\VV^N_1(0),\dots,\VV^N_N(0))$
are i.i.d. and $f_0$-distributed. 

\begin{thm}\label{mr} Assume that $\gamma\in(-2,0)$.
Consider $f_0\in\cP(\rr^3)$ and, for each $N\geq 2$, a family $(\VV_i^N(0))_{i=1,\dots,N}$
of $\rr^3$-valued random variables. Assume \eqref{condichaos}. 
For each $N\geq 2$, consider the unique solution
$(\VV^N_i(t))_{i=1,\dots,N,t\geq 0}$ to \eqref{ps} with some $\eta_N \in (0,1)$. Let also 
$(f_t)_{t\geq 0}$ be the unique weak solution to \eqref{HL3D} given by Theorem \ref{thfg} and $(\VV(t))_{t\ge 0}$
the unique solution to \eqref{nsde} (see Proposition \ref{nsdewp}). Then, as soon as 
$\lim_{N \to \infty} \eta_N =0$,
the sequence $((\VV^N_i(t))_{t\ge 0})_{i = 1 \ldots,N}$ is $(\VV(t))_{t\ge 0}$-chaotic.
In particular, if we set
\begin{equation*}
 \mu^N_t:= \frac1N \sum_{i=1}^N \delta_{\VV_i^N(t)},
\end{equation*}
then $(\mu^N_t)_{t\geq 0}$  goes in probability, in $C([0,\infty),\cP(\rr^3))$, to $(f_t)_{t\geq 0}$.
\end{thm}

\subsection{Comments}

Propagation of chaos was initiated by Kac \cite{Kac1956} as a step to the derivation
of the Boltzmann equation. Since then, many models have been studied. For non singular interactions,
things are more or less well-understood
and there has even been recently some important progress to get uniform in time propagation
of chaos, see Mischler-Mouhot \cite{MMKacProg}, to which we refer for many references 
including the important works of Sznitman \cite{S1,SSF}, M\'el\'eard \cite{m}, see also 
Mischler-Mouhot-Wennberg
\cite{MMW} and \cite{hm}. As already mentioned, most of the results concerning the Landau equation
\cite{fgm,fou1,c} concern the case where $\gamma=0$ (or, at least, where the singularity is removed).

\vip

The case of a singular interaction is much more complicated and there are only very few works.
Osada \cite{o1,o2} has obtained some remarkable results concerning the convergence of the vortex model
to the Navier-Stokes equation (in dimension $2$, with a divergence free interaction in $1/|x|$), improved 
recently by the authors and Mischler~\cite{fhm1}. 
In dimension one, Cepa and Lepingle have also studied the (very singular) Dyson model~\cite{cl}.
We shall also mention the case of a deterministic particle system (with position and velocity) in singular
interaction studied by the second author and Jabin~\cite{HaurayJabin}.
Here, the first quantitative method is in some sense inspired from~\cite{HaurayJabin}. The second one 
(without rate) relies on the entropy dissipation technique introduced in~\cite{fhm1}.

\vip

Finally, let us mention again the work in preparation \cite{bfg}, which treats a similar problem
when $\gamma \in [0,1)$, with more satisfying results. Some of the techniques used here are common 
with \cite{bfg}. In particular, the introductory part of Section \ref{finreg} is reproduced from it.

\subsection{Plan of the paper} 
In the next section, we prove several regularity estimates concerning the coefficients
$a$ and $b$, we check Proposition~\ref{wpps} (well-posedness of the particle system)
and  Proposition~\ref{nsdewp} (well-posedness of the nonlinear SDE).
Section~\ref{finreg} is devoted to the proof of Theorem~\ref{wp} (strong/weak stability estimates).
We prove the (uniform in $N$) propagation of moments for the particle system
in Section~\ref{secmom}. The proof of Theorem~\ref{mrrate} (propagation of chaos with rate when 
$\gamma \in (-1,0)$) is given in Section~\ref{secrate}.
We next study precisely, in Section~\ref{secellip},
the ellipticity of $a(\mu,v)$ when $\mu$ is an empirical measure.
We give the proof of Theorem~\ref{mr} (propagation of chaos without rate when $\gamma \in (-2,0)$)
in Section~\ref{sec:per}.
In a first appendix, we extend a coupling result of \cite{hm}.
In a second appendix, we generalize slightly a result of Figalli \cite{Fig} on the equivalence between 
PDEs and SDEs, that we use several times.
Finally, a third appendix about entropy and weighted Fisher information lies at the end of the paper.

\subsection{Final notation}\label{fin}

We recall that $a(v)=|v|^\gamma(|v|^2 I - v\otimes v)$ and that $b(v)=-2|v|^\gamma v$.
We introduce $\sigma(v)=(a(v))^{1/2}=|v|^{\gamma/2-1}(|v|^2 I - v\otimes v)$.

\vip

For $\eta \in (0,1)$, we recall that $\phi_\eta(x)=(3/(4 \pi \eta^3))\indiq_{\{|x|<\eta\}}$.
We introduce $a_\eta=a\star \phi_\eta$, $b_\eta=b\star \phi_\eta$ and $\sigma_\eta= a_\eta^{1/2}$,
and we define, for $f \in \cP(\rd)$ and $v\in \rd$, 
\[
a_\eta(f,v)=\intrd a_\eta(v-w)f(dw), \quad
b_\eta(f,v)=\intrd b_\eta(v-w)f(dw), \quad \text{and} \quad \sigma_\eta(f,v)=\bigl(a_\eta(f,v) \bigr)^{1/2}.
\]
Remark that it is very similar to the formula~\eqref{bfaf} which corresponds to the case $\eta=0$ with 
the convention that $a_0=a$ and $b_0=b$. 

\vip
For $f \in \cP(\rd)$, we put $f^\eta=f\star \phi_\eta$. We observe that
$a(f^\eta,v)=a_\eta(f,v)$, $b(f^\eta,v)=b_\eta(f,v)$ and $\sigma(f^\eta,v)=\sigma_\eta(f,v)$.

\vip

We will use the standard notation $x \land y= \min\{x,y\}$ and $x \lor y = \max\{x,y\}$ for $x,y \in \R$.

\vip

Finally, for $M,N$ two $3\times 3$ matrices, we set $\lnm M,N \rnm=\Tr MN^*$
and $\|M\|^2=\Tr MM^*$.

%
%
%
%
%
%
%
%

\section{First regularity estimates and well-posedness of the nonlinear SDE}\label{secreg}
\setcounter{equation}{0}

\subsection{Ellipticity}
We recall the ellipticity estimate of Desvillettes-Villani \cite[Proposition 4]{dv}.
It is stated when $\gamma>0$ but the proof only uses that $\gamma+2>0$. A more general result will 
occupy Section \ref{secellip}.

\begin{lem}\label{ellip1}
Let $\gamma \in (-2,0)$.
For all $f\in\cP_2(\rr^3)$ satisfying $H(f) < + \infty$, 
there exists a constant $\kappa_0>0$ depending only on $\gamma$, $H(f)$ and  $m_2(f)$ 
such that for all $v\in\rd$,
$$
\inf_{|\xi|=1}\xi^* a(f,v) \xi \geq \kappa_0(1+|v|)^\gamma.
$$
\end{lem}

\subsection{Rough regularity estimates and well-posedness of the particle system}

We will frequently use the following lemma stated in \cite[Lemma 11]{fou1} (with $C=1$ but with another norm).

\begin{lem}\label{math}
There is a constant $C>0$ such that for any pair of nonnegative symmetric $3\times 3$ matrices $A,B$, 
$$
||A^{1/2}-B^{1/2}||\leq C||A-B||^{1/2} \quad \text{and} \quad 
||A^{1/2}-B^{1/2}||\leq C(||A^{-1}||\land||B^{-1}||)^{1/2}||A-B||.
$$
\end{lem}

We next collect some rough regularity estimates on $a$, $b$ and $\sigma$.

\begin{lem} \label{lem:Lip}
For any $\gamma \in (-2,0)$, there is a constant $C$
such that for all $v,w\in\rr^3$,
\begin{align*}
|b(v) - b(w) | & \le C\, |v-w| \bigl( |v|^\gamma + |w|^\gamma \bigr), \\
\|a(v) - a(w) \| & \le C\, |v-w| \bigl( |v|^{1+\gamma} + |w|^{1+\gamma} \bigr), \\
\|\sigma(v) - \sigma(w) \| & \le C \, |v-w| \bigl( |v|^{\gamma/2} + |w|^{\gamma/2} \bigr).
\end{align*}
\end{lem}

\begin{proof}
The inequality concerning $b$ is proved in \cite[Remark 2.2]{fg}.
Since by definition $\sigma(v) = |v|^{\gamma /2-1} ( |v|^2 I - v \otimes v)$, we see that $\sigma$ is a 
$C^1$ function outside the origin and that $\| D \sigma(v)\| \le C |v|^{\gamma/2}$. 
To go from $v$ to $w$, it is always possible to find a path $S:[0,1] \rightarrow \rr^3$, of length 
smaller than $\pi |v-w|/2$ and that always remains in the crown $\{ z \in\rr^3 \; : \; \min\{|v|,|w|\}
\le |z| \le \max\{|v|,|w|\} \}$. For instance, some arc of a circle will do the job. Then, the claimed 
inequality follows from
\[
\|\sigma(v) - \sigma(w) \| \le \int_0^1 \|D\sigma(S(t))\| |S'(t)| dt 
\le C\,  |v-w|  \max\{|v|^{\gamma/2},|w|^{\gamma/2}\}.
\]
The inequality concerning $a$ is proved similarly, using that $\| D a(v)\| \le C |v|^{1+\gamma}$.
This proof may also be adapted to $b$, using that $\| D b(v)\| \le C |v|^{\gamma}$.
\end{proof}

We next study $a_\eta$, $b_\eta$ and $\sigma_\eta$.

\begin{lem}\label{rough} Let $\gamma \in (-2,0)$.

\vip

(i) 
For each $\eta \in (0,1)$, there is $C_{\eta}$
such that for  any $y,z \in \rd$, $||(a_\eta(y))^{-1}|| \leq C_\eta(1+|y|)^{|\gamma|}$ and
\begin{eqnarray*}
&|b_\eta(y)-b_\eta(z)| \leq C_\eta |y-z|, \\
&||\sigma_\eta(y) - \sigma_\eta(z)|| +||a_\eta(y) - a_\eta(z)|| 
\leq C_\eta (1+|y|+|z|)|y-z|.
\end{eqnarray*}

(ii) There is $C>0$ such that for all $\eta \in (0,1)$, all $v\in \rd$,
$$
||a_\eta (v) - a(v)||\leq C \eta^2 (\eta+|v|)^{\gamma}
\quad\hbox{and}\quad
|b_\eta (v) - b(v)|\leq 
\left\{\begin{array}{ll}C \eta (\eta+|v|)^{\gamma} 
& \hbox{if $\gamma \in (-1,0)$,}\\
C \min\{\eta,|v|\} |v|^\gamma & \hbox{if $\gamma \in (-2,-1]$.}
\end{array}\right.
$$

(iii) There is $C>0$ such that for all $\eta \in (0,1)$, all $f \in \cP(\rd)$, all $v\in\rd$,
$$
||a_\eta(f,v)-a(f,v)||\leq C \eta^{2+\gamma} \quad \hbox{and, if $\gamma\in(-1,0)$,}\quad
|b_\eta(f,v)-b(f,v)|\leq C \eta^{1+\gamma}.
$$
\end{lem}

\begin{proof}
We first check point (i). First, it holds that $a_\eta(y)=a(\phi_\eta,y)$
when using notation \eqref{bfaf}. But $\phi_\eta$
obviously belongs to $\cP_2(\rd)$ and has a finite entropy. Hence Lemma \ref{ellip1} tells us that
for all $\xi \in \rd$, all $y\in\rd$, $\xi^* a_\eta(y)\xi \geq c_\eta (1+|y|)^\gamma |\xi|^2$,
which implies that $||(a_\eta(y))^{-1}|| \leq C_\eta(1+|y|)^{|\gamma|}$. 

\vip

We next use Lemma \ref{lem:Lip}
to get 
$$
|b_\eta(y)-b_\eta(z)| \leq C \eta^{-3} \int_{|u|<\eta}|y-z|(|y+u|^\gamma+|z+u|^\gamma) \,du
\leq C_\eta|y-z|,
$$
as well as
$$
||a_\eta(y)-a_\eta(z)|| \leq C \eta^{-3} \int_{|u|<\eta}|y-z|(|y+u|^{\gamma+1}+|z+u|^{\gamma+1}) \,du,
$$
which is easily bounded by $C_\eta |y-z|$ if $\gamma\in (-2,-1)$ and by 
$C_\eta |y-z|(1+|y|^{\gamma+1}+|z|^{\gamma+1})$ if $\gamma\in [-1,0)$.
Using finally point Lemma \ref{math} and the estimate on $||(a_\eta(y))^{-1}||$, we conclude that
$||(a_\eta(y))^{1/2}-(a_\eta(z))^{1/2}|| \leq C_\eta h(|y|,|z|)|y-z|$,
where $h(|y|,|z|)= \min\{(1+|y|)^{|\gamma|/2},(1+|z|)^{|\gamma|/2}\}$ if $\gamma\in (-2,-1)$ and 
$h(|y|,|z|)= \min\{(1+|y|)^{|\gamma|/2},(1+|z|)^{|\gamma|/2}\}(1+|y|+|z|)^{1+\gamma}$ if $\gamma\in [-1,0)$.
In any case, $h(|y|,|z|) \leq C (1+|y|+|z|)$, which ends the proof of point (i).

\vip

Concerning point (ii), we first study $b$ and separate two cases.
If $|v|\leq 2\eta$, then we have
$|b(v)|=2|v|^{1+\gamma}$ and $|b_\eta(v)|\leq C |v|^{1+\gamma} + C \eta^{1+\gamma}$.
Hence $|b(v)-b_\eta(v)|\leq C |v|^{1+\gamma} +C \eta^{1+\gamma}$,
which is smaller than  $C \eta^{1+\gamma} \leq C \eta(\eta+|v|)^\gamma$ if 
$\gamma \in (-1,0]$ and than $C |v|^{1+\gamma}\leq \min\{\eta,|v|\} |v|^\gamma$
if $\gamma \in (-2,-1]$.
If now $|v|>2\eta$,
we use that $|b(y)-b(z)| \leq C |y-z|(|y|^\gamma+|z|^\gamma)$ by Lemma \ref{lem:Lip}, so that 
$|b(v)-b_\eta(v)|\leq C \eta \intrd (|v|^\gamma+|v-u|^\gamma) \phi_\eta(u)du \leq C \eta |v|^\gamma$, which is bounded
by $C\eta(\eta+|v|)^{\gamma}$ if 
$\gamma \in (-1,0]$ and by $C \min\{\eta,|v|\} |v|^\gamma$
if $\gamma \in (-2,-1]$.

\vip

We now study $||a(v)-a_\eta(v)||$. If $|v|\leq 2\eta$, we immediately get that
$||a(v)-a_\eta(v)||\leq ||a(v)||+||a_\eta(v)||\leq C \eta^{2+\gamma} \leq C \eta^2(\eta+|v|)^\gamma$.
If now $|v| \geq 2\eta$, we use a Taylor expansion: write, for $|u|<\eta$, that
$a(v-u)=a(v)-u\nabla a(v) + \zeta(u,v)$, where $||\zeta(u,v)|| \leq |u|^2 \sup_{B(v,\eta)}||D^2 a||
\leq C \eta^2 |v|^\gamma$.
Consequently, since $\phi_\eta$ is symmetric,
$$
||a(v)-a_\eta(v)||=\Big|\Big| \int_{|u|\leq \eta} (a(v)-a(v-u)) \phi_\eta(u)du \Big|\Big|= 
\Big|\Big| \int_{|u|\leq \eta}  \zeta(u,v) \phi_\eta(u)du \Big|\Big| \leq C \eta^2 |v|^\gamma,
$$
which is controlled by $C \eta^2(\eta+|v|)^\gamma$ as desired.

\vip

We finally have to prove (iii). If $\gamma \in (-1,0)$, we have  
$|b_\eta(x)-b(x)|\leq C \eta^{1+\gamma}$ by (ii), whence 
$$
|b_\eta(f,v)-b(f,v)|=\Big|\intrd [b_\eta(v-w)-b(v-w)]f(dw) \Big| \leq C\eta^{1+\gamma}.
$$
Now for any value of $\gamma \in (-2,0)$, we see from point (ii) that $||a_\eta(x)-a(x)||\leq C \eta^{2+\gamma}$,
so that we obtain $||a_\eta(f,v)-a(f,v)|| \leq C\eta^{2+\gamma}$ by a simple integration.
\end{proof}

At this point, we can prove the strong well-posedness of the particle system.

\begin{proof}[Proof of Proposition \ref{wpps}.]
Let 
\begin{align*} 
b_i(v_1,\dots,v_N) & :=
b \Bigl(\frac1N \sum_{j=1}^N \delta_{v_j} \star \phi_{\eta_N}, v_i \Bigr) 
= \frac1N \sum_{j=1}^N b_{\eta_N}(v_i - v_j),  \\
\sigma_i(v_1,\dots,v_N) & := \biggl[a \Bigl( \frac1N \sum_{j=1}^N \delta_{v_j} 
\star \phi_{\eta_N},v_i \Bigr) \biggr]^{\frac12} 
= \biggl[ \frac1N \sum_{j=1}^N a_{\eta_N} (v_i - v_j) \biggr]^{\frac12}.
\end{align*}
These are the coefficients of the system of SDEs \eqref{ps}. We claim that these coefficients have 
at most linear growth and are locally Lipschitz continuous, from which strong existence and uniqueness
classically follow. First, $b_i$ is globally Lipschitz continuous by Lemma~\ref{rough}-(i).
Next, it follows from Lemma \ref{rough}-(i) that $\sigma_i$ is locally Lipschitz continuous.
Finally, recalling that $||A||=\Tr A A^*$ and that $\Tr(a(v))=2|v|^{\gamma+2}$, we find that 
\[ \|\sigma_i(v_1,\dots,v_N)\|^2 =
\frac2N \sum_{j=1}^N \intrd |v_i-v_j - u|^{\gamma+2} \phi_{\eta_N}(u)du
\le \frac2 N \sum_{j=1}^N (|v_i|+|v_j|+\eta_N)^{\gamma+2}.
\] 
Since $\gamma+2 \in (0,2]$, we conclude that $\sigma_i$ has at most linear growth.
\end{proof}

\subsection{Regularity estimates for the nonlinear SDE}

We now give some growth and regularity estimates on the fields $a(f)$, $\sigma(f)$ and $b(f)$ created 
by some probability $f$, that will allow us to 
study the well-posedness of the linear version of the SDE~\eqref{nsde}, as well as that of
a perturbed version of it.

\begin{lem}\label{regest}
Let $\gamma \in (-2,0)$ and $f \in \cP_2(\rr^3)$. We have the following estimates
for all $v,w \in \rd$.

\vip

(i) In any case, $\|\sigma(f,v)\|^2 \leq 4+ 4 m_2(f)+ 4|v|^{\gamma+2}$.

\vip

(ii) If $\gamma \in [-1,0)$, then $|b(f,v)| \leq 2+2|v| + 2m_1(f)$.

\vip

(iii) If $\gamma \in (-2,-1)$ and $p>3/(4+\gamma)$, then $|b(f,v)| \leq 2 +C_p ||f||_{L^p}$.

\vip

(iv) If $p>p_1(\gamma)$, then $|b(f,v)-b(f,w)|\le C_{p}(1+\|f\|_{L^p})|v-w|$.

\vip

(v) If $p>p_1(\gamma)$, then
$\|\sigma(f,v)-\sigma(f,w)\|^2  \leq  C_{p} (1+\|f\|_{L^p})|v-w|^2$.

\vip

(vi)  If $\gamma \in [-1,0)$, then
$\| a(f,v)  - a(f,w) \| \le C (1+ |v| + |w| + m_1(f))|v-w|$.

\vip

(vii) If $\gamma \in (-2,-1)$ and $p>3/(4+\gamma)$, then 
$\| a(f,v)  - a(f,w) \| \le C_{p} (1+ \|f\|_{L^p})|v-w|$.
\end{lem}

\begin{proof}
Recalling that $\Tr a(v)=2|v|^{\gamma+2}$,
\[
\|\sigma(f,v)\|^2=\Tr a(f,v)= 2\intrd |v-w|^{\gamma+2}f(dw).
\] 
Point (i) follows, since $|v-w|^{\gamma+2}\leq 2|v|^{\gamma+2}+2|w|^{\gamma+2}\leq 2 + 2|w|^2+2|v|^{\gamma+2}$.

\vip

Next, $|b(f,v)|\leq 2 \intrd |v-w|^{\gamma+1}\, f(dw)$. Point (ii) follows,
since for $\gamma \in [-1,0)$, $|v-w|^{\gamma+1} \leq 1+|v|+|w|$. Point (iii) also follows, using 
\eqref{weakb2} and that $p>3/(4+\gamma)$.

\vip

Point (iv) is obtained by integration of the 
Lipschitz estimate on $b$ given in Lemma~\ref{lem:Lip} and by using \eqref{weakb2} (recall that $p>p_1(\gamma)=
3/(3+\gamma)$):
\[ |b(f,v)-b(f,w)|  \le C \, |v-w| \intrd (|v-x|^\gamma + |w-x|^\gamma) f(dx)
\le C \, |v-w| \bigl( 1 + C_{p} \| f\|_{L^p} \bigr).
\]

For (v), we use a classical result, see e.g. Stroock-Varadhan~\cite[Theorem 5.2.3]{SV}, which
asserts that there is $C>0$ such that for all $A: \rr^3 \mapsto S_3^+$ (the set
of symmetric nonnegative $3\times 3$-matrices with real entries), 
$\|D(A^{1/2})\|_\infty \leq C \|D^2 A\|_{\infty}^{1/2}$.
Here we apply it to $A(v) = a(f,v)$. 
First,
\[
\|D^2 A(v) \| = \biggl\| \intrd D^2 a(v-x)\,f(dx) \biggr\| \le C \intrd  |v-x|^\gamma \,f(dx)
\le  C \bigl( 1 + C_{p} \| f\|_{L^p} \bigr),
\]
where we have used that $\|D^2 a(v)\| \le |v|^\gamma$ and~\eqref{weakb2} (since $p>3/(3+\gamma)$). 
Consequently,
\[
\|\sigma(f,v)-\sigma(f,w)\|^2  = \bigl\| (A(v))^{1/2} - (A(w))^{1/2} \bigr\|^2
\le \|D(A^{1/2})\|_\infty^2 |v-w|^2 \leq  C \bigl( 1 + C_{p} \| f\|_{L^p} \bigr)|v-w|^2 .
\]

Finally, integrating the Lipschitz estimate on $a$ given in Lemma~\ref{lem:Lip}, we find the inequality
$\| a(f,v) - a(f,w) \| \le C \, |v-w| \intrd\bigl(|v-x|^{1+\gamma} + |w-x|^{1+\gamma} \bigr)   \,f(dx)$.
Using the same arguments as in (ii) and (iii), points (vi) and (vii) immediately follow.
\end{proof}

\subsection{Well-posedness of the nonlinear SDE}
We now have all the ingredients to give the

\begin{proof}[Proof of Theorem \ref{nsdewp}.]
Recall that $f_0 \in \cP_2(\rd)$ is assumed to satisfy $m_q(f_0)<\infty$
for some $q>q(\gamma)$ and $H(f_0)<\infty$. We denote by 
$f \in L^\infty_{loc}([0,\infty),\cP_2(\rr^3))\cap L^1_{loc}([0,\infty),L^p(\rr^3))$ (for all $p\in (p_1(\gamma),
p_2(\gamma,q))$)
the unique weak solution to~\eqref{HL3D}, see Theorem \ref{thfg}. 
Let also $\VV(0)$ be $f_0$-distributed and consider a $3D$ Brownian motion $\BB$ independent of $\VV(0)$.

\vip

{\it Step 1: the linear SDE.} Recalling that $f$ is given, we prove here that there is strong existence and
uniqueness for the linear SDE 
\begin{equation} \label{linSDE}
\VV(t)=\VV(0) + \intot b(f_s,\VV(s))ds + \intot \sigma(f_s, \VV(s)) d\BB(s).
\end{equation}
Fix $p\in (p_1(\gamma),p_2(\gamma,q))$ and recall that $f \in L^\infty_{loc}([0,\infty),\cP_2(\rr^3))\cap
L^1_{loc}([0,\infty),L^p(\rr^3))$.
By Lemma \ref{regest}-(iv)-(v),
the coefficients of this SDE are Lipschitz continuous in $v$, with a Lipschitz constant (locally) integrable in 
time for $b$ and (locally) square integrable in time for $\sigma$. 
The conclusion follows from classical arguments, 
see for instance~\cite[Theorem 3.17]{Pardoux}.

\vip

{\it Step 2: uniqueness for the linear PDE.} Recalling that $f$ is fixed, we consider the Kolmogorov
equation associated to \eqref{linSDE}:
\begin{equation}\label{pdelin}
\partial_t g_t(v) = \frac12  \diver_v \bigl(  a(f_t,v) \nabla g_t(v) - b(f_t,v) g_t(v)  \bigr).
\end{equation}
As for the Landau equation, $g \in L^\infty_{loc}([0,\infty),\cP_2(\rd))$ is a weak solution if
for all $\varphi \in C^2_b(\rd)$,
$$
\intrd \varphi(v)g_t(dv)= \intrd \varphi(v)g_0(dv)+ \intot \intrd\intrd L\varphi(v,v_*) g_s(dv) f_s(dv_*)ds.
$$
By Lemma \ref{regest}-(i)-(ii)-(iii), we know that $||\sigma(f_t,v)||^2 \leq C(1+|v|^2)$ and $|b(f_t,v)| 
\leq C (1+||f_t||_{L^p})(1+|v|)$, with $||f_t||_{L^p} \in L^1_{loc}([0,\infty))$.
Hence Proposition \ref{pfi}, which slightly extends \cite[Theorem 2.6]{Fig}, 
tells us that any weak solution $(g_t)_{t\ge 0}$ to \eqref{linSDE}
can be represented as the family of time marginals of a solution to \eqref{linSDE}.
Consequently, Step 1 implies the uniqueness of the solution to \eqref{linSDE}, for any 
given $g_0 \in \cP_2(\rd)$.

\vip

{\it Step 3: strong existence.} From Step 1, we know that there exists a solution $\VV$ to \eqref{linSDE}.
Put $g_t=\LL(\VV(t))$. A direct application of the It\^o formula shows that $g$ is a weak solution to
\eqref{pdelin}, with $g_0=f_0$. But $f$, being a weak solution to the Landau equation \eqref{HL3D} (see also
\eqref{alt}), is also a weak solution to \eqref{pdelin}. By Step 2, we deduce that $g=f$ as desired.

\vip

{\it Step 4: strong uniqueness.} Consider another solution  
$\WW$ (with $\LL(\WW(t))=h_t$) to the nonlinear SDE \eqref{nsde} (with the same initial condition and the
same Brownian motion). Assume also that $h \in L^\infty_{loc}([0,\infty),\cP_2(\rr^3))\cap 
L^1_{loc}([0,\infty),L^p(\rr^3))$ for some $p\in (p_1(\gamma),p_2(\gamma,q))$).
A direct application of the It\^o formula shows that $h$ is a weak
solution to \eqref{HL3D}, whence $f=h$ by Theorem \ref{thfg}. Consequently, $\WW$ also solves
the linear SDE \eqref{linSDE}, so that $\VV=\WW$ by Step 2.
\end{proof}

\section{Fine regularity estimates and strong/weak stability principles}\label{finreg}
\setcounter{equation}{0}

The goal of this section is to prove Theorem \ref{wp}, that is the 
strong/weak stability principles.
These principles will be proved using a coupling argument between two solutions of the nonlinear SDE \eqref{nsde}.
Unfortunately, we cannot use the same Brownian motion for both solutions: this is due to the fact 
that we really need a fine estimate and that the best coupling between two $3D$ Gaussian distribution
$\NN(0,\Sigma_1)$ and $\NN(0,\Sigma_2)$ {\it does not} consist in setting $X_1=\Sigma_1^{1/2}Y$ and
$X_2=\Sigma_2^{1/2}Y$ for the same $Y$ with law $\NN(0,I)$. Actually, as shown in Givens-Shortt
\cite{gs}, the optimal  coupling
is obtained when setting $X_1=\Sigma_1^{1/2}Y$ and $X_2=\Sigma_2^{1/2}U(\Sigma_1,\Sigma_2)Y$, where the
orthogonal matrix $U(\Sigma_1,\Sigma_2)$ is given by
\begin{equation}\label{dfU}
U(\Sigma_1,\Sigma_2)=\Sigma_2^{-1/2}\Sigma_1^{-1/2}(\Sigma_1^{1/2}\Sigma_2\Sigma_1^{1/2} )^{1/2}.
\end{equation}
More precisely, we will use the following lemma.

\begin{lem} \label{tictactoc}
For any probability measure $m$ on a measurable space $F$, 
any pair of measurable families of $3\times 3$ matrices $(\sigma_1(x))_{x\in F}$ and $(\sigma_2(x))_{x\in F}$,
setting $\Sigma_1=\int_F \sigma_1(x)\sigma_1^*(x) m(dx)$ and $\Sigma_2=\int_F \sigma_2(x)\sigma_2^*(x) m(dx)$,
it holds that 
\[
\bigl\|\Sigma_1^{1/2}-\Sigma_2^{1/2} U(\Sigma_1,\Sigma_2) \bigr\|^2 \leq \int_F \|\sigma_1(x)-\sigma_2(x)\|^2 m(dx).
\]
\end{lem}

\begin{proof}
Developing both terms and using the definition \eqref{dfU} of $U$, we realize that the issue is to prove that 
$$
\Tr \Big(\int_F \sigma_1(x)\sigma_2^*(x) m(dx) \Big) \leq \Tr \Big((\Sigma_1^{1/2}\Sigma_2\Sigma_1^{1/2} )^{1/2}\Big).
$$
Givens and Shortt \cite[Proof of Proposition 7]{gs} have checked that 
$\Tr (M) \leq \Tr ((\Sigma_1^{1/2}\Sigma_2\Sigma_1^{1/2} )^{1/2})$ for all $3\times 3$ matrix $M$ such that
$$
\begin{pmatrix} \Sigma_1 & M \\ M^* & \Sigma_2 \end{pmatrix}
$$
is nonnegative. But with $M=\int_F \sigma_1(x)\sigma_2^*(x) m(dx)$,
$$
\begin{pmatrix} \Sigma_1 & M \\ M^* & \Sigma_2 \end{pmatrix}
= \int_F \begin{pmatrix} \sigma_1(x) \\ \sigma_2(x) \end{pmatrix}
\begin{pmatrix} \sigma_1(x) \\ \sigma_2(x) \end{pmatrix}^* m(dx)
$$
is clearly nonnegative. The conclusion follows.
\end{proof}

\subsection{Main ideas of the proof of Theorem \ref{wp}}
Unfortunately, the rigorous proof is very technical. Let us explain here
the main ideas without justification. Let thus $f$ and $g$ be two weak solutions to \eqref{HL3D}.
We associate (with some work, using some regularization) 
to these solutions $(\VV(t))_{t\geq 0}$ (with $\LL(\VV(t))=f_t$) and $(\ZZ(t))_{t\geq 0}$
(with $\LL(\ZZ(t))=g_t$) solving
$$
d\VV(t)=a(f_t,\VV(t))U(t)d\BB(t) + b(f_t,\VV(t)) dt\quad\hbox{and}\quad 
d\ZZ(t)=a(g_t,\ZZ(t))d\BB(t) + b(g_t,\ZZ(t)) dt
$$
where $U(t)=U(a(g_t,\ZZ(t)),a(f_t,\VV(t))$.
Recalling the $U$ takes values in the set of orthogonal matrices, we realize that
$U(t)d\BB(t)$ is a Brownian motion.
Hence, $\VV$ and $\ZZ$ are coupled solutions, with different initial values, to the nonlinear SDE \eqref{nsde}.
We will of course bound $W_2^2(f_t,g_t)$ by $\E[|\VV(t)-\ZZ(t)|^2]$. Using the It\^o formula, we directly find that
\begin{align}\label{sss}
\frac{d}{dt}\E[|\ZZ(t)-\VV(t)|^2] = \Gamma (R_t)
\end{align}
where $R_t=\LL(\ZZ(t),\VV(t))$ and where, for $R$ a probability on $\rd\times\rd$, with marginals $\mu$
and $\nu$,
$$
\Gamma(R)= \int_{\rd\times\rd} \Big(2(z-v).(b(\mu,z)-b(\nu,v)) + \big\|\sigma(\mu,z)-\sigma(\nu,v)
U(a(\mu,z),a(\nu,v)) \big\|^2 \Big) R(dz,dv).
$$
Using Lemma \ref{tictactoc} and that $R$ has marginals $\mu$ and $\nu$, we easily bound
$$
\Gamma(R) \leq \int_{\rd\times\rd} \int_{\rd\times\rd} \Delta(z,z_*,v,v_*) R(dz_*,dv_*) R(dz,dv)
$$
where $\Delta(z,z_*,v,v_*):= 2(z-v)\cdot (b(z-z_*)-b(v-v_*)) + \|\sigma(z-z_*)-\sigma(v-v_*)\|^2$.
A simple computation, see Lemma \ref{fondam} below, shows that $\Delta=\Delta_1+\Delta_2$, where
$\Delta_1$ is antisymmetric and disappears when integrating, and where $\Delta_2$ can be controlled explicitly.
We end with 
$$
\Gamma(R) \leq 8 \int_{\rd\times\rd} \int_{\rd\times\rd} |z-v|^2 \frac{(|z-z_*|\land |v-v_*|)^{1+\gamma}}
{|z-z_*| \lor |v-v_*|}  R(dz_*,dv_*) R(dz,dv),
$$
Assume first that $\gamma \in (-1,0)$. We get
\begin{align*}
\Gamma(R) \leq& 8 \int_{\rd\times\rd} \int_{\rd\times\rd} |z-v|^2 |z-z_*|^{\gamma}
R(dz_*,dv_*) R(dz,dv) \\
=&\int_{\rd\times\rd} \intrd |z-v|^2 |z-z_*|^{\gamma} \mu(dz_*) R(dz,dv).
\end{align*}
Using finally \eqref{weakb2} with some $p>p_1(\gamma)$,
$$
\Gamma(R) \leq  C_p (1+||\mu||_{L^p}) \int_{\rd\times\rd} |z-v|^2 R(dz,dv).
$$
Coming back to \eqref{sss}, we end with 
$(d/dt)\E[|\ZZ(t)-\VV(t)|^2] \leq C_p (1+||f_t||_{L^p})\E[|\ZZ(t)-\VV(t)|^2]$,
whence, using the Gr\"onwall lemma,
$$
W^2_2(f_t,g_t) \leq W_2^2(f_0,g_0) \exp\Big(C_p \intot (1+||f_s||_{L^p})ds\Big)
$$
as desired. When now $\gamma \in (-2,-1]$, a slightly more complicated study shows that
$$
\Gamma(R) \leq  C_{p,r} (1+||\mu||_{L^p}+||\nu||_{L^r}) \int_{\rd\times\rd} |z-v|^2 R(dz,dv),
$$
for any $p>p_1(\gamma)$ and any $r>(2p-3)/[(3+\gamma)(p-1)-1]$.

\subsection{Fine regularity estimates}

The subsection is devoted to the proof of the following estimate.

\begin{prop}\label{crucru}
For $f,g \in \cP_2(\rd)$, $R\in \Pi(g,f)$ and $0\leq \eta \leq \e<1$, we set
\begin{align*}
\Gamma_{\eta,\e}(R)=& \int_{\rd\times\rd} \Big(2(z-v)\cdot(b(g^\eta,z)-b(f,v)) \\
&\hskip3cm+ \big\|\sigma(g^\eta,z)-
\sigma(f,v)U(a(g^\e,z),a(f^\e,v))
\big\|^2\Big) R(dz,dv).
\end{align*}
Assume that $f \in L^p(\rd)$ for some $p>p_1(\gamma)$.

\vip

(i) If $\gamma \in (-1,0)$, then 
$$
\Gamma_{\eta,\e}(R) \leq C \e^{2+2\gamma} + C_p(1+||f^\e||_{L^p})\int_{\rd\times\rd} |z-v|^2 R(dz,dv).
$$

(ii) If $\gamma \in (-2,-1]$, we assume also that $g\in L^r$ for some $r>(2p-3)/[(3+\gamma)(p-1)-1]$.
There is $\delta>0$ (depending only on $r,p,\gamma$) such that
$$
\Gamma_{\eta,\e}(R) \leq C_{r,p}(1+||f+g||_{L^r})(1+m_1(f+g)) \e^{\delta} 
+ C_{r,p}(1+||f^\e||_{L^p}+||g^\e||_{L^r})\int_{\rd\times\rd}\!\!\!\!\!\! |z-v|^2 R(dz,dv).
$$
\end{prop}

We start with a simple computation

\begin{lem}\label{fondam}
Assume that $\gamma\in(-2,0)$. Then for any $v,v_*,w,w_*$ in $\rd$,
\begin{align*}
\Delta(v,v_*,w,w_*):=& 2(v-w)\cdot (b(v-v_*)-b(w-w_*)) + \|\sigma(v-v_*)-\sigma(w-w_*)\|^2\\
=&\Delta_1(v,v_*,w,w_*)+\Delta_2(v,v_*,w,w_*),
\end{align*}
where 
$$
\Delta_1(v,v_*,w,w_*)=((v-w)+(v_*-w_*))\cdot(b(v-v_*)-b(w-w_*))
$$
is antisymmetric (i.e. $\Delta_1(v,v_*,w,w_*)=-\Delta_1(v_*,v,w_*,w)$) and
$$
\Delta_2(v,v_*,w,w_*)\leq 4 (|v-w|^2+|v_*-w_*|^2) 
K(|v-v_*|,|w-w_*|),
$$
where
\begin{equation} \label{defK}
K(x,y)= \frac{(x \land y)^{1+\gamma}} {x \lor y}.
\end{equation}
\end{lem}

\begin{proof}
Defining $\Delta_2=\Delta-\Delta_1$, we find, with the notation $X=v-v_*$ and $Y=w-w_*$,
$$
\Delta_2=(X-Y)\cdot(b(X)-b(Y))+ \|\sigma(X)-\sigma(Y)\|^2.
$$
Recalling that $b(X)=-2X|X|^\gamma$, using that $||\sigma(X)||^2=2|X|^{2+\gamma}$ and that
\[
\lnm\sigma(X),\sigma(Y)\rnm = |X|^{1+ \gamma /2}|Y|^{1+\gamma/ 2} 
\biggl( 1 +  \frac{(X \cdot Y)^2}{|X|^2 |Y|^2}\biggr) \ge 2\, 
|X|^{\gamma/ 2}|Y|^{\gamma /2} ( X \cdot Y)
\] 
(we used that $1+x^2\geq 2x$ for the last inequality), one easily checks that
\[
\Delta_2\le2 (X\cdot Y) (|X|^{\gamma/2}-|Y|^{\gamma/2})^2 \leq 2 |X\|Y| (|X|^{\gamma/2}-|Y|^{\gamma/2})^2
=2 (|X\|Y|)^{1+\gamma} \bigl(|X|^{|\gamma|/2}- |Y|^{|\gamma|/2}\bigr)^2.
\]
Next remark that for $\alpha \in [0,1]$ and $x,y \ge 0$,
$$
|x^\alpha-y^\alpha|= (x\lor y)^{\alpha} \Big(1- \Big[\frac{x\land y}{x\lor y}\Big]^\alpha\Big) \leq 
(x\lor y)^{\alpha}  \Big(1- \frac{x\land y}{x\lor y} \Big) =|x-y| (x\lor y)^{\alpha-1},
$$
whence
$$
\Delta_2\leq 2 |X-Y|^2 (|X\|Y|)^{1+\gamma} (|X|\lor |Y|)^{|\gamma|-2} =
2 |X-Y|^2 (|X|\land |Y|)^{1+\gamma} (|X|\lor |Y|)^{-1}.
$$
One easily concludes, using that $|X-Y|^2 \leq 2 (|v-w|^2+|v_*-w_*|^2)$.
\end{proof}

We next prove an intermediate result.

\begin{lem}\label{cru} Let $\gamma\in(-2,0)$.  Consider $f,g \in \cP_2(\rd)$ and $R \in \Pi(g,f)$. 
For any $\e\in(0,1)$,
\begin{align*}
\int_{\rd\times\rd} \Big(2(z-v) & \cdot(b(g^\eps,z)-b(f^\eps,v))
+\|\sigma(g^\eps,z)-\sigma(f^\eps,v)U(a(g^\eps,z),a(f^\eps,v))\|^2  \Big) R(dz,dv)\\
& \leq 8 \int_{\rd\times\rd} \int_{\rd\times\rd} |z-v|^2 K(|z-z_*|,|v-v_*|) R^\eps(dz_*,dv_*)R(dz,dv),
\end{align*}
where $K$ was introduced in~\eqref{defK} and where
$R^\e= \int_{\rd\times\rd}\intrd \delta_{(u+z,u+v)} 
\phi_\e(u) \,du \, R(dz,dv)$ belongs to $\Pi(g^\e,f^\e)$.
\end{lem}

In other words, $R^\e(A)=\int_{\rd\times\rd}\intrd \indiq_A(u+z,u+v) 
\phi_\e(u)R(dz,dv) \,du $ for all $A \in \cB(\rd\times\rd)$.

\begin{proof} Let us denote by $I$ the left hand side of the desired inequality.
We claim that
\begin{eqnarray*}
&(z-v)\cdot (b(g^\e,z)-b(f^\e,v))= \ds\int_{\rd\times\rd} (z-v)\cdot (b(z-z_*)-b(v-v_*)) R^\e(dz_*,dv_*),\\
&\|\sigma(g^\e,z)-\sigma(f^\e,v)U(a(g^\e,z),a(f^\e,v))\|^2 \leq  \ds\int_{\rd\times\rd} 
\|\sigma(z-z_*)-\sigma(v-v_*)\|^2 R^\e(dz_*,dv_*).
\end{eqnarray*}
The first point is obvious: it follows from the fact that the marginals of $R^\e$ are $g^\e$ and $f^\e$.
The second point relies on Lemma \ref{tictactoc}, with (here $z,v \in \rd$ are fixed) $F=\rd\times\rd$, 
$m=R^\e(dz_*,dv_*)$, $\sigma_1(z_*,v_*)=\sigma(z-z_*)$ and
$\sigma_2(z_*,v_*)=\sigma(v-v_*)$. Indeed, since $\sigma(x)=(a(x))^{1/2}$, we have
\begin{align*}
&\int_{\rd\times\rd} \sigma(z-z_*) \sigma^*(z-z_*) R^\e(dz_*,dv_*)=a(g^\e,z),\\
&\int_{\rd\times\rd} \sigma(v-v_*) \sigma^*(v-v_*) R^\e(dz_*,dv_*)=a(f^\e,v).
\end{align*}

Recalling the notation of Lemma \ref{fondam}, we thus deduce that 
\begin{align*}
I \leq & \int_{\rd\times\rd}  \int_{\rd\times\rd}  \Delta(z,z_*,v,v_*) R^\e(dz_*,dv_*) R(dz,dv).
\end{align*}
By Lemma \ref{fondam}, we know that $\Delta = \Delta_1+\Delta_2$, whence, with obvious notation,
$I \leq I_1+I_2$.
Using the substitution $u \mapsto -u$ and that 
$\Delta_1(z,z_*- u,v,v_*- u)=-\Delta_1(z_*,z +u,v_*, v +u)$,
\begin{align*}
I_1 =&  \int_{\rd\times\rd}  \int_{\rd\times\rd}  \intrd \Delta_1(z,z_*+ u,v,v_*+ u) \phi_\e(u)du\;
R(dz_*,dv_*) R(dz,dv) \\
=& \int_{\rd\times\rd}  \int_{\rd\times\rd}  \intrd \Delta_1(z,z_*- u,v,v_*- u)  \phi_\e(u)du\;
R(dz_*,dv_*) R(dz,dv) \\
=& -\int_{\rd\times\rd}  \int_{\rd\times\rd}  \intrd \Delta_1(z_*,z + u,v_*, v+ u) \phi_\e(u)du\; 
R(dz_*,dv_*) R(dz,dv) = - I_1,
\end{align*}
so that $I_1 =0$. 
This computation is licit: recalling the expression of $\Delta_1$ and that $|b(x)|\leq 2|x|^{\gamma+1}$,
\begin{align*}
\kappa_\e(z,z_*,v,v_*):=&\intrd |\Delta_1(z,z_*+ u,v,v_*+ u)| \phi_\e(u)du\\
\leq& 2(|z|+|v|+|z_*|+|v_*|)\intrd (|z-z_*-u|^{1+\gamma}+ |v-v_*-u|^{1+\gamma})\phi_\e(u)du\\
\leq& C_\e(1+|z|+|v|+|z_*|+|v_*|)^2.
\end{align*}
For the last inequality, separate the cases $\gamma \in [-1,0)$ and $\gamma \in (-2,-1)$. 
This last expression is integrable against $R(dz_*,dv_*) R(dz,dv)$
because $f$ and $g$ belong to $\cP_2(\rd)$ by assumption.

\vip

Finally, using again the substitution $u\mapsto -u$, it is easily checked that
\begin{align*}
I_2 = & 4  \int_{\rd\times\rd}  \int_{\rd\times\rd} \intrd (|z-v|^2+|z_*-v_*|^2)K(|z-z_*-u|,|v-v_*-u|)\\
&\hskip8cm\phi_\e(u)du\; R(dz_*,dv_*) R(dz,dv) \\
=& 8 \int_{\rd\times\rd}  \int_{\rd\times\rd} \intrd  |z-v|^2K(|z-z_*-u|,|v-v_*-u|)\phi_\e(u)du\; 
R(dz_*,dv_*) R(dz,dv),
\end{align*}
which is nothing but the RHS of the target inequality.
\end{proof}

We can finally give the

\begin{proof}[Proof of Proposition \ref{crucru}]
Let thus $\gamma\in(-2,0)$, $f,g \in \cP_2(\rd)$, $R\in\Pi(g,f)$ and $0\leq\eta<\e<1$ be fixed.
Let also $p>p_1(\gamma)$ be fixed and assume that $f\in L^p(\rd)$. If $\gamma \in (-2,-1]$,
we assume moreover that $g \in L^r(\rd)$ for some fixed $r>(2p-3)/[(3+\gamma)(p-1)-1]$ and that
$r<p$ without loss of generality.
We start with $\Gamma_{\eta,\e}(R)=I_0+I_1+I_2+I_3$, where
\begin{align*}
I_0:=&\int_{\rd\times\rd} \Big[ 2(z-v)\cdot\bigl(b(g^\eps,z)-b(f^\eps,v)\bigr) \\
& \hspace{20mm}+ \bigl\|\sigma(g^\eps, z)-\sigma(f^\eps, v)U\bigl(a(g^\eps,z),a(f^\eps,v)\bigr) 
\bigr\|^2 \Big] R(dz,dv),\\
I_1 :=& 2 \int_{\rd\times\rd} (z-v) \cdot \bigl(b(g^\eta,z) - b(g^\e,z) 
- b(f,v) + b(f^\eps,v)\bigr) \, R(dz,dv), \\
I_2 :=& \int_{\rd\times\rd} \Bigl \| (\sigma(g^\eps, z) - \sigma(g^\eta, z) )-
(\sigma(f^\eps, v) - \sigma(f, v))U\bigl(a(g^\eps,z),a(f^\eps,v)\bigr)  \Bigr\|^2 \,
R(dz,dv), \\
I_3  :=& 2 \int_{\rd\times\rd} \Tr \Bigl[\Bigl(\sigma(g^\eps, z)-\sigma(f^\eps, v)U\bigl(a(g^\eps,z),a(f^\eps,v)\bigr) 
\Bigr) \\
& \hspace{10mm} \Bigl( \bigl(\sigma(g^\eps, z) - \sigma(g^\eta, z) \bigr)-
\bigl(\sigma(f^\eps, v) - \sigma(f, v)\bigr)U\bigl(a(g^\eps,z),a(f^\eps,v)\bigr)  \Bigr)^*\Bigr]  \, 
R(dz,dv).
\end{align*}
Lemma \ref{cru} tells us that
$$
I_0\leq 8 \int_{\rd\times\rd} \int_{\rd\times\rd} 
|z-v|^2 K(|z-z_*|,|v-v_*|) R^\eps(dz_*,dv_*)R(dz,dv),
$$
with $R^\e \in \Pi(g^\e,f^\e)$ defined by $R^\e=\int_{\rd\times\rd}\intrd \delta_{(u+z,u+v)}
\phi_\e(u)R(dz,dv)$.

\vip

We now prove (i). We thus assume that $\gamma \in (-1,0)$. We start with $I_0$. 
Since $\gamma\in (-1,0)$, it holds that
$K(x,y)\leq \min\{x^\gamma,y^\gamma\} \leq y^\gamma$, whence
$$
I_0 \leq 8 \int_{\rd\times\rd}\int_{\rd\times\rd}|z-v|^2 |v-v_*|^\gamma R^\e(dz_*,dv_*)R(dz,dv).
$$
Using that $R^\e$ has $g^\e$ for second marginal, we deduce that 
$\sup_v \int_{\rd\times\rd}|v-v_*|^\gamma R^\e(dz_*,dv_*)=\sup_v \intrd |v-v_*|^\gamma f^\e(dv_*)
\leq 1+C_p ||f^\e||_{L^p}$ by \eqref{weakb2}, recall that $p>p_1(\gamma)$. Consequently,
$$
I_0 \leq C_p (1+||f^\e||_{L^p}) \int_{\rd\times\rd} |z-v|^2 R(dz,dv).
$$

Lemma~\ref{rough}-(iii) tells us that
$|b(f,v)-b(f^\e,v)| \leq C \e^{1+\gamma}$ and implies, since $\eta \in [0,\e)$, that $|b(g^\eta,z)-b(g^\e,z)| \leq 
C \e^{1+\gamma}$. Using that $x\cdot y \le |x|^2+|y|^2$, we thus find
\[
I_1 \le C \int_{\rd\times\rd}  \Big(\eps^{2+ 2\gamma} +|z-v|^2 \Big)R(dz,dv) =
C \eps^{2+ 2\gamma} + C \int_{\rd\times\rd} |z-v|^2 R(dz,dv).
\]

Using that $\| a - b\|^2 \le 2(\|a\|^2 + \|b\|^2)$, that $U$ is orthogonal and then Lemma \ref{math},
we see that 
\begin{align*}
I_2\leq&  2 \int_{\rd\times\rd} \Bigl(
\|\sigma(g^\eps, z) - \sigma(g^\eta, z) )  \|^2 + \|\sigma(f^\eps, v) - \sigma(f, v))\|^2\Bigr) \,R(dz,dv)\\
\leq & C \int_{\rd\times\rd} \Bigl(\|a(g^\eps, z) - a(g^\eta, z) )\|+\|a(f^\eps, v) - a(f, v))\|   \Bigr) \,R(dz,dv).
\end{align*}
Lemma~\ref{rough}-(iii) thus implies that $I_2 \le C \eps^{2+\gamma}$, which is smaller than $C\e^{2+2\gamma}$.

\vip

Finally, to bound $I_3$, we start from the inequality 
\[
\Tr (AB) \le \|A\|\, \|B\| \le \frac 12(
\eps^{-\gamma} \|A\|^2 + \eps^{\gamma} \|B\|^2)
\]
from which 
$$
I_3 \leq \e^{-\gamma} \int_{\rd\times\rd} \Bigl\| \sigma(g^\eps, z)-\sigma(f^\eps, v)U\bigl(a(g^\eps,z),a(f^\eps,v)\bigr) 
\Bigr\|^2 \,R(dz,dv) + \e^\gamma  I_2.
$$
First, we already know that $I_2 \leq C \e^{2+\gamma}$.
Next, using the same argument as in the very beginning of the proof of Lemma \ref{cru}, since 
$R^\e$ has $g^\e$ and $f^\e$ for marginals, we deduce that
$$
I_3 \leq  C \e^{2+2\gamma}+\e^{-\gamma} \int_{\rd\times\rd} \int_{\rd\times\rd} 
\Bigl\|\sigma(z-z_*)- \sigma(v-v_*)\Bigr\|^2 R^\e(dz_*,dv_*)\,R(dz,dv).
$$
By Lemma \ref{lem:Lip}, 
$$
I_3 \leq  C \e^{2+2\gamma}+C \e^{-\gamma} \int_{\rd\times\rd} (|z-v|^2+|z_*-v_*|^2)(|z-z_*|^{\gamma}+|v-v_*|^\gamma)
R^\e(dz_*,dv_*) \,R(dz,dv).
$$
Recalling the definition of $R^\e$ and using some symmetry arguments, we find that
\begin{align*}
I_3 \leq&  C \e^{2+2\gamma}+C \e^{-\gamma} \int_{\rd\times\rd} \int_{\rd\times\rd} \intrd |z-v|^2 (|z-z_*+u|^{\gamma} 
+|v-v_*+u|^{\gamma}) \phi_\e(u)du\; \\
&\hskip10cm 
R(dz_*,dv_*) \,R(dz,dv).
\end{align*}
But $\sup_{x\in \rd} \intrd |x-u|^\gamma \phi_\e(u)du \leq C \e^\gamma$. Finally,
$$
I_3 \leq  C \e^{2+2\gamma}+C \int_{\rd\times\rd} |z-v|^2 \,R(dz,dv).
$$

We next prove (ii) and thus assume that $\gamma \in (-2,-1]$. 
We need to improve the estimates of (i), but the additional
integrability allows us to do that. We first observe that our conditions
$p>p_1(\gamma)$ and $r> (2p-3)/[(3+\gamma)(p-1)-1]$ imply that $p>3/2$, that
$0\leq  p|1+\gamma|/(2p-3) < 1$ and that $p|1+\gamma|/(2p-3)<4+\gamma-3/r$.
Hence, we can find $\delta \in (0,1)$ such that $\delta< 4+\gamma-3/r$
and $\delta> p|1+\gamma|/(2p-3)$. We fix such a $\delta$ for the whole proof.

\vip

We first treat $I_0$. By the  Young inequality
\begin{align*}
K(x,y) & = \frac{[\min\{x,y\}]^{1+\gamma}}{\max\{x,y\}} \leq \frac{x^{1+\gamma}}y + y^{\gamma} 
\le C_{\delta} \bigl( x^{1+\gamma- \delta} +  y^{-a} \bigr) + y^\gamma,
\end{align*}
where $a = (\delta -1 -\gamma)/\delta$. But $a>|\gamma|$ (because $\delta<1$), whence finally
$$
K(x,y) \leq C_\delta(1+x^{1+\gamma-\delta}+y^{-(\delta-1-\gamma)/\delta}).
$$
Consequently, 
$$
I_0 \leq C_{\delta} \int_{\rd\times\rd}\int_{\rd\times\rd}|z-v|^2 (1+
|z-z_*|^{1+\gamma-\delta}   +  |v-v_*|^{-(\delta-1-\gamma)/\delta} )R^\e(dz_*,dv_*)R(dz,dv).
$$
Using that the marginals of $R^\e$ are $g^\e$ and $f^\e$, as well as \eqref{weakb2},
we conclude that 
$$
I_0 \leq C_{r,p,\delta} (1+ ||g^\e||_{L^r}+ ||f^\e||_{L^p})\int_{\rd\times\rd}|z-v|^2 R(dz,dv),
$$
provided that $r>3/[4+\gamma-\delta]$, i.e. $\delta < 4+\gamma-3/r$, and 
$p>3/[3-(\delta-1-\gamma)/\delta]$, i.e. $\delta > p|1+\gamma|/(2p-3)$.
Both hold true.

\vip

We next treat $I_1$. Lemma~\ref{rough}-(ii) implies that
\[
| b(x) - b_\eps(x)| \le C \, \min\{\eps,|x|\} \, |x|^\gamma \le C \eps^\delta  |x|^{1 +\gamma -\delta}.
\]
This implies 
\[
| b(f,v)  - b(f^\eps,v) |=| b(f,v)  - b_\e(f,v) | \le C\, \eps^\delta \, \intrd |v-v_*|^{1 +\gamma -\delta} \, f(dv_*) 
\le C_{r, \delta}\,  (1+\|f\|_{L^r}) \,\eps^\delta 
\]
by \eqref{weakb2} because $r > 3/(4 + \gamma - \delta)$.
Using a similar computation for $g$, we easily find
\[
I_1  \le  C_{r, \delta}  \bigl(1+ \|f\|_{L^r} + \|g\|_{L^r} \bigr)  
\bigl( m_1(f) + m_1(g) \bigr) \eps^\delta .
\]

As in point (i),
\begin{align*}
I_2 \leq  C \int_{\rd\times\rd} \Bigl(\|a(g^\eps, z) - a(g^\eta, z) )\|+\|a(f^\eps, v) - a(f, v))\|\Bigr) \,R(dz,dv).
\end{align*}
By Lemma \ref{rough}-(ii), we have $\|a(x)-a_\e(x)\|\leq C \e^2(\e+|x|)^\gamma \leq C \e^{1+\delta} |x|^{1+\gamma-\delta}$,
so that, with the same arguments as when bounding $I_1$,
$$
\| a(f^\eps,v) - a(f,v) \| \le  C \, \eps^{1+ \delta} \intrd |v-v_*|^{1+\gamma- \delta} \,f(dv_*)  
\le C_{r, \delta}  (1+\|f\|_{L^r})\, \eps^{1+ \delta},
$$
and the same holds for $g$. Consequently,
\[
I_2 \le C_{r,\delta} \bigl(1+ \|f\|_{L^r} + \|g\|_{L^r} \bigr)   \eps^{1+ \delta}.
\]

To bound $I_3$, we start as in (i): for any $\alpha>0$,
$$
I_3 \leq \e^{\alpha} \int_{\rd\times\rd} \Bigl\|\sigma(g^\eps, z)-\sigma(f^\eps, v)U\bigl(a(g^\eps,z),a(f^\eps,v)\bigr)
\Bigr\|^2 \,R(dz,dv) + \e^{-\alpha}  I_2.
$$
Exactly as in (i), we deduce that
\begin{align*}
I_3 \leq&  \e^{-\alpha}  I_2 +C \e^{\alpha} 
\int_{\rd\times\rd} \int_{\rd\times\rd} \intrd |z-v|^2 (|z-z_*+u|^{\gamma} 
+|v-v_*+u|^{\gamma}) \phi_\e(u)du\; \\
&\hskip9cm R(dz_*,dv_*) \,R(dz,dv).
\end{align*}
But for all $x\in \rd$, $\intrd |x+u|^\gamma \phi_\e(u)du \leq C(\e+|x|)^{\gamma} \leq C\e^{-1+\delta}|x|^{1+\gamma-\delta}$,
whence
\begin{align*}
I_3 \leq&  C \e^{-\alpha}  I_2 +C_{r,\delta} \e^{\alpha-1+\delta}
\int_{\rd\times\rd} \int_{\rd\times\rd} \intrd |z-v|^2 (|z-z_*|^{1+\gamma-\delta} +|v-v_*|^{1+\gamma-\delta})\\
& \hskip 9cm R(dz_*,dv_*) \,R(dz,dv).
\end{align*}
With the same arguments as in the bound of $I_1$, we conclude,
using that $R \in \Pi(g,f)$, that
\begin{align*}
I_3 \leq & C \e^{-\alpha}  I_2 +C_{r,\delta} \e^{\alpha-1+\delta} (1+||f||_{L^r}+||g||_{L^r}) 
\int_{\rd\times\rd} |z-v|^2 \, R(dz,dv) \\
\leq & C_{r,\delta} (1+||f||_{L^r}+||g||_{L^r}) \Big(\e^{1+\delta-\alpha} 
+\e^{\alpha-1+\delta} \int_{\rd\times\rd} |z-v|^2 \, R(dz,dv)\Big).
\end{align*}
Choosing $\alpha=1-\delta$, we end with
$$
I_3 \leq C_{r,\delta} (1+||f||_{L^r}+||g||_{L^r}) \Big(\e^{2\delta} 
+ \int_{\rd\times\rd} |z-v|^2 \, R(dz,dv)\Big).
$$
This completes the proof.
\end{proof}

\subsection{Strong/weak stability principles}\label{secwp}

The rest of the section is devoted to the 

\begin{proof}[Proof of Theorem \ref{wp}.]
Let $\gamma \in (-2,0)$ and $f_0 \in \cP_2(\rr^3)$ satisfy also $H(f_0)<\infty$ and $m_q(f_0)<\infty$ 
for some $q >q(\gamma)$. We consider the unique weak solution $f$ to \eqref{HL3D} starting from $f_0$ built in
Theorem \ref{thfg}, as well as another weak solution $g \in L^\infty_{loc}([0,\infty),\cP_2(\rd))$
starting from $g_0 \in\cP_2(\rd)$. We fix $p\in (p_1(\gamma),p_2(\gamma,q))$. We know that
$f \in L^\infty_{loc}([0,\infty),\cP_2(\rr^3))\cap L^1_{loc}([0,\infty),L^p(\rr^3))$. If finally
$\gamma \in (-2,-1]$, we assume that $g\in L^1_{loc}([0,\infty),L^r(\rr^3))$
for some $r>(2p-3)/[(3+\gamma)(p-1)-1]$. We assume without loss of generality that $r<p$.

\vip

{\it Step 1.} By Lemma \ref{regest}-(i)-(ii)-(iii), we see that $\sigma(g_t,z)$ 
and $b(g_t,z)$ satisfy the assumptions
of Proposition \ref{pfi}. Indeed, $||\sigma(g_t,z)||^2\leq C(1+|z|^2)$ and, if $\gamma \in[-1,0)$, 
$|b(g_t,z)| \leq C(1+|z|)$.
If now $\gamma\in (-2,-1)$, since $r>3/(4+\gamma)$, $b(g_t,z)\leq C(1+||f_t||_{L^r})$, 
which belongs to $L^1_{loc}([0,\infty))$ by assumption.
Consequently, we can find, on some probability space, a Brownian motion
$(\BB(t))_{t \ge 0}$, independent of a $g_0$-distributed random variable $\ZZ(0)$
such that there is a solution $(\ZZ(t))_{t\geq 0}$ to
\begin{equation*}
\ZZ(t)=\ZZ(0) + \intot b(g_s,\ZZ(s))ds + \intot \sigma(g_s, \ZZ(s)) d\BB(s)
\end{equation*}
satisfying $\LL(\ZZ(t))=g_t$ for all $t\geq 0$.

\vip

{\it Step 2.} Let $\VV(0)$ with law $f_0$ be such that $\E[|\VV(0)-\ZZ(0)|^2]= W_2^2(f_0,g_0)$.
Recall that the function $U$ taking value in the set of orthogonal matrices was defined in \eqref{dfU}.
We claim that for any $\eps \in (0,1)$, there is a strong solution $(\VV_\e(t))_{t \ge 0}$ to
\begin{align}
\VV_\e(t)=&\VV(0) + \intot b(f_s,\VV_\e(s))ds + \intot \sigma(f_s, \VV_\e(s)) 
U\bigl(a(g_s^\eps,\ZZ(s)),a(f_s^\eps,\VV_\e(s))\bigr)  d\BB_s, \label{nsdecoup1}
\end{align}
where $f_s^\eps := f_s \star \phi_\eps$ and that furthermore, $\LL(\VV_\e(t))=f_t$ for all $t\geq 0$.

\vip

Recalling that $(\ZZ(t))_{t\geq 0}$ has already been built in Step 1,
the fact that strong uniqueness holds \eqref{nsdecoup1} classically 
follows the facts, which we will check below, that 
$b(f_s,v)$ is Lipschitz in $v$ (with a constant locally integrable in time), 
$\sigma(f_s,v)$ is Lipschitz in $v$ (with a constant locally square-integrable in time), the boundedness 
and local Lipschitz property of $v \mapsto U\bigl(a(g_s^\eps,\ZZ(s)),a(f_s^\eps,v)\bigr)$.

\vip

Next, recalling that  $U(\Sigma_1,\Sigma_2)$ is orthogonal for
all symmetric positive matrices $\Sigma_1$ and $\Sigma_2$, we deduce that 
$\beta(t)=\intot U(a(f_s,\VV_\e(s)),a(g_s,\ZZ(s)))d\BB(s)$ is also a $3D$ Brownian motion: it suffices to note
that $\beta$ is a $3D$ martingale and that its quadratic variation matrix is $I_3 t$. Hence,
$\VV_\e$ also solves the linear SDE~\eqref{linSDE}, with initial condition $\VV(0)$ and Brownian motion 
$(\beta(t))_{t\geq 0}$, associated with the 
weak solution $f$ to~\eqref{HL3D}. Consequently, one can check as in the
proof of Theorem \ref{nsdewp}-Step 2 that  $\LL(\VV_\e(t))=f_t$ for all $t\geq 0$.

\vip

As already seen in Step 1 of the proof of Theorem \ref{nsdewp}, using that
$f \in L^\infty_{loc}([0,\infty),\cP_2(\rr^3))\cap L^1_{loc}([0,\infty),L^p(\rr^3))$
and Lemma \ref{regest}-(iv)-(v), we see that $b(f_s,v)$ is Lipschitz continuous in $v$, with a Lipschitz constant 
(locally) integrable in time and that $\sigma(f_s,v)$ is Lipschitz continuous in $v$, with a Lipschitz constant 
(locally) square integrable in time as desired. It only remains to check that
$v \mapsto U\bigl(a(g_s^\eps,\ZZ(s)),a(f_s^\eps,v)\bigr)$ is locally Lipschitz continuous.
First, we obviously have, by convexity and invariance by translation of the entropy,
$H(f^\eps_s) = H(f_s \star \phi_\eps) \le H(\phi_\eps) \leq C_\e$, while 
$m_2(f^\eps_s) \leq 2 m_2(f_s) + 2 \e^2$. The same arguments apply for $g^\eps_s$, so that,
since the second moment of $f$ and $g$ are (locally) bounded in time by assumption,
we get that 
\[
\sup_{s \in [0,T]} H(f^\eps_s) + H(g^\eps_s) + m_2(f^\eps_s) + m_2(g^\eps_s) \le C_{\eps,T}.
\]
Consequently, we can apply Proposition \ref{ellip1} and deduce that, for $s\in [0,T]$ and $v \in \rd$,
\begin{align}\label{elll}
|| (a(f_s^\e,v))^{-1} || + ||(a(g_s^\e,v))^{-1}|| \leq  C_{\eps,T} (1+|v|)^{|\gamma|}.
\end{align}
By Lemma \ref{regest}-(i), we also see that
\begin{align}\label{util}
|| a(f_s^\e,v) || + ||a(g_s^\e,v)|| \leq  C (1+|v|)^{2}.
\end{align}
Using next Lemma \ref{regest}-(vi)-(vii) and that $||f_t^\e||_{L^p} \leq ||\phi_\e||_{L^p} \leq C_\e$,
\begin{align}\label{util2}
|| a(f_s^\e,v) - a(f_s^\e,v')|| \leq  C_\e (1+|v|+|v'|)|v-v'|.
\end{align}
From the definition~\eqref{dfU} and Lemma~\ref{math}, it is tedious but straightforward to see that
there is $i>0$ and $C>0$ such that for all symmetric positive matrices $\Sigma_1,\Sigma_2,\Sigma_2'$, 
\begin{align*}
&\bigl\|U(\Sigma_1,\Sigma_2)-U(\Sigma_1,\Sigma_2') \bigr\|  \\
\le& C\, \bigl( \| \Sigma_1^{-1/2}\| + \| \Sigma_1^{1/2}\| + \| \Sigma_2^{-1/2} \| + 
\| \Sigma_2^{1/2} \| +  \| (\Sigma_2')^{-1/2} \| + 
\| (\Sigma_2')^{1/2} \|  \bigr)^i  \| \Sigma_2 - \Sigma_2' \|.\nonumber
\end{align*}
Using this inequality with $\Sigma_1=a(g_s^\eps,\ZZ(s))$, $\Sigma_2=a(f_s^\eps,v)$ and  $\Sigma_2'=a(f_s^\eps,v')$,
using also \eqref{elll},  \eqref{util} and \eqref{util2}, we conclude that indeed, 
\begin{align*}
||U\bigl(a(g_s^\eps,\ZZ(s)),a(f_s^\eps,v)\bigr)- U\bigl(a(g_s^\eps,\ZZ(s)),a(f_s^\eps,v')\bigr)||
\leq C_{\e,T} (1+|v|+|v'|+|\ZZ(s)|)^j |v-v'|
\end{align*}
for some $j>0$. This completes the step.

\vip

{\it Step 3.} Recall that $\LL(\ZZ(t))=g_t$ and  $\LL(\VV_\e(t))=f_t$  (for any value of $\e\in(0,1)$).
We will bound $W_2^2(f_t,g_t)$ by $\E[|\ZZ(t)-\VV_\e(t)|^2]$.
Applying the It\^o formula, we directly find that
\begin{align*}
\frac{d}{dt} \E[|\ZZ(t)-\VV_\e(t)|^2]
=& \E\Bigl[2(\ZZ(t)- \VV_\e(t))\cdot\bigl(b(g_t,\ZZ(t))-b(f_t,\VV_\e(t))\bigr) \\
& \hspace{20mm} + \bigl\| \sigma(g_t, \ZZ(t))-\sigma(f_t, \VV_\e(t))U\bigl(a(g_t^\eps,\ZZ(t)),a(f_t^\eps,\VV_\e(t))
\bigr) \bigr\|^2 \Bigr].
\end{align*}
Setting $R_{\e,t}=\LL(\ZZ(t),\VV_\e(t))$, which belongs to $\Pi(g,f)$, 
we realize that with the notation of Lemma \ref{crucru},
\begin{align*}
\frac{d}{dt} \E[|\ZZ(t) -\VV_\e(t)|^2] =& \Gamma_{0,\e}(R_{\e,t}).
\end{align*}

{\it Step 4: proof of point (i).} We thus assume that $\gamma \in (-1,0)$. 
Applying Proposition \ref{crucru}-(i), we conclude that
$$
\frac{d}{dt} \E[|\ZZ(t) -\VV_\e(t)|^2] \leq C\e^{2+2\gamma} + C_p(1+||f_t^\e||_{L^p})\E[|\ZZ(t) -\VV_\e(t)|^2].
$$
But $\E[|\ZZ(0)-\VV(0)|^2]=W_2^2(f_0,g_0)$, so that, by the Gr\"onwall lemma,
\[
\E[|\ZZ(t)-\VV_\e(t)|^2] \leq \bigl(W_2^2(f_0,g_0) + C \, \eps^{2 + 2 \gamma} \bigr)
\exp\Big(C_{p} \intot (1+\|f_s^\eps\|_{L^p})  ds  \Big).
\]
Since finally $\LL(\VV_\e(t))=f_t$ and $\LL(\ZZ(t))=g_t$, it holds that 
$W_2^2(f_t,g_t) \leq \E[|\ZZ(t)-\VV_\e(t)|^2]$.
To conclude the proof, it suffices to let $\e$ tend to $0$, noting that
$\| f_t^\eps\|_{L^p} \leq \| f_t\|_{L^p}$ for all $\e\in(0,1)$.

\vip

{\it Step 5: proof of point (ii).} We assume here that $\gamma \in (-2,-1)$.
Applying Proposition \ref{crucru}-(ii), we conclude that there is $\delta>0$ such that
\begin{align*}
\frac{d}{dt} \E[|\ZZ(t)-\VV_\e(t)  |^2]   \le& C_{r,p}\, \bigl(1+ \|f_t^\e\|_{L^p} 
+ \|g_t^\e\|_{L^r}\bigr) \E[|\ZZ(t)-\VV_\e(t)|^2] \\
& +  C_{r,p} \,\bigl(1+ \|f_t\|_{L^r} + \|g_t\|_{L^r}\bigr)(1+ m_1(f_t) + m_1(g_t)) \, \eps^{\delta}.
\end{align*}
Using that $\E[|\ZZ(0)-\VV(0)  |^2]=W_2^2(f_0,g_0)$,
that $(1+ \|f_t\|_{L^r} + \|g_t\|_{L^r})(1+ m_1(f_t) + m_1(g_t))$ is (locally) time-integrable
because $f,g \in L^\infty_{loc}([0,\infty),\cP_2(\rd))\cap L^1_{loc}([0,\infty),L^r(\rd))$ by assumption,
the Gr\"onwall lemma, that $W_2^2(f_t,g_t)\leq \E[|\VV_\e(t)-\ZZ(t)  |^2]$ for any $\e\in(0,1)$
and letting $\e$ tend to $0$, we find that 
$$
W_2^2(f_t,g_t) \leq W_2^2(f_0,g_0) \exp\Bigl( C_{r,p} \intot \bigl(1+ \|f_s\|_{L^p} 
+ \|g_s\|_{L^r}\bigr) ds\Bigr).
$$
The proof is complete.
\end{proof}

%
%
%
%
%
%
%
%

\section{Moment estimates for the particle system}\label{secmom}
\setcounter{equation}{0}

The aim of this section is to check the following propagation of moments.
When $\gamma \in [-1, 0)$, we could handle a simpler proof, but the case $\gamma \in (-2,-1)$ really requires
a tedious computation.

\begin{prop}\label{momprime}
Let $\gamma \in (-2,0)$. Let $N\geq 2$ and $\eta_N \in (0,1)$. Assume that $(\VV^N_i(0))_{i=1,\dots,N}$ is 
exchangeable and that $\E[|\VV^N_1(0)|^r]<\infty$ for some $r\geq 2$. Consider the unique solution
$(\VV^N_i)_{i=1,\dots,N}$ to \eqref{ps}. For all $T>0$, there is a finite constant $C_{T,r}$ (not depending on $N$)
such that 
$$
\sup_{[0,T]} \E [|\VV_1^N(t)|^r]\leq C_{T,r}\E[|\VV^N_1(0)|^r].
$$
\end{prop}

\begin{proof} 
By the It\^o formula, setting $\varphi(x)=|x|^r$,
\begin{align*}
\E \bigl[|\VV_1^N(t)|^r \bigr] &= \E\bigl[ |\VV_1^N(0)|^r\bigr] 
+ r \intot \E\bigl[|\VV_1^N(s)|^{r-2} \VV_1^N(s).b(\tilde \mu^N_s,\VV^N_1(s))\bigr]ds   \\
& \hspace{20pt}  +\frac 1 2\intot 
\E\Bigl[ \, \Tr \Bigl( \nabla^2 \varphi(\VV^N_1(s)) a(\tilde \mu^N_s,\VV^N_1(s) \Bigr) \Bigr] ds\\
& =: \E\bigl[ |\VV_1^N(0)|^r\bigr] + I^N_t + J^N_t.
\end{align*}
Using that $b(\tilde \mu^N_s,\VV^N_1(s))=N^{-1}\sum_{j=1}^N(b\star \phi_{\eta_N})(\VV_1^N(s)-\VV_j^N(s))$, that
$(b\star \phi_{\eta_N})(0)=0$ and exchangeability, we find
\begin{align*}
I^N_t=& \frac{r(N-1)}N \intot \E\Big[|\VV_1^N(s)|^{r-2}\VV_1^N(s) \cdot
(b\star \phi_{\eta_N})(\VV_1^N(s)-\VV_2^N(s)) ds.
\end{align*}
Using again exchangeability (and that $(b\star \phi_{\eta_N})(-x)=-(b\star \phi_{\eta_N})(x)$),
we can symmetrize the expression, which gives
\begin{align*}
I^N_t=& \frac{r(N-1)}{2N} \intot \E\Big[(b\star \phi_{\eta_N})(\VV_1^N(s)-\VV_2^N(s))
\cdot \Big(\VV_1^N(s)|\VV_1^N(s)|^{r-2} - \VV_2^N(s)|\VV_2^N(s)|^{r-2} \Big)
\Big] ds.
\end{align*}
Finally, we observe that clearly, since $b(x)=-2|x|^\gamma x$ and since $\eta_N\in(0,1)$, 
$|(b\star \phi_{\eta_N})(x)|\leq C(1+|x|^{\gamma+1})$ for all $x\in\rd$. Furthermore, 
$\bigl| x|x|^{r-2}-y|y|^{r-2} \bigr| \leq C_r |x-y|(1+|x|^{r-2}+|y|^{r-2})$. All this implies that 
\begin{align*}
I^N_t \leq & C_r \intot \E\Big[ (1+ |\VV_1^N(s)-\VV_2^N(s) |^{\gamma+1}) |\VV_1^N(s)-\VV_2^N(s) | 
(1+ |\VV_1^N(s)|^{r-2} +|\VV_2^N(s)|^{r-2}) \Big] ds.
\end{align*}
Since $\gamma\in (-2,0)$, we easily check that $(1+|v-w|^{\gamma+1})|v-w|(1+|v|^{r-2}+|w|^{r-2})
\leq C_r(1+|v|^r+|w|^r)$, whence finally
\begin{align*}
I^N_t \leq & C_r \intot \E\Big[1+ |\VV_1^N(s)|^{r}+|\VV_2^N(s)|^{r}\Big] ds
\leq C_r  \intot \E\Big[1+ |\VV_1^N(s)|^{r} \Big] ds
\end{align*}
by exchangeability.

\vip

Using next the easy estimate $\|a(\tilde \mu^N_s,\VV^N_1(s))\|=
\|N^{-1}\sum_{j=1}^N(a\star \phi_{\eta_N})(\VV_1^N(s)-\VV_j^N(s))\|
\leq C N^{-1}\sum_{j=1}^N (1+|\VV_1^N(s)-\VV_j^N(s)|^{\gamma+2})$
(because $\eta_N \in (0,1)$), that $|D^2 \varphi(x)| \leq C_r |x|^{r-2}$, and exchangeability,
\begin{align*}
J^N_t & \leq  C_r \intot \E\Big[ |\VV_1^N(s)|^{r-2}\Big(1+ \frac 1 N \sum_{j\ne 1} |\VV^N_1(s)-\VV^N_j(s)|^{\gamma+2}  
\Big)  \Big] ds \\
& \leq   C_r \intot \E\Big[ |\VV_1^N(s)|^{r-2}\Big(1+ |\VV^N_1(s)|^{\gamma+2}  + |\VV^N_2(s)|^{\gamma+2}  
\Big)  \Big] ds.
\end{align*}
Using the H\"older inequality  and exchangeability again, we easily conclude that
$$
J^N_t \leq  C_r \intot \E\Big[1+ |\VV_1^N(s)|^{r+\gamma}\Big] ds \leq 
C_r \intot \E\Big[1+ |\VV_1^N(s)|^{r}\Big] ds.
$$
We have checked that $\E[|\VV_1^N(t)|^r] \leq \E[ |\VV_1^N(0)|^r] +C_r \intot \E[1+ |\VV_1^N(s)|^{r}] ds$.
The conclusion follows by the Gr\"onwall Lemma.
\end{proof}

%
%
%
%
%
%
%
%

\section{Chaos with rate}\label{secrate}
\setcounter{equation}{0}

The aim of this section is to give the proof of Theorem \ref{mrrate}.
We first introduce a suitable coupling between our particle system $(\VV^N_i)_{i=1,\dots,N}$
and an i.i.d. family $(\WW^N_i)_{i=1,\dots,N}$ of solutions to the nonlinear SDE \eqref{nsde}.
We next prove that we can control the $L^2$-norm of a suitable blob approximation 
of the limit empirical measure (that of $\WW^N_i(t)$). This allows us to apply
our strong/weak stability principle to estimate the mean squared distance between $\VV^N_1(t)$ and  $\WW^N_1(t)$.

\vip

In the whole section, we assume that $\gamma \in (-1,0)$, that $f_0 \in \cP_2(\rd)$ satisfies $H(f_0)<\infty$
and $m_q(f_0)<\infty$ for some $q\geq 8$. Observe that $8>q(\gamma)$, recall \eqref{def:pq},
because $\gamma \in (-1,0)$.
Hence we can apply Theorem \ref{thfg} and we denote by $f$ the unique weak solution to the Landau equation 
\eqref{HL3D}, which furthermore belongs to $L^1_{loc}([0,\infty),L^p(\rd))$ for any $p \in [1,p_2(\gamma,q))$.
Actually, we will only use this estimate with $p=2$, which is indeed smaller than $p_2(\gamma,q)$,
since $q\geq 8$ and $\gamma \in(-1,0)$.

\vip

We fix $N\geq 2$, an independent family of $3D$ Brownian motions $(\BB_i)_{i=1,\dots,N}$, as well as an exchangeable
family $(\VV^N_i(0))_{i=1,\dots,N}$, with $\E[|\VV^N_1(0)|^4]$ bounded (uniformly in $N$).
We consider $\eta_N\in (0,N^{-1/3})$ as in the statement and
the unique solution $(\VV^N_i(t))_{i=1,\dots,N,t\geq 0}$ to \eqref{ps} with this $\eta_N$.
Finally, we fix $\delta\in (0,1)$ (close to $0$) and $\e_N=N^{-(1-\delta)/3}$. Observe that $\eta_N < \e_N$.

\subsection{The coupling}

We recall that for $\e>0$ and $u \in \rd$, $\phi_\e(u)= (4\pi\e^3/3)^{-1}\indiq_{\{|u|<\e\}}$,
that $\mu^N_t=N^{-1}\sum_{i=1}^N \delta_{\VV^N_i(t)}$, that $\tilde \mu^N_t = \mu^N_t \star \phi_{\eta_N}$
and we introduce $\bar\mu^{N}_t=\mu^N_t\star \phi_{\e_N}$.
For $i=1,\dots,N$, by definition
\begin{align}\label{ps2}
\VV^N_i(t) = &\VV^N_i(0)+  \intot b(\tilde \mu^N_s,\VV^N_i(s))ds+  \intot \sigma(\tilde \mu^N_s,\VV^N_i(s)) d\BB_i(s).
\end{align}
By Proposition \ref{coures}, we can introduce an i.i.d. family $(\WW^N_i(0))_{i=1,\dots,N}$ 
of $f_0$-distributed random variables such that, denoting by $F^N_0$ the law
of $(\VV^N_i(0))_{i=1,\dots,N}$, the following properties hold true:
\begin{align}
&\hbox{$\E \bigl[ \sum_1^N |\VV^N_i(0)-\WW^N_i(0)|^2 \bigr] = W_2^2\bigl(F^N_0,f_0^{\otimes N} \bigr)$, } \label{ttt0} \\
&\hbox{the family $\bigl\{(\VV^N_i(0),\WW^N_i(0)),\; i=1,\dots,N \bigr\}$ is exchangeable}, \label{ttt1}\\
&\hbox{a.s., $W_2^2\bigl(N^{-1}\sum_1^N \delta_{\VV^N_i(0)},N^{-1}\sum_1^N \delta_{\WW^N_i(0)} \bigr)=N^{-1}\sum_1^N 
|\VV^N_i(0)-\WW^N_i(0)|^2 $}.  \label{ttt2}
\end{align}
We finally introduce the system of S.D.E.s with unknown $(\WW^N_i(t))_{t\geq 0, i=1\dots N}$: for $i=1,\dots,N$,
\begin{align}\label{ls}
\WW^N_i(t) = \WW^N_i(0)+  \intot b(f_s,\WW^N_i(s))ds  
+  \intot \sigma(f_s,\WW^N_i(s) )U(a(\bar\mu^{N}_s,\VV^N_i(s)),a(\bar\nu^{N}_s,\WW^N_i(s))) d\BB_i(s)
\end{align}
with the notation $\nu^N_t=N^{-1}\sum_{i=1}^N \delta_{\WW^N_i(t)}$ and $\bar \nu^{N}_t=\nu^N_t\star \phi_{\e_N}$.

\begin{lem}\label{lswp}
The system \eqref{ls} has a unique strong solution  $(\WW^N_i(t))_{t\geq 0, i=1\dots N}$. Furthermore, the family  
$((\WW^N_i(t))_{t\geq 0})_{i=1\dots N}$ is independent and for each $i=1,\dots,N$, $(\WW^N_i(t))_{t\geq 0}$ 
has the same law as the unique solution $(\VV(t))_{t\geq 0}$ to the nonlinear SDE \eqref{nsde}. 
\end{lem}

\begin{proof}
The existence and uniqueness can be checked as in Step 2 of the proof of Theorem \ref{wp}:
the only difficulty is to verify that 
$(w_1,\dots,w_N) \mapsto U(a(\bar\mu^{N}_s,\VV^N_i(s)),a( N^{-1}\sum_{k=1}^N \delta_{w_k} \star \phi_{\e_N},w_i))$
is locally Lipschitz continuous, which is not very hard using that $\bar\mu^{N}_s$ and 
$N^{-1}\sum_{k=1}^N \delta_{w_k} \star \phi_{\e_N}$ are bounded probability density functions 
(recall that $\e_N>0$ is fixed here). 

\vip

Since for each $i=1,\dots,N$, the matrix
$U^N_i(s)=U(a(\bar\mu^{N}_s,\VV^N_i(s)),a(\bar\nu^{N}_s,\WW^N_i(s)))$ is orthogonal, it follows that
the family $\beta^N_i(t)=\intot U^N_i(s) d \BB_i(s)$ consists of $N$ independent $3D$ Brownian motion.
To get convinced, it suffices to observe that these are continuous martingales and to compute the
quadratic variation matrix. We thus can write 
\begin{align*}
\WW^N_i(t) = \WW^N_i(0)+  \intot b(f_s,\WW^N_i(s))ds +  \intot \sigma(f_s,\WW^N_i(s) ) d\beta_i^N(s),
\end{align*}
and the family $((\WW^N_i(t))_{t\geq 0})_{i=1\dots N}$  consists of $N$ i.i.d.
solutions to the nonlinear SDE \eqref{linSDE}.
\end{proof}

\begin{rk}\label{ex}
The family $\{(\VV^N_i(t),\WW^N_i(t))_{t\geq 0},\; i=1,\dots, N \}$ is exchangeable.
This follows from \eqref{ttt1} and from the symmetry and well-posedness of the systems
\eqref{ps2} and \eqref{ls}.
\end{rk}

\subsection{Preliminaries}
Here we prove two easy lemmas.

\begin{lem}\label{tass} Here we only use that
$f \in L^1_{loc}([0,\infty),L^2(\rd))$ and that $m_2(f_t) = m_2(f_0)$ for all $t\geq 0$.

\vip

(i) There is a constant $c_0>0$ such that for all $t\geq 0$, $\|f_t\|_{L^2} \geq c_0$.

\vip

(ii) For any $T\geq 0$, we can find $0=t^N_0<t_1^N<\dots<t^N_{K_N} \leq T \leq t^N_{K_N+1}$,
with $K_N \leq 2T N^{1/3}$, such that $\sup_{\ell=0,\dots,K_N} (t^N_{\ell+1}-t^N_\ell)\leq N^{-1/3}$ and such that,
setting
$$
h^N(t)=\sum_{\ell=1}^{K_N+1} \|f(t^N_\ell)\|_{L^2} \indiq_{\{t\in(t^N_{\ell-1},t^N_\ell]\}},
$$
it holds that $\int_0^T h^N(t)dt \leq 2 \int_0^T \|f(t)\|_{L^2} dt$.
\end{lem}

\begin{proof}
For point (i), we recall that $\intrd f_t(dv)=\intrd f_0(dv)=1$ and $m_2(f_t)=m_2(f_0)$. We fix
$M=\sqrt{2m_2(f_0)}$ and we note that $\|f_t\|_{L^2}^2 \geq \|f_t\indiq_{B(0,M)}\|_{L^2}^2 
\geq (3/(4\pi M^3))(f_t(B(0,M)))^2$
by the Cauchy-Schwarz inequality. Next, $f_t(B(0,M))=1-f_t(B(0,M)^c)\geq 1-m_2(f_t)/M^2=1/2$. We conclude with
$c_0=(3/(16\pi M^3))^{1/2}$.

\vip

Point (ii) is not difficult: consider $a_\ell=\ell N^{-1/3} / 2$ for $\ell=0,\dots,K_N+1$, with 
$K_N=\lfloor 2 T N^{1/3}\rfloor$. Put $t^N_0=0$ and, for each $i=1,\dots,K_N+1$, consider 
$t^N_{\ell} \in (a_{\ell-1},a_\ell]$
such that $\|f(t^N_\ell)\|_{L^2} \leq 2N^{1/3} \int_{a_{\ell-1}}^{a_\ell} \|f(t)\|_{L^2}$ (here, $2N^{1/3}$ is the length
of $(a_{\ell-1},a_\ell]$). 
One easily checks
that all the conditions are satisfied.
\end{proof}

The second lemma states a few easy and standard properties of the solution to the nonlinear SDE
(which actually hold true for any SDE of which the coefficients have at most linear growth).

\begin{lem}\label{momnsde}
Recall that $\gamma \in (-1,0)$, and that $m_q(f_0)<\infty$ for some $q > 6$.
Consider the unique solution $\VV$ to the nonlinear SDE \eqref{nsde}, see Proposition \ref{nsdewp}.
For all $T>0$,
$$
\E\Big[\sup_{t\in[0,T]} |\VV(t)|^q + \Big(\sup_{0\leq s< t \leq T} \frac{|\VV(t)-\VV(s)|}{|t-s|^{1/3}}\Big)^q \Big] 
\leq C_{T,q}.
$$
\end{lem}

\begin{proof}
Since $\gamma \in (-1,0)$ and since $m_2(f_t)=m_2(f_0)$ for all $t\geq 0$, we know 
from Lemma \ref{regest}-(i)-(ii) that $b(f_t,\cdot)$ and $\sigma(f_t,\cdot)$ have at most
linear growth, uniformly in $t\geq 0$.
Since $\E[|\VV(0)|^q]=m_q(f_0)<\infty$, 
standard computations involving the Burkholder-Davis-Gundy inequality show that for all $T>0$,  
$\E[\sup_{t\in[0,T]} |\VV(t)|^q]\leq C_{T,q}$.
Standard computations again show that for all
$0\leq s < t \leq T$, $\E[|\VV(t)-\VV(s)|^q] \leq C_{T,q}|t-s|^{q/2}$. By the Kolmogorov criterion,
see Revuz-Yor \cite[Theorem 2.1 page 26]{ry}, we conclude that for any $\alpha \in(0,1/2-1/q)$, 
\[
\E\Big[\Big(\sup_{0\leq s< t \leq T} \frac{|\VV(t)-\VV(s)|}{|t-s|^{\alpha}}\Big)^q \Big] 
\leq C_{T,q}.
\]
The choice $\alpha=1/3$ is licit since $q>6$.
\end{proof}

\subsection{On the $L^2$ norm of the blob limit empirical measure}

The following proposition is an important step. Similar considerations were used
in~\cite{HaurayJabin}.

\begin{prop}\label{controletauN}
Recall that $\gamma \in (-1,0)$, that $f$ is the unique weak solution to \eqref{HL3D} 
starting from $f_0 \in \cP(\rd)$ with a finite entropy and a finite moment of order $q >6$.
Fix $T>0$ and consider $h^N$ built in Lemma \ref{tass}. Consider the solution $(\WW^N_i)_{i=1,\dots,N}$
to \eqref{ls} and recall that $\nu^N_t=N^{-1}\sum_1^N \delta_{\WW^N_i(t)}$ and $\bar \nu^N_t
=\nu^N_t \star \phi_{\e_N}$ with $\e_N=N^{-(1-\delta)/3}$.
We have
$$
\Pr\Big(\forall \; t\in[0,T], \; \|\bar \nu^N_t\|_{L^2} \leq 173 h^N(t) \Big) \geq 1- C_{T,q,\delta}N^{1- \delta q/3}.
$$
\end{prop}

Of course, $173$ is not at all the optimal constant.

\begin{proof} It is quite complicated, so we break it into several steps.

\vip

{\it Step 1.} We introduce the event $\Omega^1_{T,N}$ on which
$$
\forall \; i=1,\dots,N,\quad
\sup_{[0,T+1]} |\WW^N_i(t)| + \sup_{0\leq s < t \leq T+1} (t-s)^{-1/3} |\WW^N_i(t)-\WW^N_i(s)| \leq N^{\delta/3}
$$ 
and we prove that
$$
\Pr((\Omega^1_{T,N})^c) \leq C_{T,q}N^{1- \delta q/3}.
$$
Since each $\WW^N_i$ has the same law as the solution $\VV$ to \eqref{nsde},
$$
\Pr((\Omega^1_{T,N})^c) \leq N \Pr\Big(\sup_{[0,T+1]} |\VV(t)|
+\sup_{0\leq s < t \leq T+1} (t-s)^{-1/3} |\VV(t)-\VV(s)|\geq N^{\delta/3}\Big).
$$
Thus Lemma \ref{momnsde} and the Markov inequality give 
us $\Pr((\Omega^1_{T,N})^c) \leq C_{T,q} N.N^{-\delta q/3}$.

\vip

{\it Step 2.} We consider the natural partition $\cP_N$ of $\rd$ in cubes with edge length $\e_N$
and call $\cP_N^\delta$ the subset of its elements that intercept $B(0,N^{\delta/3})$. Observe that
$\#(\cP_N^\delta) \leq (2 (N^{\delta/3}+1) \e_N^{-1})^3 \leq 64 N^\delta \e_N^{-3}=64N$.
We claim that for any  $(x_1,\dots,x_N) \in B(0,N^{\delta/3})^N$ and any $(y_1,\dots,y_N) \in B(0,N^{\delta/3})^N$
such that $\sup_{i=1,\dots,N}|x_i-y_i|\leq \e_N$, 
$$
\| (N^{-1}\sum_1^N \delta_{y_i})\star \phi_{\e_N}\|_{L^2}^2 
\leq \frac{3731}{N^2\e_N^3} \sum_{D \in \cP_N^\delta} (\#\{i \; : \; x_i\in D\})^2.
$$
We start with
$$
(N^{-1}\sum_1^N \delta_{y_i})\star \phi_{\e_N}(v)=\frac{3}{4\pi N\e_N^3}
\#\{i\; : \; y_i \in B(v,\e_N) \}\leq \frac{3}{4\pi N\e_N^3}\#\{i\; : \; x_i \in B(v,2\e_N) \},
$$
whence
$$
(N^{-1}\sum_1^N \delta_{y_i})\star \phi_{\e_N}(v)
\leq \frac{3}{4\pi N\e_N^3}\sum_{D \in \cP_N^\delta} \#\{i\; : \; x_i \in D\} \indiq_{\{D\cap B(v,2\e_N)\ne\emptyset \}}.
$$
Consequently, setting $A=\| (N^{-1}\sum_1^N \delta_{y_i})\star \phi_{\e_N}\|_{L^2}^2$,
$$
A 
\leq \frac{9}{16 \pi^2 N^2\e_N^6}\sum_{D,D' \in \cP_N^\delta} \#\{i\; : \; x_i \in D\} \#\{i\; : \; x_i \in D'\} 
\intrd \indiq_{\{D\cap B(v,2\e_N)\ne\emptyset, D'\cap B(v,2\e_N)\ne\emptyset \}} dv.
$$
Using that $2xy\leq x^2+y^2$ and a symmetry argument, we find that
$$
A \leq \frac{9}{16 \pi^2 N^2\e_N^6}\sum_{D \in \cP_N^\delta} (\#\{i\; : \; x_i \in D\})^2
\intrd \indiq_{\{D\cap B(v,2\e_N)\ne\emptyset\}} \sum_{D' \in \cP^\delta_N}\indiq_{\{D'\cap B(v,2\e_N)\ne\emptyset \}} dv.
$$
But for each $v\in\rd$,  $\sum_{D' \in \cP^\delta_N}\indiq_{\{D'\cap B(v,2\e_N)\ne\emptyset \}}=\#\{D' \in\cP^\delta_N\; : \;
D'\cap B(v,2\e_N)\ne\emptyset \} \leq 5^3$ and for each $D\in\cP_N^\delta$, we easily check that
$\intrd \indiq_{\{D\cap B(v,2\e_N)\ne\emptyset\}}dv \leq 4\pi (5\e_N)^3/3$. Finally, we have checked that 
$$
A \leq \frac{9}{16 \pi^2 N^2\e_N^6}\times \frac{5^34\pi (5\e_N)^3}{3}
\sum_{D \in \cP_N^\delta} (\#\{i\; : \; x_i \in D\})^2 \leq \frac{3731}{N^2\e_N^3}
\sum_{D \in \cP_N^\delta} (\#\{i\; : \; x_i \in D\})^2.
$$

{\it Step 3.} We now fix $t\in[0,T+1]$, we consider the event 
$$
\Omega^2_{t,N}=\Big\{N^{-2}\e_N^{-3} \sum_{D \in \cP_N^\delta} (\#\{i \; : \;
\WW_i^N(t)\in D\})^2 \leq 8 \|f_t\|_{L^2}^2 \Big\}
$$
and we prove that there are some positive constants $C$ and $c$ (depending only on $\delta$ and
$c_0$, recall that $\|f_t\|_{L^2}\geq c_0$ by Lemma \ref{tass}) such that
$$
\Pr((\Omega^2_{t,N})^c) \leq C \exp(-c N^{\delta/2}).
$$
To this end, we introduce, for $D \in  \cP_N^\delta$, $Z_D=\#\{i \; : \;\WW_i^N(t)\in D\}$. It follows a 
Binomial$(N,f_t(D))$-distribution. Next, it thus holds that 
\begin{equation}\label{devbin}
\Pr(Z_D \geq x)\leq \exp(-x/8) \quad \hbox{for all $x\geq 2Nf_t(D)$.}
\end{equation}
Indeed, $\E[\exp(Z_D)]=\exp(N\log(1+f_t(D)(e-1)))\leq \exp(N(e-1)f_t(D))$, whence
$\Pr(Z_D \geq x)\leq \exp(-x+N(e-1)f_t(D))\leq  \exp(-x+(e-1)x/2)$ if $x\geq 2Nf_t(D)$.
And $1-(e-1)/2>1/8$ holds true.
We next observe that on the one hand, by the Cauchy-Schwarz inequality,
$$
\|f_t\|_{L^2}^2 \geq \sum_{D\in\cP_N^\delta} \int_D f_t^2(v)dv \geq \e_N^{-3} \sum_{D\in\cP_N^\delta} (f_t(D))^2,
$$
and on the other hand, since $\#(\cP_N^\delta) \leq 64 N^\delta \e_N^{-3}$, 
$$
\|f_t\|_{L^2}^2 \geq 64^{-1}N^{-\delta} \e_N^{3} \sum_{D\in\cP_N^\delta} \|f_t\|_{L^2}^2.
$$
All in all,
$$
8\|f_t\|_{L^2}^2 \geq \sum_{D\in\cP_N^\delta} \Big(4\e_N^{-3}(f_t(D))^2+ 16^{-1} N^{-\delta} \e_N^{3} \|f_t\|_{L^2}^2 \Big).
$$
Consequently, on the event $(\Omega^2_{t,N})^c$, there is at least one $D \in \cP_N^\delta$ for which there holds 
$Z_D^2 \geq N^2\e_N^3[4\e_N^{-3}(f_t(D))^2+16^{-1} N^{-\delta} \e_N^{3} \|f_t\|_{L^2}^2]$, whence 
\begin{align*}
\Pr((\Omega^2_{t,N})^c  ) 
\leq &\sum_{D \in \cP_N^\delta}\Pr(Z_D \geq N\e_N^{3/2}[4\e_N^{-3}f_t(D)+16^{-1} N^{-\delta} \e_N^{3} \|f_t\|_{L^2}^2]^{1/2}).
\end{align*}
But $x_N:=N\e_N^{3/2}[4\e_N^{-3}(f_t(D))^2+16^{-1} N^{-\delta} \e_N^{3} \|f_t\|_{L^2}^2]^{1/2}
\geq N\e_N^{3/2}(4\e_N^{-3}(f_t(D))^2)^{1/2}=2Nf_t(D)$, so that we can apply \eqref{devbin}.
Since we also have $x_N \geq N\e_N^{3/2}4^{-1} N^{-\delta/2} \e_N^{3/2} \|f_t\|_{L^2}=N^{\delta/2} \|f_t\|_{L^2}/4$,
\begin{align*}
\Pr((\Omega^2_{t,N})^c ) \leq & \sum_{D \in \cP_N^\delta} \exp(-N^{\delta/2} \|f_t\|_{L^2}/32)\leq
64 N \exp( -c_0 N^{\delta/2}/32).
\end{align*}
We used that $\#(\cP_N^\delta) \leq 64N$, that $N \e_N^3 = N^\delta$ and that $\|f_t\|_{L^2}\geq c_0$.
This completes the step.

\vip

{\it Step 4.} We finally introduce the event 
$$
\Omega_{T,N}= \Omega^1_{T,N} \cap \Big(\cap_{\ell=1}^{K_N+1} \Omega^2_{t^N_\ell,N}\Big),
$$
where $0=t^N_0<t_1^N<\dots<t^N_{K_N}<T<t^N_{K_N+1}$, with $K_N \leq 2T N^{1/3}$ and 
$\sup_{\ell=0,\dots,K_N} (t^N_{\ell+1}-t^N_\ell)\leq N^{-1/3}$, were defined in Lemma \ref{tass}. Recall also that
$h^N(t)=\sum_{\ell=1}^{K_N+1} \|f(t^N_\ell)\|_{L^2} \indiq_{\{t\in(t^N_{\ell-1},t^N_\ell]\}}$.

\vip

We first claim that, $\Pr(\Omega_{T,N} )\geq 1- C_{T,q,\delta} N^{1-q \delta/3}$. 
This follows from the fact that $K_N \leq 2TN^{1/3}$ and from Steps 1 and 3: we have 
$\Pr(\Omega_{T,N}^c) \leq \Pr((\Omega^1_{T,N})^c) + \sum_{\ell=1}^{K_N+1}  \Pr((\Omega^2_{t^N_\ell,N})^c)
\leq C_{T,q} N^{1-q \delta/3} + C(K_N+1)\exp(-cN^{\delta/2})
\leq C_{T,q,\delta}N^{1-q \delta/3}$. The claim follows.

\vip

It only remains to prove that indeed, $\|\bar \nu^N_t\|_{L^2} \leq 173 h^N(t)$ for all $t\in[0,T]$
on $\Omega_{T,N}$. Recall that $\bar \nu^N_t= (N^{-1}\sum_1^N \delta_{\WW^N_i(t)})\star \phi_{\e_N}$. 
On $\Omega_{T,N}$, we know that

\vip

$\bullet$ $\WW^N_i(t)$ belongs
to $B(0,N^{\delta/3})$ for all $i=1,\dots,N$ and all $t\in[0,T+1]$ (thanks to $\Omega^1_{T,N}$);

\vip

$\bullet$ for all $\ell=1,\dots,K_N+1$, all $t\in [t^N_{\ell-1},t^N_\ell]$, all $i=1,\dots,N$,
$|\WW^N_i(t)-\WW^N_i(t^N_\ell)|\leq N^{\delta/3}|t^N_\ell-t^N_{\ell-1}|\leq N^{\delta/3-1/3}=\e_N$ (due to $\Omega^1_{T,N}$ 
again);

\vip

$\bullet$ for all $\ell=1,\dots,K_N+1$, $N^{-2}\e_N^{-3} \sum_{D \in \cP_N^\delta} (\#\{i \; : \;
\WW_i^N(t^N_\ell)\in D\})^2 \leq 8 \|f_{t^N_{\ell}}\|^2$ (due to $\Omega^2_{t^N_\ell,N}$).

\vip

Using Step 2, we conclude that indeed, on $\Omega_{T,N}$, for all $t\in[0,T]$, defining $\ell$ as the index
such that $t\in(t^N_{\ell-1},t^N_\ell]$,
$$
\|\bar \nu^N_t\|_{L^2}^2 \leq 3731 N^{-2}\e_N^{-3} \sum_{D \in \cP_N^\delta} (\#\{i \; : \;
\WW_i(t^N_\ell)\in D\})^2 \leq 3731 \times 8 \|f_{t^N_{\ell}}\|^2= 29848 (h^N(t))^2.
$$
This ends the proof, since $\sqrt {29848} < 173$.
\end{proof}

\subsection{Computation of the mean squared error}

Here is the main computation of the section. 

\begin{prop}\label{groscalcul}
Recall that $\gamma \in (-1,0)$, that $f$ is the unique weak solution to \eqref{HL3D} 
starting from $f_0 \in \cP(\rd)$ with a finite entropy and with $m_8(f_0)<\infty$.
Recall that $\eta_N\in (0,N^{-1/3})$, that $\e_N=N^{-(1-\delta)/3}$ and that $(\VV^N_i)_{i=1,\dots,N}$
and $(\WW^N_i)_{i=1,\dots,N}$ are the solutions to \eqref{ps2} and \eqref{ls}.
Recall that $\mu^N_t=N^{-1}\sum_1^N \delta_{\VV^N_i(t)}$, $\nu^N_t=N^{-1}\sum_1^N \delta_{\WW^N_i(t)}$,
$\tilde \mu^N_t=\mu^N_t \star \phi_{\eta_N}$, $\bar \mu^N_t=\mu^N_t \star \phi_{\e_N}$ and  
$\bar \nu^N_t=\nu^N_t \star \phi_{\e_N}$. 
Fix $T>0$, recall that $h^N$ was defined in Lemma \ref{tass} and consider the stopping time
$$
\tau_N=\inf\{t\geq 0 \; : \; \|\bar \nu^N_t\|_{L^2} \geq 173 h^N(t) \}.
$$
There is a constant $C_{T,\delta}>0$ such that for all $t\in[0,T]$,
\[
\E \bigl[|\VV^N_1(t\land \tau_N)-\WW^N_1(t\land \tau_N)|^2 \bigr] 
\leq  C_{T,\delta} \Bigl(N^{-(1-\delta)(2+2\gamma)/3}+N^{-1/2}
+ \E \bigl[|\VV^N_1(0)-\WW^N_1(0)|^2 \bigr]  \Bigr).
\]
\end{prop}

It seems that we could remove the term $N^{-1/2}$ with some work, but this would not
improve the final result of Theorem \ref{mrrate}.

\begin{proof} We put $u_t=\E[|\VV^N_1(t\land \tau_N)-\WW^N_1(t\land \tau_N)|^2]$. To lighten notation,
we also set $U^N_i(s)=U(a(\bar\mu^{N}_s,\VV^N_i(s)),a(\bar\nu^{N}_s,\WW^N_i(s)))$. By Lemma \ref{momnsde},
$\sup_{[0,T]} m_8(f_t)<\infty$.

\vip

{\it Step 1.}
A direct application of the It\^o formula gives
\begin{align*}
u_t=&\E\Bigl[|\VV^N_1(0)-\WW^N_1(0)|^2+ 
\int_0^{t\land \tau_N} \Big(2(\VV^N_1(s)-\WW^N_1(s))\cdot(b(\tilde \mu^N_s,\VV^N_1(s))-b(f_s,\WW^N_1(s)))\\
&\hskip6cm + \|\sigma(\tilde\mu^N_s,\VV^N_1(s))-
\sigma(f_s,\WW^N_1(s) )U^N_1(s)\|^2 \Big) ds\Bigr] \\
=& \E\bigl[|\VV^N_1(0)-\WW^N_1(0)|^2\bigr] + \E\Big[\int_0^{t\land \tau_N} (I_s +J_s + K_s) ds \Big],
\end{align*}
where 
\begin{align*}
I_s :=& 2(\VV^N_1(s)-\WW^N_1(s))\cdot(b(\tilde \mu^N_s,\VV^N_1(s))-b(\nu^N_s,\WW^N_1(s)))\\
&\hskip5cm + \|\sigma(\tilde\mu^N_s,\VV^N_1(s))-
\sigma(\nu^N_s,\WW^N_1(s) )U^N_1(s)\|^2 ,\\
J_s :=& 2(\VV^N_1(s)-\WW^N_1(s))\cdot(b(\nu^N_s,\WW^N_1(s))-b(f_s,\WW^N_1(s))),\\
K_s :=& \|\sigma(\tilde\mu^N_s,\VV^N_1(s))-\sigma(f_s,\WW^N_1(s))U^N_1(s)\|^2
- \|\sigma(\tilde\mu^N_s,\VV^N_1(s))-\sigma(\nu^{N}_s,\WW^N_1(s))U^N_1(s)\|^2 .
\end{align*}
Let us notice at once that
$$
J_s \leq |\VV^N_1(s)-\WW^N_1(s)|^2 + |b(\nu^N_s,\WW^N_1(s))-b(f_s,\WW^N_1(s))|^2 =:  |\VV^N_1(s)-\WW^N_1(s)|^2 + J_s^1
$$
and, using that $||A-B||^2-||A-B'||^2 \leq ||B-B'||^2 + 2||A-B'||||B'-B||$ and that $U^N_1(s)$ is orthogonal,
$$
K_s \leq K^1_s + \sqrt{M_s K^1_s },
$$
where
\begin{align*}
K^1_s := \|\sigma(f_s,\WW^N_1(s))-\sigma(\nu^{N}_s,\WW^N_1(s))\|^2 \quad\hbox{and}\quad 
M_s :=\|\sigma(\tilde\mu^N_s,\VV^N_1(s))-\sigma(\nu^{N}_s,\WW^N_1(s))U^N_1(s)\|^2.
\end{align*}

{\it Step 2.} Using exchangeability, we realize that
\begin{align*}
\E\Big[\int_0^{t\land \tau_N} I_s \, ds \Big]=& \E\Big[\int_0^{t\land \tau_N} \frac1N \sum_{i=1}^N \Big(
2(\VV^N_i(s)-\WW^N_i(s)).(b(\tilde \mu^N_s,\VV^N_i(s))-b(\nu^N_s,\WW^N_i(s)))\\
& + \|\sigma(\tilde\mu^N_s,\VV^N_i(s))-
\sigma(\nu^N_s,\WW^N_i(s) )U(a(\bar\mu^{N}_s,\VV^N_i(s)),a(\bar\nu^{N}_s,\WW^N_i(s)))   \|^2 \Big) ds \Big].
\end{align*}
Setting $R^N_t=N^{-1}\sum_1^N \delta_{(\VV^N_i(s),\WW^N_i(s))}$, which belongs to $\Pi(\mu^N_t,\nu^N_t)$,
observing that $\tilde \mu^N_s=(\mu^N_s)^{\eta_N}$,
$\bar\mu^{N}_s=(\mu^N_s)^{\e_N}$ and $\bar\nu^{N}_s=(\nu^N_s)^{\e_N}$, and recalling 
the notation of Proposition \ref{crucru}, this 
directly gives
\begin{align*}
\E\Big[\int_0^{t\land \tau_N} I_s \, ds \Big]=& \E\Big[\int_0^{t\land \tau_N} \Gamma_{\eta_N,\e_N}(R^N_s)ds \Big].
\end{align*}
We thus use Proposition \ref{crucru}-(i) (with $p=2$ which is indeed greater than $p_1(\gamma)$)
and obtain, since $||\bar \nu^N_s||_{L^2} \leq 173 h^N(s)$
for all $s \leq \tau_N$,
\begin{align*}
\E\Big[\int_0^{t\land \tau_N} I_s \, ds \Big]\leq & C t \e_N^{2+2\gamma} + 
C \E\Big[\int_0^{t\land \tau_N} (1+h_N(s)) \frac1N \sum_{i=1}^N |\VV^N_i(s)-\WW^N_i(s)|^2    ds \Big].
\end{align*}
Using exchangeability again, we end with
\begin{align*}
\E\Big[\int_0^{t\land \tau_N} I_s \, ds \Big]\leq & C  t \e_N^{2+2\gamma} + 
C \intot (1+h_N(s)) u_s \,ds.
\end{align*}

{\it Step 3.} We now check that $\E[J^1_s + K^1_s] \leq C_T N^{-1}$ for all $s\in [0,T]$.
This will imply that 
$$\E\Big[\int_0^{t\land\tau_N} (J_s +K^1_s) ds \Big] \leq C_T N^{-1} + \intot u_s ds.$$

Recall that $m_2(f_t)= m_2(f_0)$ and $H(f_t)\leq H(f_0)$ for all $t\geq 0$ (see Theorem \ref{thfg}),
so that Lemma \ref{ellip1} ensures us that $||(a(f_t,v))^{-1}||\leq C (1+|v|)^{|\gamma|}$ for all $t\geq 0$
and all $v\in \rd$. Consequently, using Lemma \ref{math}, we deduce that
$K^1_s \leq C   (1+|\WW^N_1(s)|)^{|\gamma|} \|a(\nu^N_s,\WW^N_1(s))-a(f_s,\WW^N_1(s)) \|^2$.
By the Cauchy-Schwarz inequality, and since $\cL(\WW^N_1(s))=f_s$,
\begin{align*}
\E[K^1_s] \leq& C (1+m_{2|\gamma|}(f_s))^{1/2} \E[ \|a(\nu^N_s,\WW^N_1(s))-a(f_s,\WW^N_1(s)) \|^4 ]^{1/2}\\
\leq& C_T \E[ \|a(\nu^N_s,\WW^N_1(s))-a(f_s,\WW^N_1(s)) \|^4 ]^{1/2}.
\end{align*}
We also have $J^1_s=|b(\nu^N_s,\WW^N_1(s))-b(f_s,\WW^N_1(s))|^2$, so that, using Cauchy-Schwarz again,
$$
\E[J^1_s] \leq C \E[ \|b(\nu^N_s,\WW^N_1(s))-b(f_s,\WW^N_1(s)) \|^4 ]^{1/2}.
$$

To conclude the step, it thus suffices to prove that for $\varphi:\rd\mapsto\rr$ 
with at most quadratic growth
(which is the case of all the entries of $a$ and $b$),
with the notation $\varphi(\mu,x)=\intrd \varphi(x-y)\mu(dy)$, 
$$
\E[|\varphi(\nu^N_s,\WW^N_1(s))-\varphi(f_s,\WW^N_1(s))|^4] \leq C_{T,\varphi} N^{-2}.
$$

To this end, let us recall that there is a constant
$C>0$ such that for any sequence of i.i.d. real-valued random variables $(X_n)_{n\geq 1}$
(see Rosenthal \cite{r} for a much more general inequality),
\begin{align}\label{ros}
\E[|N^{-1}\sum_1^N (X_i-\E[X_1])|^4 ] \leq C(1+\E[X_1^4]) N^{-2}.
\end{align}

But the family $(\WW^N_i(s))_{i=1,\dots,N}$ is i.i.d. with common law $f_s$. Let us 
denote by $\E_1$ the expectation concerning only $\WW^N_1$, and by $\E_{2,N}$
the expectation concerning only $\WW_2^N,\dots,\WW^N_N$. 
We have $\varphi(\WW^N_1(s),\nu^N_s)=N^{-1}\sum_{i=1}^N\varphi(\WW^N_1(s)-\WW^N_i(s))$ and, by independence,
$\varphi(\WW^N_1(s),f_s)=(N-1)^{-1} \sum_{i=2}^N\E_{2,N}[\varphi(\WW^N_1(s)-\WW^N_i(s))]$.
As a consequence, since $(x+y)^4 \leq 8x^4+8y^4$,
\begin{align*}
\E[|\varphi(\nu^N_s,\WW^N_1(s))-\varphi(f_s,\WW^N_1(s))|^4] \leq&
8\E\Big[\Big|\Big(\frac{1}{N}-\frac{1}{N-1}\Big) \sum_1^N \varphi(\WW^N_1(s)-\WW^N_i(s)) \Big|^4\Big]\\
& \hskip-5.5cm+8\E_1\Big[
\E_{2,N}\Big[\Big|\frac{1}{N-1}\sum_2^N \varphi(\WW^N_1(s)-\WW^N_i(s))- 
\E_{2,N}\Big[\frac{1}{N-1}\sum_2^{N} \varphi(\WW^N_1(s)-\WW^N_i(s))\Big] \Big|^4 \Big]
\Big]\\
=:&8A_N+8B_N.
\end{align*}
Since $|\varphi(x)|\leq C_\varphi(1+|x|^2)$, we easily get 
$A_N \leq C_\varphi(1+m_8(f_s))N^{-4} \leq C_{\varphi,T}N^{-4}$.
Since the random 
variables $\varphi(\WW^N_1(s)-\WW^N_i(s))$ are i.i.d. under $\E_{2,N}$, we may apply \eqref{ros}, which gives
\begin{align*}
B_N \leq CN^{-2} \E_1 [1+ \E_{2,N} ( \varphi^4(\WW^N_1(s)-\WW^N_i(s)))].
\end{align*}
Using again that $|\varphi(x)|\leq C_\varphi(1+|x|^2)$, we deduce that 
$B_N \leq C_\varphi(1+m_8(f_s))N^{-2} \leq C_{\varphi,T}N^{-2}$.

\vip

{\it Step 4.} We finally show that $\sup_{[0,T]}\E[M_s] \leq C_T$.
This will imply, by Step 3 and the Cauchy-Schwarz inequality, that $\E[\sqrt{K_s^1 M_s}] \leq C N^{-1/2}$,
whence
$$\E\Big[\int_0^{t\land\tau_N} \sqrt{K_s^1 M_s} ds \Big] \leq C_T N^{-1/2}.$$

Using Lemma \ref{regest}-(i), we find that
$$
M_s \leq C (1+m_2(\tilde\mu^N_s)+m_2(\nu^N_s) + |\VV^N_1(s)|^{2+\gamma} +  |\WW^N_1(s)|^{2+\gamma}).
$$
Using exchangeability, that $\eta_N \leq 1$ and that $0\leq \gamma+2\leq2$, we see that
$$
\E[M_s] \leq C (1+ \E[ |\VV^N_1(s)|^2+|\WW^N_1(s)|^2]).
$$
We conclude the step using that $\E[|\WW^N_1(s)|^2]=m_2(f_s)=m_2(f_0)$ and that $\sup_{[0,T]} \E[ |\VV^N_1(s)|^2]$
is bounded (uniformly in $N$) by Proposition \ref{momprime}.

\vip

{\it Step 5.} We now gather everything:
$$
u_t \leq u_0 +  C \intot (1+h^N_s) u_s \, ds+ C_T( \e_N^{2\gamma+2}  + N^{-1/2}).
$$
Using the Gr\"onwall lemma, we get
$$
\sup_{[0,T]} u_t \leq C_T (u_0 + \e_N^{2\gamma+2}+N^{-1/2}) \exp\Big(C \int_0^T h^N(s)ds\Big)
\leq C_T (u_0+ \e_N^{2\gamma+2}+N^{-1/2}),
$$
since $\int_0^T h^N(s)ds \leq 2 \int_0^T \|f(s)\|_{L^2} ds$ by Lemma \ref{tass}, and since 
$f \in L^1_{loc}([0,\infty),L^2(\rd))$. This ends the proof.
\end{proof}

\subsection{Conclusion}
We finally can give the

\begin{proof}[Proof of Theorem \ref{mrrate}.]
Recall that $q\geq 8$, fix $\delta = 6 /q$ and 
recall that $\tau_N=\inf\{ t \geq 0 \; : \; \|\bar \nu^N_t\|_{L^2} \geq 173 h^N(t)\}$.
By Proposition \ref{controletauN}, we know that $\Pr(\tau_N \leq T) \leq C_{T,q,\delta}N^{1-q \delta /3} = C_{T,q}N^{-1}$.
We then write, for $t\in [0,T]$,
\begin{align*}
\E[|\VV^N_1(t)-\WW^N_1(t)|^2] & \leq \E[|\VV^N_1(t\land \tau_N)-\WW^N_1(t\land \tau_N)|^2] + 
\E[|\VV^N_1(t)-\WW^N_1(t)|^2\indiq_{\{\tau_N \leq T\}}]\\
& \leq  C_{T,\delta} \Bigl( \E[|\VV^N_1(0)-\WW^N_1(0)|^2]+ N^{-(1-\delta)(2\gamma+2)/3} + N^{-1/2} \Bigr) \\
& \hspace{1cm} + C \E[|\VV^N_1(t)|^4+|\WW^N_1(t)|^4 ]^{1/2}(\Pr(\tau_N\leq T))^{1/2}
\end{align*}
by Proposition \ref{groscalcul} and the Cauchy-Schwarz inequality. 
By assumption, $\E[|\VV^N_1(0)|^4]$ is bounded (uniformly in $N$),
so that Proposition \ref{momprime} implies that $\E[|\VV^N_1(t)|^4] \leq C_T$. 
Next, Lemma \ref{momnsde} gives us $\E[|\WW^N_1(t)|^4] \leq C_T$. As a conclusion,
\begin{align*}
\sup_{[0,T]} \E[|\VV^N_1(t)-\WW^N_1(t)|^2] \leq 
C_{T,\delta} \bigl( \E[|\VV^N_1(0)-\WW^N_1(0)|^2] +  N^{-(1-\delta)(2\gamma+2)/3} + N^{-1/2} \bigr),
\end{align*}
whence, by exchangeability,
\begin{align*}
\sup_{[0,T]} \E[W_2^2(\mu^N_t,\nu^N_t)] & \leq \sup_{[0,T]} \E\Big[\frac 1N \sum_1^N |\VV^N_i(t)-\WW^N_i(t)|^2\Big] \\
& \leq C_{T,\delta} \Big(\E\Big[\frac 1N \sum_1^N |\VV^N_i(0)-\WW^N_i(0)|^2\Big] +  N^{-(1-\delta)(2\gamma+2)/3}
+ N^{-1/2} \Big)\\
&\leq C_{T,\delta} \Big(\E\Big[W_2^2(\mu^N_0,\nu^N_0)\Big] +  N^{-(1-\delta)(2\gamma+2)/3}+ N^{-1/2} \Big)
\end{align*}
where we used \eqref{ttt2} for the last inequality.
But for each $t\in [0,T]$, $\nu^N_t$ is, by Lemma \ref{lswp}, the empirical
measure of $N$ i.i.d. $f_t$-distributed random variables. We thus infer from \cite[Theorem 1]{fgui} 
(with $d=3$, $p=2$ and $q=5$) that $\E[W_2^2(\nu^N_t,f_t)]\leq C (m_5(f_t))^{2/5} N^{-1/2} \leq C_T N^{-1/2}$ 
for any $t \in [0,T]$.
As a conclusion,
\begin{align*}
\sup_{[0,T]} \E[W_2^2(\mu^N_t,f_t)] \leq& 2 \sup_{[0,T]} \E[W_2^2(\mu^N_t,\nu^N_t)+W_2^2(\nu^N_t,f_t)] \\
\le& 2 \, C_{T,\delta} \bigl( \E[W_2^2(\mu^N_0,\nu^N_0)] + N^{-(1-\delta)(2\gamma+2)/3} + N^{-1/2}\bigr).
\end{align*}
Observing finally that $ \E[W_2^2(\mu^N_0,\nu^N_0)] \leq  
2\E[W_2^2(\mu^N_0,f_0)]+2 \E[W_2^2(\nu^N_0,f_0)]\leq 2\E[W_2^2(\mu^N_0,f_0)] + C N^{-1/2}$ ends the proof.
\end{proof}

%
%
%
%
%
%
%
%

\section{More ellipticity estimates}\label{secellip}
\setcounter{equation}{0}

We now turn to the proof of our second result on the propagation of chaos,
which includes the case where $\gamma \in (-2,0)$.

\vip

To control the singularity of the coefficients, we will need
some regularity of the law of the particle system.
Such regularity will be obtained thanks to the diffusion. We thus will need to show that 
the diffusion coefficients $a(\mu^N,v)$ are sufficiently elliptic. 
Lemma \ref{ellip1}, which proves some ellipticity of $a(f,v)$,
requires the finiteness of entropy $f$ and thus cannot apply to empirical measures.
We will rather use the following lemmas. They all rely on a geometric condition on triplet of points, 
saying roughly that they are far enough from being aligned. 

\begin{defi}
Let $\delta>0$. We say that a triplet of points $(x_1,x_2,x_3)$ satisfy the $\delta$-non alignment condition if
\begin{equation} \label{cond-balls}
|x_2 -x_1| \ge 6 \,\sqrt \delta, \qquad 
\bigl|p_{(x_2-x_1)^\perp}(x_3-x_1) \bigr| \ge  24 \, \delta  + 2\, \sqrt \delta \,  |x_3-x_1|,
\end{equation}
where $p_{(x_2-x_1)^\perp}$  is the projection onto the plane orthogonal to $x_2-x_1$.
\end{defi}

This condition is not invariant by permutation of the three points. Also, a triplet
satisfying the $\delta$-non alignment condition also satisfy the $\delta'$-non alignment condition for all 
$\delta' \in [0,\delta]$.

\begin{lem}\label{ellip2}
Let $\gamma \in (-2,0)$, $\delta\in(0,1)$ and $R>1$. 
There exists a constant $\kappa>0$ depending only on $\gamma,\delta,R$ such that,
for any triplet of points $(x_1,x_2,x_3)$ belonging to $B(0,R)$ and satisfying the
$\delta$-non alignment condition, 
for  any $f\in\cP(\rr^3)$,
$$
\inf_{|\xi|=1} \xi^* a(f,v) \xi \geq \kappa \,   (1+|v|)^\gamma   \, \inf_{k=1,2,3} f(B(x_k,\delta)).
$$
\end{lem}

The  proof of this Lemma is strongly inspired by that of Desvillettes-Villani 
\cite[Proposition 4]{dv} (which is Lemma \ref{ellip1}).
In fact, \cite[Proposition 4]{dv} may be seen as a consequence of Lemma~\ref{ellip2},
thanks to the following Lemma.

\begin{lem} \label{ellip3} 
Let $\ell \in \nn^*$, $H_0>0$ and $\EE_0>0$ be fixed.
There exist some constants $\delta \in (0,1)$, $R>0$ and $\kappa>0$ such that
for any $f \in\cP(\rd)$ with $H(f)\leq H_0$ and $m_2(f)\leq \EE_0$, there
are $x_1,x_2,x_3$ belonging to $B(0,R)$, satisfying the $(\ell \delta)$-non alignment condition and
such that
\begin{equation*}
\inf_{k=1,2,3} f \bigl(B(x_k,\delta)\bigr) \ge \kappa.
\end{equation*}
\end{lem}

We shall also need a version of Lemma~\ref{ellip3} valid uniformly on small time intervals when $f$ 
is a solution to the Landau equation~\eqref{HL3D}. 

\begin{lem} \label{ellip4}
Let $\gamma \in (-2,0)$ and $q>q(\gamma)$. Let $f_0 \in \cP_2(\rd)$ satisfy also $H(f_0)<\infty$ and 
$m_q(f_0)<\infty$. Let $T>0$, $p \in ( p_1(\gamma), p_2(\gamma,q) )$, 
and $f \in L^\infty_{loc}([0,\infty),\cP_2(\rr^3)) 
\cap  L^1_{loc}([0,\infty),L^p(\rr^3))$ the solution 
of~\eqref{HL3D} given by Theorem~\ref{thfg}. There exist
four constants $\delta_0 \in (0,1)$, $R_0>0$, $\kappa_0 >0$ and $\tau_0>0$
such that for any $t \in [0,T]$, we can find three points $x_1^t,x_2^t,x_3^t$ in $B(0,R_0)$ satisfying  
the $(4  \delta_0)$-non alignment condition and such that
\begin{equation*}
\inf_{s \in [t, t+\tau_0]} \, \inf_{k=1,2,3} f_s\bigl(B(x_k^t,\delta_0)\bigr) \ge \kappa_0.
\end{equation*}
\end{lem}

We now prove all these lemmas.

\begin{proof}[Proof of Lemma~\ref{ellip2}]
We fix $(x_1,x_2,x_3) \in B(0,R)^3$ satisfying the $\delta$-non alignment condition,
we fix  $v\in \rr^3$ and $\xi\in \rr^3$ such that $|\xi|=1$, and we 
divide the proof in two steps.

\vip

{\sl Step 1. A geometric claim.}
We introduce the cone $C$ centered at $v$, with axis $\xi$ and angle $\arcsin [\delta/(2+R+|v|)]$ 
that can also be defined by
$$
C = \biggl\{v_* \in \rr^3 \; : \; \frac{| p_{\xi^\perp}(v-v_*) |}{|v-v_*|} \le \frac \delta{2+R+|v|}  \biggr\},
$$ 
We claim that the cone $C$ can not intersect the three balls $B_k^{2\delta}:=B(x_k,2\delta)$.

\vip

We thus assume that $C$ intersects the balls $B_1^{2\delta}$ and $B_2^{2\delta}$ and show it does not intersect
$B_3^{2\delta}$. We first check that $\xi$ and $\xi_0 := (x_2-x_1)/| x_2 - x_1 |$ are almost aligned, 
in the sense that
\begin{align}
|p_{\xi^\perp} (\xi_0) | & \le \sqrt \delta \quad \hbox{and}\quad 
|p_{\xi_0^\perp} (\xi) |  \le \sqrt \delta. \label{xi//e1}
\end{align}
We may find $w_1,w_2 \in C$ such that $| w_1 - x_1 | \le 2\, \delta$ and $|w_2 -  x_2 | \le 2\, \delta$.
This implies that $|w_i| \le R + 2\delta \le R+2$ for $i=1,2$ and, starting from
$x_2- x_1 =  (x_2 - w_2)  + (w_2- v) + (v-w_1) + (w_1 - x_1)$,
\begin{align*}
|p_{\xi^\perp} (x_2-x_1) | & \le  2\, \delta + |p_{\xi^\perp}(v-w_2)| + |p_{\xi^\perp}(v- w_1)| + 2\, \delta,\\
& \le 4\,  \delta + \frac \delta {2+R+|v|} \bigl(  |v - w_2| + |v - w_1|  \bigr)  
\le 6 \, \delta.
\end{align*}
We conclude that $|p_{\xi^\perp} (\xi_0) | \le 6  \delta /| x_2 - x_1 | \le \sqrt \delta$
thanks to the first assumption in~\eqref{cond-balls}. Since finally $\xi_0$ and $\xi$ are 
unit vectors, $|p_{\xi^\perp} (\xi_0) |  \le \sqrt  \delta$ means that 
$| \xi \cdot \xi_0 |^2 \ge 1 -  \delta$, whence $|p_{\xi_0^\perp} (\xi) |  \le \sqrt  \delta$.

\vip

Next, for any $w \in C \backslash \{ v\}$, we show that there is a point $w^\ast$ on the line $(v,w)$ 
(whence $w^\ast \in C$) that is very close to $x_1$, in the sense that $|w^\ast - x_1|\leq 6\delta$. 
Let thus 
$$
w^\ast = v + \lambda^\ast (w-v) \quad \hbox{with} \quad \lambda^\ast = \frac{(w_1 -v)\cdot\xi}{(w-v)\cdot\xi},
$$ 
which satisfies $w^\ast \cdot \xi = w_1 \cdot \xi$, whence $| w_1 - w^\ast  | = | p_{\xi^\perp} (w_1- w^\ast)|$.
Then
\begin{align*}
| w_1 - w^\ast  |& \le \bigl| p_{\xi^\perp} (w_1 - v) \bigr| + \bigl| p_{\xi^\perp} (v - w^\ast) \bigr| \\
& \le \frac \delta {2+R+ |v|} \bigl( |w_1 - v| +  |w^\ast - v| \bigr) \\
& \le \frac \delta {2+R + |v|} \bigl( 2 |w_1 - v| +  |w^\ast - w_1| \bigr).
\end{align*}
We conclude that (recall that $\delta \in (0,1)$, that $R>1$ and that $|w_1|\leq R+2$)
\begin{align*}
| w^\ast - w_1 | & \le \frac{2 |w_1 - v| \delta}{ 2+R+|v|-\delta} \leq 4\delta,
\end{align*}
whence $|w^\ast - x_1| \le 6 \, \delta$ as desired.

\vip
 
This allows us to conclude that $C$ is included in a narrow  cone centered at $x_1$ and directed by $\xi_0$. 
More precisely, for any $w\in C$, 
\begin{align}\label{obnow}
\bigl|p_{\xi_0^\perp}( w - x_1) \bigr| \leq  18\,  \delta +  2\sqrt \delta \,  |w-x_1|.
\end{align}
Indeed, consider $w^\ast$ as previously, use that $x= p_{\xi^\perp} (x)+(x\cdot \xi)\xi$ for all $x$ (because
$\xi$ is unitary), recall \eqref{xi//e1} and also that $w-w^*=(1-\lambda^*)(w-v)$ whence, since $w$ belongs 
to $C$,
$|p_{\xi^\perp}(w- w^\ast)| / |w-w^\ast|\leq \delta/(2+R+|v|)$, to write
\begin{align*}
\bigl|p_{\xi_0^\perp}( w - x_1) \bigr| & =\bigl| p_{\xi_0^\perp}( w^\ast - x_1) 
+  p_{\xi_0^\perp}( w- w^\ast)  \bigr| \\
& \le 6 \, \delta + \bigl| p_{\xi_0^\perp}\bigl( p_{\xi^\perp}(w- w^\ast)  \bigr) \bigr|
+ \bigl|(w- w^\ast)\cdot \xi \bigr| \bigl| p_{\xi_0^\perp} (\xi)\bigr| \\
& \le 6 \, \delta + \frac \delta {2+R + |v|}|w-w^\ast| +  \sqrt \delta \, |w - w^\ast| \\
& \le 6 \, \delta + 2\,   \sqrt \delta\,|w-w^\ast|.
\end{align*}
Recalling that $|w-w^\ast|\leq |w-x_1|+|x_1-w^\ast| \leq |w-x_1| + 6\delta$, the conclusion follows.

\vip

Assume finally that there is $w_3 \in B_3^{2\delta} \cap C$. Then, using \eqref{obnow} with $w=w_3$,
one finds that
$$
| p_{\xi_0^\perp} (x_3- x_1) | <  | p_{\xi_0^\perp} (w_3 - x_1) | + 2 \, \delta
\le   20 \,  \delta + 2 \sqrt \delta \,  |w_3 - x_1|
\le   24 \,  \delta + 2 \sqrt \delta \,  |x_3 - x_1|.
$$
But $\xi_0=(x_2-x_1)/|x_2-x_1|$ so that $p_{\xi_0^\perp}=p_{(x_2-x_1)^\perp}$: this contradicts 
\eqref{cond-balls}. 

\vip
{\sl Step 2. Conclusion.}
We choose $k \in \{1,2,3 \}$ such that $C$ does not intersect $B_k^{2\delta}$ and we write,
recalling that $a(f,v)= \int_{\rd}a(v-v_*)f(dv_*)$ and the expression of $a$,
\begin{align*}
\xi^* a(f,v) \xi   = \intrd |v-v_*|^{\gamma+2} \Big(1 - \frac{((v-v_*)\cdot\xi )^2}{|v-v_*|^2} \Big) 
f(dv_*)
= \intrd |v-v_*|^{\gamma+2} \frac{\bigl| p_{\xi^\perp}(v-v_*)\bigr|^2}{|v-v_*|^2}  f(dv_*).
\end{align*}
Since now  $C$ does not intersect $B(x_k,\delta)\subset B_k^{2\delta}$, we find that
\begin{align*}
\xi^* a(f,v) \xi & \ge \biggl(\frac \delta {2+R + |v|} \biggr)^2 \int_{B(x_k,\delta)} |v-v_*|^{\gamma+2} f(dv_*),\\
& \ge \biggl(\frac \delta {2+R+ |v|} \biggr)^2  f(B(x_k,\delta))  \, \inf_{v_*\in B(x_k,\delta)} |v-v_*|^{\gamma+2}.
\end{align*}
Since moreover $v \notin B_k^{2\delta}$ (because $v \in C$), 
one easily bounds the infimum by
$$
\inf_{v_*\in B(x_k,\delta)} |v-v_*|^{\gamma+2}  \geq  
\max\bigl\{ |v|- R-1, \delta \bigr\}^{\gamma+2} \ge  \biggl( \frac{(1+|v|) \, \delta}{R+3} \biggr)^{\gamma+2}.
$$ 
The last inequality is easily checked separating the cases $|v| \leq R+2$ and $|v|>R+2$.
Finally, we have checked that 
$$
\xi^* a(f,v) \xi \geq
\frac{\delta^2}{(2+R+ |v|)^2} \biggl( \frac{(1+|v|) \, \delta}{R+3} \biggr)^{\gamma+2}f(B(x_k,\delta)) 
\ge \kappa (1+ |v|)^\gamma f(B(x_k,\delta)) 
$$
with $\kappa := [\delta/(R+3)]^{4 + \gamma}$. This concludes the proof.
\end{proof}

\begin{proof}[Proof of Lemma~\ref{ellip3}] Let thus $f\in \cP(\rd)$ with
$m_2(f) \leq \EE_0$, $H(f) \leq H_0$ and let $\ell \in \nn^*$.
We set $R := 1+\sqrt{2 \EE_0}$ and observe that $f( B(0,R)) \ge 1- m_ 2(f)/R^2 \geq 1/2$.

\vip

By Lemma \ref{ieth}-(ii), there is a universal constant $C>0$ such that 
for any Borelian $A \subset \rd$ (here $|\cdot|$ stands for the Lebesgue measure)
$$
|A| \leq \exp(-4(C+H_0+\EE_0)) \quad \hbox{implies} \quad f(A) \leq 1/4.
$$

We now fix
$$
\delta= \min \Big\{\frac{\exp(-4(C+H_0+\EE_0))}{2\pi R \ell^2 (24+4R)^2} , 
\Big(\frac{\exp(-4(C+H_0+\EE_0))}{ 1000 \ell^{3/2}}\Big)^{2/3}\Big\}.
$$ 
For any couple of points $y_1,y_2$ in $B(0,R)$, we introduce the zone
$$
D_{y_1,y_2} := \bigl\{ y \in B(0,R) \; : \; |p_{(y_2-y_1)^\perp}(y-y_1)| \le  24\ell\delta + 2\sqrt{\ell\delta}|y-y_1|
\bigr\}.
$$
Using the rough upperbound $|y-y_1|\leq 2R$, we see that $D_{y_1,y_2}$ is included in a truncated cylinder
with length $2R$ and radius $24 \ell\delta + 4R \sqrt{\ell\delta}$, so that
$|D_{y_1,y_2}| \leq 2\pi R (24 \ell\delta + 4R \sqrt{\ell\delta})^2 \leq 2\pi R \ell^2(24 + 4R)^2\delta$.
By definition of $\delta$, we deduce that $|D_{y_1,y_2}|\leq \exp(-4(C+H_0+\EE_0))$,
whence $f(D_{y_1,y_2}) \leq 1/4$.
Similarly, for any $x\in \rd$, $|B(x,6\sqrt{\ell \delta})| \leq 1000 \ell^{3/2}\delta^{3/2} \leq
\exp(-4(C+H_0+\EE_0))$, so that $f(B(x,6\sqrt{\ell \delta})) \leq 1/4$.

\vip

Since we can cover $B(0,R)$ with $(10 R \delta^{-1})^3$ balls of radius $\delta$ (and centered
in $B(0,R)$) and since $f(B(0,R)) \geq 1/2$, we can find $x_1 \in B(0,R)$ such that
$$
f\bigl( B(x_1,\delta) \bigr) \ge \frac12 \biggl( \frac \delta{10 R}\biggr)^3.
$$
But we know that $f(B(x_1,6 \sqrt{ \ell\, \delta})) \leq 1/4$, so that  $f(B(0,R) \setminus 
B(x_1,6 \sqrt{ \ell\, \delta})  )\geq 1/4$. Of course, $B(0,R) \setminus 
B(x_1,6 \sqrt{ \ell\, \delta})$ can also be covered by $(10 R \delta^{-1})^3$ balls of radius $\delta$
(and centered in $B(0,R) \setminus B(x_1,6 \sqrt{ \ell\, \delta})$). We deduce
that there is $x_2 \in B(0,R)\setminus B(x_1,6 \sqrt{ \ell\, \delta})$ such that
$$
f\bigl( B(x_2,\delta) \bigr) \ge \frac14 \biggl( \frac \delta{10 R}\biggr)^3.
$$
We obviously have $|x_1-x_2| \geq 6 \sqrt{ \ell\, \delta}$.
We finally recall that $f(D_{x_1,x_2}) \leq 1/4$, whence $f(B(0,R)\setminus D_{x_1,x_2})\geq 1/4$. 
We thus we can find, as usual, $x_3 \in B(0,R)\setminus D_{x_1,x_2}$ such that
$$
f\bigl( B(x_3,\delta) \bigr) \ge \frac14 \biggl( \frac \delta{10 R}\biggr)^3.
$$
By definition of $D_{x_1,x_2}$, and since $|x_1-x_2| \geq 6 \sqrt{ \ell\, \delta}$,
the triplet $(x_1,x_2,x_3)$ satisfies the $(\ell \delta)$-non alignment condition.
This completes the proof, with $\kappa := (\delta/(10R))^3/4$.
\end{proof}

\begin{proof}[Proof of Lemma~\ref{ellip4}]
Recalling that $H(f_t)\leq H(f_0)$ and $m_2(f_t)=m_2(f_0)$ for all $t\geq 0$, we can apply Lemma \ref{ellip3},
with $\ell=8$,
with the same constants for all times: there are $\delta \in (0,1)$, $R_0>1$ and $\kappa>0$ such that for all
$t\in [0,T]$, there are three points $(x_1^t,x_2^t,x_3^t)$ in $B(0,R_0)$ satisfying the $(8\delta)$-non alignment
condition and such that $f_t (B(x_k^t,\delta)) \ge \kappa$ for $k=1,2,3$.

\vip

Next, we choose a smooth function $h:\rr^3\mapsto [0,1]$ such that 
$\indiq_{\{ |x| \le 1 \}} \le h \le \indiq_{\{ |x| \le 2 \}}$.
For $k =1,2,3$, we define, for $0 \leq t \leq s$,
$$
w_k^t(s) := \int_{\rr^3}  h \biggl( \frac{v-x_k^t}\delta \biggr) \,f_s(dv).
$$
By construction, we have $w_k^t(t)\geq f_t(B(x_k^t,\delta)) \ge \kappa$. 
Using the weak formulation~\eqref{wl} and the bound~\eqref{weakb2}  
with the smooth function $\varphi = h \bigl( \delta^{-1} (\, \cdot\, - x_k^t)\bigr)$, 
we get (recall that $p\in(p_1(\gamma),p_2(\gamma,q))$ is fixed and consider $p'\in(p_1(\gamma),p)$),
for $0\leq t\leq s \leq T+1$,
\begin{align*}
\bigl| w_k^t(s) - w_k^t(t) \bigr|  \le & 
C_\delta \int_t^s \intrd\intrd (|v-v_*|^{\gamma+1}+|v-v_*|^{\gamma+2}) f_u(dv)f_u(dv_*)du \\
\leq & C_\delta \int_t^s \intrd\intrd (1+|v|^2+|v_*|^2+|v-v_*|^{\gamma}) f_u(dv)f_u(dv_*)du \\
\leq & C_\delta \int_t^s \Bigl( 1 +m_2(f_u) + \| f_u \|_{L^{ p'}}\Bigr)\,du.
\end{align*}
The last inequality uses \eqref{weakb1}. Since $f_u$ is a probability density function and since
$p>p'$, we have $\| f_u \|_{L^{ p'}} \leq \| f_u \|_{L^{ p}}^{a}$ with $a=[p(p'-1)]/[p'(p-1)] \in (0,1)$.
Using the H\"older inequality and that $1 +m_2(f_u) + \| f_u \|_{L^p} \in L^1([0,T+1])$, we deduce that,
still for $0\leq t\leq s \leq T+1$,
\begin{align*}
\bigl| w_k^t(s) - w_k^t(t) \bigr| \leq &  C_{\delta,T} (s-t)^{1-a}.
\end{align*}
From this and since $w_k^t(t)\geq\kappa$ for all $t\in [0,T]$, 
it is clear that there exists $\tau_0$ such that, setting $\kappa_0=\kappa/2$,
$$ 
\forall \; t\in [0,T], \quad \forall  s \in [t, t+ \tau_0], \quad w_k^t(s)  \ge \kappa_0.
$$ 
The bound is valid for $k= 1,2,3$. By definition of $w_k^t$, this implies that
$$
\inf_{t\in [0,T]}\inf_{s \in [t, t+ \tau_0]} \inf_{k=1,2,3} f_s\bigl( B(x_k^t,2 \delta)\bigr) \ge \kappa_0.
$$
Setting $\delta_0=2\delta$ ends the proof: since $8 \delta = 4 \delta_0$, the points $x_k^t$ satisfy the
 $(4 \delta_0)$-non alignment condition.
\end{proof}

%
%
%
%
%
%
%
%

\section{Chaos without rate} \label{sec:per}
\setcounter{equation}{0}

The aim of this section is to prove Theorem \ref{mr}.
When applying the methods of \cite{fhm1}, 
the main difficulty is that the matrix $a$ is not uniformly elliptic.
Ellipticity is important, since it provides the regularity that allows us to show that the singularity
of the coefficients is not too much visited. In the particle system \eqref{ps},
there is a lack of diffusion in some directions when the particles are almost aligned.
We thus will proceed as follows. We will first
introduce a perturbed particle system, by adding some diffusion when particles are in a bad (almost aligned)
configuration. We then will prove propagation of chaos for this perturbed system. Finally,
we will show that the artificial noise is used only with a very small probability (as $N\to \infty$).

\vip

In the rest of that section, we consider $\gamma \in (-2,0)$ and $f_0 \in \cP_2(\rd)$ satisfying $H(f_0)<\infty$
and $m_q(f_0)<\infty$ for some $q>q(\gamma)$. We consider the unique
solution $f\in L^\infty_{loc}([0,\infty),\cP_2(\rr^3)) \cap  L^1_{loc}([0,\infty),L^p(\rr^3))$ for all $p \in
(p_1(\gamma),p_2(\gamma,q))$ to~\eqref{HL3D} given by Theorem~\ref{thfg}. We recall that it satisfies 
$H(f_t) \leq H(f_0)<\infty$ and $m_2(f_t)=m_2(f_0)<\infty$. We also 
fix an arbitrary final time $T$. We finally consider, for each $N\geq 2$, an exchangeable initial condition 
$(\VV^N_1(0),\dots,\VV^N_N(0))$ satisfying the set of conditions \eqref{condichaos}.

\subsection{Definition of a perturbed system} \label{sec:defpsp}

We first introduce a few notation.

\begin{nota}\label{ddd1}
(i) Recall the constants $R_0>0$, $\delta_0>0$, $\kappa_0>0$ and $\tau_0>0$ introduced in Lemma~\ref{ellip4}.
We also denote by $n_0=\lfloor T/\tau_0 \rfloor$, so that 
$$
[0,T] \subset \bigcup_{l=0}^{n_0} \lfloor l \tau_0, (l+1) \tau_0 \rfloor.
$$
Lemma~\ref{ellip4} implies that for each $l=0,\dots,n$, we can find three points $x_1^l,x_2^l,x_3^l$ in $B(0,R_0)$
satisfying the $(4 \delta_0)$-non alignment condition~\eqref{cond-balls} and such that
\begin{equation} \label{kappa-inf}
\inf_{l=0,\dots,n_0} \; \inf_{t \in [l \tau_0, (l+1) \tau_0]} \; \inf_{k=1,2,3} f_t(B(x_k^l,\delta_0)) \ge \kappa_0.
\end{equation}

(ii) Let $h:\rr^3\mapsto [0,1]$ and $\chi: [0,\infty) \mapsto [0,1]$ 
be smooth and satisfy $\indiq_{\{|v|\leq 1\}} \leq h(v) \leq \indiq_{\{|v|\leq 2\}}$
and $\indiq_{\{r \leq 1\}} \leq \chi(r) \leq \indiq_{\{r\leq 2\}}$.
For $l=0,\dots,n_0$ and $g \in \cP(\rr^3)$, we put
$$
c_l(g) := \sum_{k=1}^3 \chi \biggl( \frac 4 {\kappa_0} \intrd h \Bigl(\frac{v-x_k^{l}}
{2\delta_0}\Bigr)g(dv)  \biggr) 
\in [0,3].
$$
\end{nota}

We also shorten the notation of the coefficients of the particle system.

\begin{nota}
For $\bm v^N=(v_1^1,\dots,v_N^N) \in (\rd)^N$, we introduce
$$
b^N_i(\bm v^N)=b\Big(\frac 1 N \sum_1^N \delta_{v_j^N} \star \phi_{\eta_N}, v_i^N\Big) \quad \hbox{and}\quad
a^N_i(\bm v^N)=a \Big(\frac 1 N \sum_1^N \delta_{v_j^N} \star \phi_{\eta_N}, v_i^N\Big),
$$
and $\sigma^N_i(\bm v^N)= (a^N_i(\bm v^N))^{1/2}$. For $l=1,\dots,n_0$, we set
$$
c^N_l(\bm v^N)=c_l\Big(\frac 1 N \sum_1^N \delta_{v_j^N}\Big).
$$
\end{nota}

We will use the following properties.

\begin{prop}\label{adnoiseprop}
Let $\gamma\in(-2,0)$.
Let $g\in \cP(\rr^3)$, $\bm v^N=(v_1,\dots,v_N)\in (\rd)^N$, and $l\in \{0,\dots,n_0\}$.

\vip

(i) We have $c_l(g)=0$ if $\inf_{k=1,2,3} g(B(x_k^l,2\delta_0)) \geq \kappa_0/2$.

\vip

(ii) We have $c_l(g)\geq 1$ as soon as $\inf_{k=1,2,3} g (B(x_k^l,4\delta_0) ) \leq \kappa_0/ 4$.

\vip

(iii) There is a  constant $C>0$ such that 
$$
|\nabla_{v_i} c_l^N(\bm v^N)| \leq \frac C{N \kappa_0 \delta_0}.
$$

(iv) There is $\kappa_1 >0 $ such that for all $\xi \in\rd$,
$$
\xi^* a_i^N(\bm v^N) \xi + |\xi|^2 c_l^N(\bm v^N) \geq  \kappa_1 (1+|v_i|)^\gamma |\xi|^2.
$$

(v) $c_l(f_t) = 0$  for all $t \in [l \tau_0, (l+1) \tau_0]$.

\vip

(vi) There is a constant $C>0$ such that for all $g_1, g_2 \in \cP(\rr^3)$,
$$
|c_l(g_1)-c_l(g_2)| \leq \frac C{\kappa_0 \delta_0} W_1(g_1,g_2) \leq \frac C{\kappa_0 \delta_0} W_2(g_1,g_2).
$$
\end{prop}

We need the following easy remark.
\begin{lem}\label{erfc}
For all $\delta>0$, there is $c_\delta>0$ such that for all $x\in\rd$, all $\mu \in \cP(\rd)$,
all $\eta \in (0,1)$, 
$$
(\mu\star \phi_\eta)(B(x,\delta)) \geq c_\delta \mu(B(x,\delta)).
$$
\end{lem}

\begin{proof}
It of course suffices to treat the case where $x=0$. Let thus $\delta>0$ be fixed. 
For $z \in B(0,\delta)$, we introduce $C_z:=\{u \in B(0,\delta/2)\; : \; u\cdot z \leq - |u||z|/2\}$.
Clearly, the Lebesgue measure $p_\delta$ of $C_z$ does not depend on $z$ and is positive.
We next claim that for any $u\in C_z$ and $\eta \in (0,1)$, $|z+\eta u|<\delta$. Indeed, this is obvious
if $|z|\leq \delta/2$, while one easily checks that $|z+\eta u|^2 \leq |z|^2<\delta^2$ when
$|z| \in[\delta/2,\delta)$. Consequently,
$$
(\mu\star \phi_\eta)(B(0,\delta))= \frac{3}{4\pi}\intrd \mu(dz) \int_{B(0,1)} du \indiq_{\{|z+\eta u|<\delta\}}
\geq  \frac{3}{4\pi}\int_{B(0,\delta)} \mu(dz) \int_{C_z} du \geq \frac{3p_\delta}{4\pi} \mu(B(0,\delta))
$$
as desired.
\end{proof}

\begin{proof}[Proof of Proposition \ref{adnoiseprop}]
Point (i) is immediate: for example, $g(B(x_k^l,2\delta_0))\geq \kappa_0/2$ implies that 
$ (4/\kappa_0) \intrd h ((v-x_k^l)/(2\delta_0))g(dv) \geq 2$, whence 
 $\chi((4/\kappa_0) \intrd h ((v-x_k^l)/(2\delta_0))g(dv))=0$.
Point (ii) is also obvious: if there is $k$ such that $g(B(x_k^l,4\delta_0))\leq \kappa_0/4$, then 
$ (4/\kappa_0) \intrd h ((v-x_k^l)/(2\delta_0))g(dv) \leq 1$, so that
$\chi((4/\kappa_0) \intrd h ((v-x_k^l)/(2\delta_0))g(dv))=1$.
Point (iii) follows from the fact that
$h$ and $\chi$ are smooth (with bounded derivatives), since 
$$
c_l^N(\bm v^N)=  \sum_{k=1}^3 \, \chi\Biggl(\frac 4 {N \kappa_0} \sum_{i=1}^N 
h\biggl(\frac{v_i^N-x_k^l}{\delta_0} \biggr)
\Biggr).
$$
We next check (iv), recalling that $a^N_i(\bm v^N)=a(\mu^N\star \phi_{\eta_N},v_i)$ and $c^N_l(\bm v^N)=c_l(\mu^N)$, 
with $\mu^N=N^{-1}\sum_1^N \delta_{v_i}$.
Assume first that $\inf_{k=1,2,3} \mu^N(B(x_k^l,4\delta_0)) \geq \kappa_0/4$. 
Then  by Lemma \ref{erfc}, we have $\inf_{k=1,2,3} (\mu^N\star \phi_{\eta_N})(B(x_k^l,4\delta_0))
\geq c_{\delta_0} \kappa_0/4$.
Since now the triplet $(x_1^l,x_2^l,x_3^l)$ satisfy the $(4 \delta_0)$-non alignment condition, we can apply Lemma~\ref{ellip2} and obtain 
$$
\xi^* a_i^N(\bm v^N) \xi \geq \kappa (1+|v_i|)^\gamma |\xi|^2 \inf_{k=1,2,3}
(\mu^N\star \phi_{\eta_N})(B(x_k^l,4\delta_0)) \geq \frac{\kappa c_{\delta_0} \kappa_0}{4} (1+|v_i|)^\gamma |\xi|^2.
$$
Assume next that $\inf_{k=1,2,3} \mu^N \bigl(B(x_k^l,4\delta_0))\leq \kappa_0/4$. Then $c^N_l(\bm v^N) \geq 1$
by point (ii), so that
$$
|\xi|^2 c^N_l(\bm v^N) \geq |\xi|^2 \geq (1+|v_i|)^\gamma |\xi|^2
$$ 
since $\gamma \le 0$. We conclude with the choice $\kappa_1:=(\kappa c_{\delta_0} \kappa_0 /4)\land 1$.
Point (v) is a direct consequence of~\eqref{kappa-inf} and of point (i).
Finally, to prove point (vi), we consider $g_1,g_2\in\cP(\rd)$ and $R\in \Pi(g_1,g_2)$ 
such that $\int_{\rd\times\rd} |v_1-v_2| R(dv_1,dv_2)= W_1(g_1,g_2)$. 
Then we write, using that
$\chi$ and $h$ are globally Lipschitz-continuous and bounded,
\begin{align*}
|c_l(g_1)-c_l(g_2)|=&
\biggl|\chi\biggl(\frac 4 {\kappa_0}  \intrd h\biggl(\frac{v-x_k^l}{2\delta_0} \biggr)  g_1(dv) \biggr) 
-\chi\biggl(\frac 4 {\kappa_0} \intrd h\biggl(\frac{v-x_k^l}{2\delta_0} \biggr) g_2(dv) \biggr)\biggr| \\
\leq& \frac C {\kappa_0}  \Big|\intrd h\biggl(\frac{v-x_k^l}{2\delta_0} \biggr) g_1(dv)
- \intrd h\biggl(\frac{v-x_k^l}{2\delta_0} \biggr) g_2(dv)\Big|\\
= &\frac C {\kappa_0}  \Big|\int_{\rd\times \rd} \biggl[h\biggl(\frac{v_1-x_k^l}{2\delta_0} \biggr) 
- h\biggl(\frac{v_2-x_k^l}{2\delta_0} \biggr)  \biggr] R(dv_1,dv_2)\Big|\\
\leq&  \frac C {\kappa_0 \delta_0} \int_{\rd\times\rd} |v_1-v_2| R(dv_1,dv_2),
\end{align*}
from which the conclusion follows.
\end{proof}

We now introduce our perturbed particle system $\bUU^N = \bigl( \UU^N_1, \ldots \UU^N_N \bigr)$, defined on the 
time interval $[0,T]$, as the solution to
\begin{align}\label{psp}
\UU^N_i(t) =&\VV^N_i(0)+  \intot b^N_i(\bUU^N_s)ds 
+  \intot \sigma^N_i(\bUU^N_s) d\BB^i(s)  \\
&+ \intot c_{\lfloor s / \tau_0\rfloor}^N(\bUU^N_s) d\WW_i(s)  + \intot c_{\lfloor s / \tau_0\rfloor}^N(\bUU^N_s)  
\nabla_{v_i} c_{\lfloor s / \tau_0 \rfloor}^N(\bUU^N_s) ds, \nonumber
\end{align}
for all  $i=1,\ldots,N$. Here $((\WW_i(t))_{t \in [0,T]})_{i=1,\ldots,N}$ is an independent family of $3D$ 
standard Brownian motions independent of $(\VV_i^N(0),(\BB_i(t))_{t\in [0,T]})_{i=1,\ldots,N}$.

\begin{prop}\label{pspwp}
Recall that $\gamma \in (-2,0)$, that $\eta_N \in (0,1)$ is fixed and that $(\VV^N_i(0))_{i=1,\dots,N}$ is 
exchangeable. There is strong existence and uniqueness for \eqref{psp}. Furthermore, if 
$\E[|\VV^N_1(0)|^r]<\infty$ for some $r\geq 2$, then 
$$
\sup_{N\geq 2} \sup_{[0,T]} \E\Big[|\UU_1^N(t) |^r\Big] <C_{T,r},
$$
for some constant $C_{T,r}$ not depending on $N$.
\end{prop}

\begin{proof}
The strong existence and uniqueness is clear, since all the coefficients are locally Lipschitz continuous
and have at most linear growth: this has already been seen in the proof of Proposition \ref{wpps} for
$b_i^N$ and $\sigma_i^N$, and $c_{\lfloor s / \tau_0\rfloor}^N$ and $c_{\lfloor s / \tau_0\rfloor}^N   
\nabla_{v_i} c_{\lfloor s / \tau_0 \rfloor}^N$ are obviously bounded and smooth in $x$ (recall that here, $N$ is fixed).
Concerning the moment estimates, it suffices to handle the same proof as that of Proposition \ref{momprime}.
The additional terms cause no difficulty (and are much easier to treat) because  $c_{\lfloor s / \tau_0\rfloor}^N$
is uniformly bounded (by $3$) and $c_{\lfloor s / \tau_0\rfloor}^N   
\nabla_{v_i} c_{\lfloor s / \tau_0 \rfloor}^N$ is uniformly bounded (by $3C/(N\kappa_0\delta_0) \leq C$) thanks to
Proposition \ref{adnoiseprop}-(iii).
\end{proof}

We will first study propagation of chaos for the perturbed particle system. We thus introduce
the corresponding nonlinear process $(\UU(t))_{t\in [0,T]}$, solution to 
\beqn\label{nsdep}
\UU(t)= \VV(0) + \intot b(g_s,\UU(s))ds + \intot \sigma(g_s,\UU(s))d\BB(s)+ 
\intot c_{\lfloor s / \tau_0\rfloor}(g_s) 
d\WW(s),
\eeqn
where $\VV(0)$ is $f_0$-distributed, independent of the (independent) $3D$-Brownian motions 
$(\BB(t))_{t \in [0,T]}$ and $(\WW(t))_{t\in [0,T]}$, and where $g_t \in \cP(\rr^3)$ is the law of $\UU(t)$.
Observe that for $(\UU(t))_{t\in [0,T]}$ a solution to \eqref{nsdep}, its family $(g_t)_{t\in [0,T]}$ of time marginals 
is a weak solution to
\begin{align}\label{edpp}
\partial_t g_t(v) = \frac 1 2 \diver_v\bigl(a(g_t,v)\nabla g_t(v) - b(g_t,v)g_t(v) + c^2_{\lfloor t/\tau_0\rfloor}(g_t)
\nabla g_t(v)   \bigr).
\end{align}

\begin{prop}\label{nsdepwp} Recall that $\gamma \in (-2,0)$ and that $f_0 \in \cP_2(\rd)$ satisfies 
$H(f_0)<\infty$ and $m_q(f_0)<\infty$ for some $q>q(\gamma)$.

\vip

(i) There exists a unique weak solution $(g_t)_{t\in [0,T]}$ to \eqref{edpp} such that $g_0=f_0$ and such that
$(g_t)_{t\in [0,T]} \in L^\infty([0,T],\cP_2(\rd)) \cap L^1([0,T],L^p(\rr^3))$ for some 
$p\in(p_1(\gamma),p_2(\gamma,q))$. Furthermore, it holds that $(g_t)_{t\in[0,T]}=(f_t)_{t\in[0,T]}$,
where $(f_t)_{t\geq 0}$ is the unique weak solution to the Landau equation \eqref{HL3D} built in
Theorem \ref{thfg}.

\vip

(ii) There is a pathwise unique continuous adapted
solution $(\UU(t))_{t\in [0,T]}$ to~\eqref{nsdep}
such that, setting $g_t=\LL(\UU(t))$,  
$(g_t)_{t\in [0,T]} \in L^\infty([0,T],\cP_2(\rd)) \cap L^1([0,T],L^p(\rr^3))$ for some $p\in(p_1(\gamma),
p_2(\gamma,q))$. Furthermore, $(g_t)_{t\in [0,T]}=(f_t)_{t\in[0,T]}$.

\vip

(iii) Finally $(\UU(t))_{t\in [0, T]}$ equals $(\VV(t))_{t\in [0, T]}$, the unique
solution to \eqref{nsde} (see Theorem \ref{nsdewp}).
\end{prop}

\begin{proof} 
Consider the unique solution $(\VV(t))_{t\geq 0}$ to the nonlinear SDE \eqref{nsde},
the unique weak solution $f \in L^\infty_{loc}([0,\infty),\cP_2(\rd)) \cap L^1_{loc}([0,\infty),L^p(\rr^3))$
to \eqref{HL3D}, and recall that $f_t=\cL(\VV(t))$.
Then we know from Proposition \ref{adnoiseprop}-(v) that $c_{\lfloor t/\tau_0\rfloor}(f_t)=0$ for all $t\in [0, T]$. 
Consequently, $(\VV(t))_{t\geq 0}$ also solves \eqref{nsdep} and $(f_t)_{t\in[0,T]}$ also solves \eqref{edpp}
(compare with \eqref{alt}). Next we claim that once the uniqueness of the weak solution to
\eqref{edpp} is established, the pathwise uniqueness of the solution $(\UU(t))_{t\in [0,T]}$ to~\eqref{nsdep}
can be checked exactly as in the proof of Theorem \ref{nsdewp}. We thus only have to prove the
uniqueness part in point (i).

\vip

Let thus $p\in (p_1(\gamma),p_2(\gamma,q))$ be fixed, and consider two solutions
$(g_t)_{t\in[0,T]}, (k_t)_{t\in [0,T]}$ to \eqref{edpp}, both lying in 
$L^\infty_{loc}([0,\infty),\cP_2(\rd)) \cap L^1_{loc}([0,\infty),L^p(\rr^3))$.
Exactly as in the proof of Theorem \ref{wp}-Steps 1-2, we can build, for each $\e>0$,
two processes $(\ZZ(t))_{t\in [0,T]}$ and $(\VV_\e(t))_{t\in [0,T]}$ with respective families
of time marginals $(g_t)_{t\in [0,T]}$ and $(k_t)_{t\in [0,T]}$, solving
\begin{align*}
\ZZ(t)=& \VV(0) + \intot b(g_s,\ZZ(s))ds + \intot \sigma(g_s, \ZZ(s)) d\BB(s) + \intot 
c_{\lfloor s / \tau_0\rfloor}(g_s) d\WW(s),\\
\VV_\e(t)=& \VV(0) + \intot b(k_s,\VV_\e(s))ds + \intot \sigma(k_s, \VV_\e(s)) 
U\bigl(a(g_s^\eps,\ZZ(s)),a(k_s^\eps,\VV_\e(s))\bigr)  d\BB_s\\
& \hskip8cm + \intot c_{\lfloor s / \tau_0\rfloor}(k_s) d\WW(s).
\end{align*}
We can then reproduce exactly the same computations as in Theorem \ref{wp}-Steps 3-4-5
to prove uniqueness in the class $L^\infty_{loc}([0,\infty),\cP_2(\rd)) \cap L^1_{loc}([0,\infty),L^p(\rr^3))$
(some strong/weak stability estimates could also be checked here).
When computing $(d/dt)\E[|\ZZ(t)-\VV_\e(t)|^2]$, there is the additional term 
$3(c_{\lfloor t / \tau_0\rfloor}(g_t)- c_{\lfloor t / \tau_0\rfloor}(k_t))^2$. But this is controlled,
using Proposition \ref{adnoiseprop}, by $CW_2^2(g_t,k_t)$ which is itself bounded by $C\E[|\ZZ(t)-\VV_\e(t)|^2]$.
This term thus causes no difficulty.
\end{proof}

\subsection{Regularity estimate for the perturbed particle system}

We now need to introduce the entropy and weighted Fisher information. These functionals are studied
in details in the appendix.

\begin{nota}\label{HetI}
For $\bVV^N$ a random variable with law $F \in\cP \bigl((\rr^3)^N \bigr)$ with a density, we define
\begin{align*} 
H(\bVV^N) = H(F) & := \frac1N\int_{(\rr^3)^N} F(\bm v^N)\log(F(\bm v^N)) \, d \bm v^N ,\\
I_\gamma(\bVV^N)= I_\gamma(F) &:= \frac1N\int_{(\rr^3)^N} \frac{| \nabla_\gamma F(\bm v^N )|^2}{F(\bm v^N)} \,
d \bm v^N
\end{align*}
where the differential operator $\nabla_\gamma$ is the weighted gradient
$$
\nabla_\gamma F := \Bigl(  (1+|v_1|^2)^{\frac \gamma 4} \nabla_{v_1} F, \ldots, 
(1+|v_N|^2)^{\frac \gamma 4} \nabla_{v_n} F  \Bigr).
$$ 
If $ F$ has no density, we simply state $H(\bVV^N)= H(F)=\infty$ and $I_\gamma(\bVV^N)=I_\gamma(F)=\infty$.
\end{nota}

The following lemma uses strongly the perturbation.

\begin{prop} \label{HIbound}
For each $N\geq 2$, we consider the solution $(\UU^N_1(t),\dots,\UU^N_N(t))_{t\in[0,T]}$ to
\eqref{psp}. For each $t \in [0,T]$, we denote by $G^N_t$ the law of  
$\bUU^N(t) =  \bigl(\UU_1^N(t),\dots,\UU^N_N(t) \bigr)$. 
Recall that by \eqref{condichaos}, $\sup_{N\geq 2} H(G^N_0) <\infty$ and $\sup_{N\geq 2} \E[|\UU_1^N(0)|^q]<\infty$
for some $q>q(\gamma)$. It holds that
$$
\sup_{N\geq 2} \, \sup_{ t \in [0,T]} H(G^N_t)<\infty \quad \hbox{and} \quad
\sup_{N\geq 2} \, \int_0^T I_\gamma(G^N_t) dt <\infty.
$$
\end{prop}

\begin{proof}
The first step is to derive the Master (or Kolmogorov) equation for the time marginals $G^N_t$.
Applying the It\^o formula to compute the expectation of $\varphi(\bUU^N_t)$, 
we deduce that for all $\varphi \in C^2_b((\rd)^N)$,
\begin{align*}
\int_{(\rr^3)^N}  \varphi & (\bm v^N)G^N_t(d \bm v^N) 
= \int_{(\rr^3)^N} \varphi(\bm v^N)G^N_0(d\bm v^N) \\
&\;\;+ \intot \int_{(\rr^3)^N} \sum_{i=1}^N \nabla_{v_i} \varphi(\bm v^N) 
\cdot  [b^N_i(\bm v^N)+ c^N_{\lfloor s/\tau_0\rfloor}(\bm v^N) \nabla_{v_i}c^N_{\lfloor s/\tau_0\rfloor}(\bm v^N)]  
G^N_s(d\bm v^N) ds \\
&\;\;+\frac 12 \intot \int_{(\rr^3)^N} \sum_{i=1}^N \sum_{k,l=1}^3  \partial_{v_{ik} v_{il}} \varphi(\bm v^N) 
\Big((a^N_i(\bm v^N))_{kl}+ (c^N_{\lfloor s/\tau_0 \rfloor}(\bm v^N))^2 \delta_{kl}\Big)  G^N_s(d\bm v^N) ds.
\end{align*}
Recall now that $b_k=\sum_{l=1}^3\partial_{l}a_{lk}$. Hence, 
we see that $G^N$ is a weak solution to 
\begin{align*}
\partial_t G^N_t(\bm v^N)
 =& \frac 1 2 \sum_{i=1}^N \sum_{k,l=1}^3  \partial_{v_{ik}} 
\bigl[  \{(a^N_i(\bm v^N))_{kl}+ (c^N_{\lfloor t/\tau_0\rfloor}(\bm v^N))^2\delta_{kl})\} 
\partial_{v_{il}} G^N_t(\bm v^N) \bigr] \\
 & - \frac 1 2 \sum_{i=1}^N\sum_{k=1}^3 \partial_{v_{ik}} \bigl[(b^N_i(\bm v^N))_k  \, G^N_t(\bm v^N) \bigr] ,
\end{align*}
or equivalently
$$
\partial_t G^N_t 
= \frac 1 2 \sum_{i=1}^N  \diver_{v_i} \bigl[  \{ a^N_i+ (c^N_{\lfloor t/\tau_0\rfloor})^2 I\} 
\nabla_{v_i} G^N_t \bigr]  
- \frac 1 2 \sum_{i=1}^N \diver_{v_i}  \bigl[b^N_i  \, G^N_t\bigr]. 
$$
We thus have, performing some integrations by parts,
\begin{align*}
\frac{d}{dt} H(G^N_t) = &\frac 1 N \int_{(\rr^3)^N} (1+\log G^N_t(\bm v^N))\partial_t G^N_t(\bm v^N)d\bm v^N \\
=& - \frac 1 {2N} \sum_{i=1}^N \int_{(\rr^3)^N} \nabla_{v_i} \log G^N_t  \cdot \Big[a^N_i(\bm v^N) + 
(c^N_{\lfloor t/\tau_0\rfloor}(\bm v^N))^2  I \Big] 
\nabla_{v_i} \log  G^N_t(\bm v^N) G^N_t(d \bm v^N)  \\
&  - \frac1{2N}  \sum_{i=1}^N \int_{(\rr^3)^N} \diver b_i^N(\bm v^N) G^N_t(d \bm v^N).
\end{align*}
We next use Proposition \ref{adnoiseprop}-(iv) for the diffusion term and, for the drift term, 
that $\diver b(v) = -(6+ 2 \gamma) |v|^\gamma$,
whence $-\diver (b\star \phi_{\eta_N})(v)=(6+ 2 \gamma) \intrd |v-z|^\gamma\phi_{\eta_N}(z)dz \leq C |v|^\gamma$,
and exchangeability. We find that
\begin{align*}
\frac{d}{dt} H(G^N_t) & \le  - 
\frac {\kappa_1} {2N} \sum_{i=1}^N \int_{(\rr^3)^N} \frac{(1+|v_i|)^\gamma |\nabla_{v_i} G^N_t(\bm v^N)|^2}
{G^N_t(\bm v^N)} d\bm v^N  + C  \int_{(\rr^3)^N}  |v_1-v_2|^\gamma G^N_t(d \bm v^N) \\
& \leq 
-  2^{\gamma/2-1}\kappa_1 I_\gamma(G^N_t) + C  \E \Bigl[  \bigl| \UU^N_1(t) - \UU^N_2(t) \bigr|^\gamma \Bigr].
\end{align*}
We used that $(1+x)^\gamma \geq 2^{\gamma/2}(1+x^2)^{\gamma/2}$.
Denote by $G^N_{t,2}$ the law of $(\UU^N_1(t),\UU^N_2(t))$.
We use Lemma~\ref{lem:FishInteg} with $\kappa=-\gamma$. Since 
$q > q(\gamma)= \gamma^2/(2+\gamma)$, we have $\alpha:=\max\{2,q\}>|\gamma|\max\{1,\kappa/(2-\kappa)\}$,
whence, with $r:=\alpha/(\alpha-\gamma)$,
\begin{align*}
\frac{d}{dt} H(G^N_t) \le& 
- 2^{\gamma/2-1}\kappa_1  I_\gamma(G^N_t) + C_{\gamma,q}(1 +  I_\gamma(G^N_{t,2})^r (1+m_\alpha(G^N_{2,t})^{1-r}))\\
\leq& - 2^{\gamma/2-1}\kappa_1  I_\gamma(G^N_t) + C_{T,\gamma,q}(1+ I_\gamma(G^N_{t})^r).
\end{align*}
We finally used the moment bound proved in Proposition \ref{pspwp} (and the moment assumption in 
\eqref{condichaos}) and Lemma~\ref{lem:IrIw}-(i), from which $I_\gamma(G^N_{t,2})\leq I_\gamma(G^N_t)$.
Since $r \in (0,1)$, we easily conclude,
using the Young inequality, that
\begin{align*}
\frac{d}{dt} H(G^N_t) \le&  - 2^{\gamma/2-2}\kappa_1I_\gamma(G^N_t) + C_{T,\gamma,q}.
\end{align*}
Integrating this inequality and using again \eqref{condichaos}, we get, for any $t\in [0,T]$, 
$$
H(G^N_t) + \frac{\kappa_1}8 \int_0^t I_\gamma(G^N_s) ds \leq H(F^N_0) + C_{T,\gamma,q} \, t  \leq C_{T,\gamma,q}.
$$
We immediately deduce that $H(F^N_t)$ is bounded uniformly in $N$ and $t \in [0,T]$. 
Finally, it follows from the fact that $\sup_{N\geq 2}\sup_{[0,T]}\E[|\UU^N_1(t)|^2]<\infty$ 
(by \eqref{condichaos} and Proposition \ref{pspwp}) and Lemma~\ref{ieth}-(i) 
that $\inf_{N\geq 2}  H(F^N_T) > - \infty$. Consequently, $\int_0^T I_\gamma(F^N_s) ds$ is also uniformly bounded.
This completes the proof.
\end{proof}

\subsection{Some more estimates}

We will of course need, in several steps, to control the singularity of $b$.
Also, to verify that the limit points of the empirical measure of the particle system 
belong, in some sense, to $L^\infty([0,T],\cP_2(\rd))$,
we will need to control $\E[\sup_{[0,T]}|\UU^N_1(s)|^2]$, with the supremum inside the expectation.
All this is more or less obvious when $\gamma \in (-1,0)$, but requires
a little work when $\gamma\in (-2,-1]$, based on the regularity estimate checked in the previous
subsection.

\begin{lem}\label{singetmom}
For each $N\geq 2$, we consider the solution $(\UU^N_1(t),\dots,\UU^N_N(t))_{t\in[0,T]}$ to
\eqref{psp}. Recall that by \eqref{condichaos}, $\sup_{N\geq 2} \E[|\UU_1^N(0)|^q]<\infty$
for some $q>q(\gamma)$. It holds that

\vip

(i) $\sup_{N\geq 2} \E[ \sup_{ t \in [0,T]} |\UU^N_1(s)|^2 ] <\infty$;

\vip

(ii) $\sup_{N\geq 2}\int_0^T \E[|\UU^N_1(s)-\UU^N_2(s)|^{\gamma}] ds <\infty$.
\end{lem}

Here we prove point (i) using point (ii), but it seems that a refinement of the proof
of Proposition \ref{momprime} could also work.

\begin{proof}
We put $\alpha=2\lor q$ and recall that $\sup_{N\geq 2} \sup_{ t \in [0,T]} \E[|\UU^N_1(s)|^\alpha ] <\infty$ 
by Proposition \ref{pspwp}. We first prove (ii). Denote by $G^N_{s,2}$ the two-marginal of $G^N_s$, which is
the law of $(\UU^N_1(s),\UU^N_2(s))$. Observe that 
$\alpha >|\gamma|\max\{1,|\gamma|/(2+\gamma)\}$ 
(because $\alpha \geq 2$ and $\alpha\geq q>q(\gamma)=\gamma^2/(2+\gamma)$).
By Lemma \ref{lem:FishInteg}, we know that, with $r=\alpha/(\alpha-\gamma)$, 
$\E[|\UU^N_1(s)-\UU^N_2(s)|^{\gamma}] \leq C_{\kappa,\alpha} (1+ (I_\gamma(G^N_{s,2}))^r (1+m_\alpha(G^N_{s,2}))^{1-r})
\leq C_{\kappa,\alpha,T} (1+I_\gamma(G^N_{s,2}))$ because $r \in (0,1)$ and $m_\alpha(G^N_{s,2})=\E[|\UU^N_1(s)|^\alpha]$
is controlled uniformly in $N\geq 2$ and $s\in [0,T]$. Finally, we know from Lemma \ref{lem:IrIw}-(i) that
$I_\gamma(G^N_{s,2})\leq I_\gamma(G^N_{s})$, whence $\E[|\UU^N_1(s)-\UU^N_2(s)|^{\gamma}] 
\leq C_{\kappa,\alpha} (1+ (I_\gamma(G^N_{s}))$. Integrating in time and using Proposition \ref{HIbound} 
completes the proof of (ii). 

\vip

To prove (i), we start from $\sup_{[0,T]} |\UU^N_1(t)|^2 \leq  C [|\UU^N_1(0)|^2 + I^N+J^N+K^N+L^N]$,
where
\begin{align*}
I^N:=& \sup_{[0,T]} \Big(\intot b^N_1(\bm \UU^N(s))ds \Big)^2, \\
J^N:=& \sup_{[0,T]} \Big(\intot \sigma^N_1(\bm \UU^N(s))d\BB_1(s) \Big)^2\\
K^N:=& \sup_{[0,T]} \Big(\intot c^N_{\lfloor s/\tau_0\rfloor}(\bm \UU^N(s))d\WW_1(s) \Big)^2\\
L^N=&\sup_{[0,T]} \Big(\intot c^N_{\lfloor s/\tau_0\rfloor}(\bm \UU^N(s)) \nabla_{v_1} 
c^N_{\lfloor s/\tau_0\rfloor}(\bm \UU^N(s)) ds \Big)^2.
\end{align*}
First, $\sup_N \E[|\UU^N_1(0)|^2]<\infty$ by assumption \eqref{condichaos} and 
$\E[K^N+L^N]$ is uniformly bounded because $c^N_l$ and $\nabla_{v_1} c^N_l$ are uniformly bounded,
see Proposition \ref{adnoiseprop}.
By Doob's inequality,
\begin{align*}
\E[J^N] \leq C \int_0^T \E[\|\sigma^N_1(\bm \UU^N(s))\|^2] ds \leq C \int_0^T 
\E\Big[\frac 1 N \sum_{j=1}^N\|(a\star \phi_{\eta_N}(\UU^N_1-\UU^N_j(s))\|\Big] ds.
\end{align*}
Using exchangeability and that $\|(a\star\phi_{\eta_N})(v)\|\leq C(1+|v|^{\gamma+2})\leq C(1+|v|^2)$,
\begin{align*}
\E[J^N]\leq C \int_0^T \E[1+|\UU^N_1(s)-\UU^N_2(s)|^2]ds
\leq C \int_0^T \E[1+|\UU^N_1(s)|^2]ds,
\end{align*}
which is also uniformly bounded. Finally, H\"older's inequality, exchangeability,
and the inequality $|(b\star \phi_{\eta_N})(v)|\leq C (1+|v|^{\gamma+1})$ lead us to
\begin{align*}
\E[I^N] \leq C_T \int_0^T \E[|(b\star \phi_{\eta_N})(\UU^N_1(s)-\UU^N_2(s))|^2] ds 
\leq C_T\int_0^T \E[1+|\UU^N_1(s)-\UU^N_2(s)|^{2\gamma+2}]ds.
\end{align*}
If $\gamma \in [-1,0)$, we bound $\E[|\UU^N_1(s)-\UU^N_2(s)|^{2\gamma+2}]$
by  $C E[1+|\UU^N_1(s)|^2+|\UU^N_2(s)|^2]\leq C(1+\E[|\UU^N_1(s)|^2])$ and immediately deduce that $\E[I^N]$
is uniformly bounded.
If now $\gamma \in (-2,-1)$, then it holds true that $\gamma<2\gamma+2<0$, so that also we conclude,
using point (ii), that
$\E[I^N]$ is uniformly bounded.
\end{proof}

\subsection{Tightness for the perturbed particle system}

We can now prove the tightness of our perturbed particle system.

\begin{prop}\label{tight}
We still assume \eqref{condichaos} and consider, for each $N\geq 2$, 
the unique  solution $(\UU^N_1(t),\dots,\UU^N_N(t))_{t\in[0,T]}$ to \eqref{psp}.
We also set 
$$
\QQ^{N}:=\frac1N \sum_{i=1}^N\delta_{(\UU^{N}_i(t))_{t\in [0,T]}}.
$$

(i) The family $\bigl\{\LL((\UU^{N}_1(t))_{t\in [0, T]}), N\geq 2 \bigr\}$ is tight in $\PP(C([0, T],\rr^3))$.

(ii) The family $\bigl\{\LL(\QQ^N), N\geq 2 \bigr\}$ is tight in $\PP(\PP(\rr\times C([0, T],\rr^3)))$.
\end{prop}

\begin{proof}
As is well-known, point (ii) is implied by point (i) and the exchangeability of the system, 
see Sznitman \cite[I-Proposition 2.2]{SSF}.
To prove point (i), we use the shortened notation 
$b^N_i(s)=b^N_i(\bUU^N_s)$, $\sigma^N_i(s)=\sigma^N_i(\bUU^N_s) $ and 
$c^N(s)=c_{\lfloor s/\tau_0 \rfloor}^N(\bUU^N_s)$ and we write
\begin{align*}
\UU^{N}_1(t)  & =
\VV^N_1(0) + \intot b_1^N(s)ds + \intot \sigma_1^N(s)d\BB_1(s) + \intot c^N(s)d\WW_1(s) 
+ \intot c^N(s)\nabla_{v_1} c^N(s)ds \\
& = \VV^N_1(0) + \hspace{13pt}  I^N(t) \hspace{12pt} + \hspace{23pt} J^N(t) \hspace{22pt} + \hspace{22pt} K^N(t)
 \hspace{21pt} +  \hspace{30pt} L^N(t)
\end{align*} 
and prove separately that each term is tight.

\vip

First, $\{\VV^N_1(0)\}_{N\geq 2}$ is tight by \eqref{condichaos}-(iii). 
Next, $\{(K^N(t))_{t\in[0,T]}\}_{N\geq 2}$
and $\{(L^N(t))_{t\in [0,T]}\}_{N\geq 2}$ are obviously tight (use the Kolmogorov criteria for $K^N$), because $c^N$ and $\nabla_{v_1} c^N(s)$ are 
uniformly bounded, see Proposition \ref{adnoiseprop}-(iii). 

\vip

To prove that $\{(J^N(t))_{t\in[0,T]}\}_{N\geq 2}$ is tight, we use the Kolmogorov criterion,
see e.g. Stroock-Varadhan \cite[Corollary 2.1.4]{SV}: it suffices to show that  there are some constants 
$p>0$, $\beta>1$ and $C$ such that for all $N\geq 2$, all $0\leq s \leq t\leq T$,
$$ 
\E \bigl[|J^N(t)-J^N(s)|^p \bigr] \leq C|t-s|^\beta.
$$
We consider $p=4/(2+\gamma)>2$. By the Burkholder-Davis-Gundy inequality and since
$[\sigma^N_1(s)]^2=a^N_1(s)$, we have, for $0\leq s \leq t \leq T$, 
\begin{align*}
\E \bigl[|J^N(t)-J^N(s)|^p \bigr] \leq& C \, \E\Big[ \Big(\int_s^t \|a^N_1(u)\| du \Big)^{p/2}\Big] \\
\leq &  C \,  \E\biggl[ \biggl(\int_s^t \frac 1N \sum_{j \ne 1} \|(a \star \phi_{\eta_N})(\UU^N_1(u)-\UU^N_j(u))\| 
du \biggr)^{p/2}\biggr].
\end{align*}
Recalling next that $\|(a \star \phi_{\eta_N})(v)\|\leq C(1+|v|^{\gamma+2})$ and using the H\"older inequality, 
\begin{align*}
\E \bigl[|J^N(t)-J^N(s)|^p \bigr] 
\leq &  C (t-s)^{p/2 -1}  \int_s^t \E\biggl[\Big(1+\frac1N \sum_{j \ne 1} |\UU^N_1(u)-\UU^N_j(u)|^{2+\gamma}\Big)^{p/2} 
\biggr]du\\
\leq &  C (t-s)^{p/2} \sup_{[0,T]} \E\biggl[1+\frac1N \sum_{j \ne 1} |\UU^N_1(u)-\UU^N_j(u)|^{2} \biggr]
\end{align*}
since $p/2>1$ and $(2+\gamma)p/2=2$. Using finally exchangeability and the moment bound proved in Lemma \ref{pspwp}
(together with \eqref{condichaos}), we deduce that
$\E \bigl[|J^N(t)-J^N(s)|^p \bigr] \leq C_p (t-s)^{p/2}$. Since $p/2>1$, we conclude the tightness of 
$\{(J^N(t))_{t\geq 0}\}_{N\geq 2}$.

\vip
The remaining term $I^N(t)$ is the more difficult, since it is singular (when $\gamma \in (-2,-1)$). 
For some $p>1$ to be chosen later, we write
$$
\bigl|  I^N(t) - I^N(s) \bigr|  = \biggl|  \int_s^t  b^N_1(s) \,ds  \biggr| 
\le   (t-s)^{1-1/p} \biggl( \int_s^t  |b^N_1(s)|^p \,ds \biggl)^{1/p}.
$$
Thus with $\alpha=1-1/p>0$, 
$$
\sup_{0 \le s < t \le T} \frac{\bigl|  I^N(t) - I^N(s) \bigr|}{|t-s|^\alpha}  \le 
\biggl( \int_0^T  \bigl[  b^N_1(s) \bigr]^p \,ds \biggl)^{1/p}.
$$
Taking expectations and using the H\"older inequality, we get
$$
\E \biggl[ \sup_{0 \le s < t \le T} \frac{\bigl|  I^N(t) - I^N(s) \bigr|}{|t-s|^\alpha} \biggr] \le 
\biggl(\int_0^T  \E   \Bigr[  \bigl(  b^N_1(s) \bigr)^p  \Bigr] \,ds \biggl)^{1/p}.
$$
Using the expression of $b^N_1(s) = N^{-1} \sum_{j=2}^N (b\star \phi_{\eta_N})(\UU^N_1(s) - \UU^N_j(s))$, 
the H\"older inequality and exchangeability, we find
$$
\E \biggl[ \sup_{0 \le s < t \le T} \frac{\bigl|  I^N(t) - I^N(s) \bigr|}{|t-s|^\alpha} \biggr]^p \le 
C \int_0^T  \E   \Bigr[1+  \bigl|\UU^N_1(s) - \UU^N_2(s)\bigr|^{(\gamma+1)p} \Bigr] \,ds .
$$
If $\gamma \in (-1,0)$, we choose $p = 2/(1+\gamma)$.
By the boundedness of the moment of order two obtained in Proposition~\ref{pspwp}, we get,
with $\alpha = (1-\gamma)/2$,
$$
\E \biggl[ \sup_{0 \le s < t \le T} \frac{\bigl|  I^N(t) - I^N(s) \bigr|}{|t-s|^\alpha} \biggr]^p \le  C.
$$
If $\gamma=-1$, we choose $p=2$ (or any other value) and deduce that, with $\alpha=1/2$,
$$
\E \biggl[ \sup_{0 \le s < t \le T} \frac{\bigl|  I^N(t) - I^N(s) \bigr|}{|t-s|^\alpha} \biggr]^2 \le  C,
$$
We now consider the case where $\gamma \in (-2,-1)$. Since $\gamma<2\gamma+2<0$, we deduce from
Lemma \ref{singetmom}-(ii) that $\sup_{N \geq 2} \int_0^T \E[|\UU^N_1(s) - \UU^N_2(s)\bigr|^{2(\gamma+1)}]ds <\infty$.
Consequently, we may choose $p=2$ and get, with $\alpha=1/2$,
$$
\E \biggl[ \sup_{0 \le s < t \le T} \frac{\bigl|  I^N(t) - I^N(s) \bigr|}{|t-s|^\alpha} \biggr]^2 \le  C.
$$
In any case, we conclude the tightness of the family $\{ (I^N(t))_{t \in [0,T]}\}$ by the Ascoli theorem.
\end{proof}

\subsection{Propagation of chaos for the perturbed system}

With the several estimates obtained in the previous section, we are now in position to prove the 
propagation of chaos 
for the perturbed system. The proof uses some martingale problem,
as was initiated by Sznitman \cite{SSF}.
The only difficulty in the present case is to control the singularity of $b$ (when $\gamma \in (-2,-1]$).

\begin{prop}\label{cvgcep}
Assume \eqref{condichaos} and consider, for each $N\geq 2$, the unique
solution to \eqref{psp} $(\UU^N_1(t),\dots,\UU^N_N(t))_{t\in [0,T]}$.
Let also $(\UU(t))_{t\in [0,T]}$ be the unique solution to the perturbed nonlinear SDE \eqref{nsdep}, given by Proposition~\ref{nsdepwp}.
The sequence $((\UU^N_1(t))_{t\in[0,T]},\dots,(\UU^N_N(t))_{t\in [0,T]})$ is $(\UU(t))_{t \in [0,T]}$-chaotic.
\end{prop}

\begin{proof}
We define $\cS$ as the set of all probability measures 
$g\in\PP(C([0,T],\rr^3))$ such that $g$ is the law of
$(\UU(t))_{t \in [0,T]}$ solution to the perturbed nonlinear SDE \eqref{nsdep} associated 
with $f_0$ and satisfying, for $g_t\in \PP(\rd)$ the law of $\UU(t)$,
\begin{equation} \label{tif}
(g_t)_{t \in [0,T]} \in L^1([0,T],L^p(\rd)) \quad \hbox{and}\quad (g_t)_{t \in [0,T]} \in L^\infty([0,T],\cP_2(\rd))
\end{equation}
for some $p>p_1(\gamma)$. 
The uniqueness result 
shown in Proposition~\ref{nsdepwp} implies that the set $\cS$ contains only one element.

\vip

We recall that $\QQ^N=N^{-1}\sum_{i=1}^N \delta_{(\UU^N_i(t))_{t\in[0,T]}}$ stands for the empirical distribution 
of trajectories. By Proposition~\ref{tight}, this sequence is tight, we thus can consider  a (not relabeled) 
subsequence of $\QQ^N$ going in law 
to some $\QQ$. We will show that $\QQ$ almost surely belongs to $\cS$. This will conclude the proof, since
$\cS$ contains only one element. The definition of propagation of
chaos was recalled in Definition \ref{dfc}.

\vip

{\it Step 1.} Consider the identity map $ \beta: C([0,T],\rr^3) 
\mapsto C([0,T],\rr^3)$.
Using the classical theory of martingale problems, we realize that $g$ belongs to $\cS$ as soon as

(a) $g\circ \beta_0^{-1}=f_0$;

(b) setting $g_t= g\circ \beta_t^{-1}$, (\ref{tif}) holds true; 

(c) for all $0<t_1<\dots <t_k<s< t \le T$, all
$\varphi_1,\dots,\varphi_k \in C_b(\rr^3)$, all $\varphi \in C^2_b(\rr^3)$,
\begin{multline*}
\cF(g):=\int\!\!\int g(d\beta)g(d\tg)\varphi_1(\beta_{t_1})\dots\varphi_k(\beta_{t_k})\\
 \biggl[
\varphi(\beta_t)-\varphi(\beta_s)-\int_s^t  \indiq_{\{\beta_u\neq\tg_u\}}b(\beta_u-\tg_u)\cdot \nabla \varphi(\beta_u)du 
- \frac12 \int_s^t  [a(\beta_u-\tg_u) \colon \nabla^2 \varphi(\beta_u)]du   \\
- \frac12 \int_s^t c^2_{\lfloor u/ \tau_0\rfloor} (  g_u)  \Delta\varphi(\beta_u)du
\biggr]=0. 
\end{multline*} 
Here and below, $A\colon B= \sum_{k,l=1}^3 A_{kl}B_{kl}$ for two $3\times 3$-matrices $A$ and $B$.

\vip

Indeed, let $(\UU(t))_{t\geq 0}$ be $g$-distributed. Then (a) implies that $\UU(0)$ is
$f_0$-distributed and (b) says that the requirement (\ref{tif}) is fulfilled. 
Finally, point (c) tells us that for all $\varphi \in C^2_b(\rr^2)$,
\begin{multline*}
\varphi(\UU(t))-\varphi(\UU(0))-\intot \int \indiq_{\{\UU(s)\neq\tg_s\}}b(\UU(s)-\tg_s) 
\cdot \nabla\varphi(\UU(s))\, g(d\tg)ds  \\
- \frac12 \intot  \int [a(\UU(s)-\tg_s) \colon \nabla^2 \varphi(\UU(s))] \, g(d\tg)ds    
- \frac12 \intot c^2_{\lfloor s/ \tau_0\rfloor}\bigl(g_s\bigr) \Delta\varphi(\UU(s))ds
\end{multline*}
is a martingale.
Observe that $\int \indiq_{\{\UU(s)\neq\tg_s\}}b(\UU(s)-\tg_s) g(d\tg) ds=
\intrd \indiq_{\{\UU(s)\neq x\}}b(\UU(s)-x) g_s(dx)=\intrd b(\UU(s)-x) g_s(dx)
=b(g_s,\UU(s))$.
We used here that $g_s$ does not weight points, since it has a density by point (b).
Similarly, $\int  a(\UU(s)-\tg_s)  g(d\tg)= a(g_s,\UU(s))$.
All this classically implies 
the existence of two independent $3D$-Brownian motions
$(\BB(t))_{t\geq 0}$ and $(\WW(t))_{t\geq 0}$ such that 
$$
\UU(t)=\UU(0) +\intot b(g_s,\UU(s)) ds + \intot \sigma(g_s,\UU(s)) d\BB(s) +
\intot c_{\lfloor s/ \tau_0\rfloor}(g_s) d\WW(s).
$$
Hence $(\UU(t))_{t\in[0,T]}$ solves~\eqref{nsdep} as desired.

\vip

We thus only have to prove that $\QQ$ a.s. satisfies points (a), (b) and (c). For each 
$t\in[0,T]$,
we set $\QQ_t= \QQ \circ \beta_t^{-1}$ and $\QQ_t^N= \QQ^N \circ \beta_t^{-1}=N^{-1}\sum_1^N \delta_{\UU^N_i(t)}$.

\vip

{\it Step 2.} 
We know from (\ref{condichaos}) that the sequence $G^N_0=F^N_0$ is $f_0$-chaotic, which implies that
$\QQ^N_0= \QQ^N\circ \beta_0^{-1}$ (this is nothing but the law of $N^{-1}\sum_1^N \delta_{\VV^N_i(0)}$) 
goes weakly to $f_0$ in law (and thus in probability
since $f_0$ is deterministic), whence
$\QQ_0=f_0$ a.s. Hence $\QQ$ satisfies (a).

\vip

{\it Step 3.} Point (b) follows from Lemma \ref{singetmom}, Proposition \ref{pspwp} 
and Corollary \ref{tbu}. 
First, Lemma \ref{singetmom}-(i) tells us that $\sup_N \E[\sup_{[0,T]} |\UU^N_1(s)|^2]<\infty$,
which implies that $\sup_N \E[\sup_{[0,T]} m_2(\QQ^N_t)] <\infty$ by exchangeability.
Since $\QQ$ is a weak limit of $\QQ^N$, we easily conclude that $\E[\sup_{[0,T]} m_2(\QQ_t)] <\infty$,
whence $(\QQ_t)_{t\in[0,T]} \in L^\infty([0,T],\cP_2(\rd))$ a.s.

\vip

The integrability condition is slightly more complicated. Since $\QQ^N_t$ goes in law to $\QQ_t$ 
for each $t\in [0,T]$, we may apply Corollary \ref{tbu} for each $t\in [0, T]$ and deduce that
$\E \bigl[ I_\gamma(\QQ_t) \bigr] \le \liminf_N I_\gamma(G^N_t)$. By the Fatou Lemma, this yields
$$
\int_0^T \E[I_\gamma(\QQ_s)]ds
\leq  \int_0^T \liminf_N I_\gamma(G^N_t)dt \leq \liminf_N \int_0^T I_\gamma(G^N_t)dt < \infty
$$
by Proposition~\ref{HIbound}. Next, we know from Proposition \ref{pspwp} and condition \eqref{condichaos}
that, for $\alpha=\max\{2,q\}$ (with $q> q(\gamma)$), $\sup_{N} \sup_{[0,T]} \E[|\UU^N_1(s)|^\alpha] <\infty$.
Using exchangeability, this gives $\sup_{N} \sup_{[0,T]} \E[m_\alpha(\QQ^N_t)] <\infty$,
from which we easily deduce that $\sup_{[0,T]} \E[m_\alpha(\QQ_t)] <\infty$.
Consider now $p \in (\max\{3/2,p_1(\gamma)\},\max\{(6+3|\gamma|)/(2+3|\gamma|),p_2(\gamma,q)\})$ 
(observe that $(6+3|\gamma|)/(2+3|\gamma|)>3/2$ because $|\gamma|<2$). We always have $p\in (3/2,3)$
(because $p_2(\gamma,q)<3$). Set $r=(3p-3)/(2p)$, which belongs to
$(1/2,1)$ and $a=|\gamma|r/(1-r)=|\gamma|(3p-3)/(3-p)$.
Then, applying Lemma \ref{lem:FishInteg1}, we get
$$
\|\QQ_t\|_{L^p} \leq C (I_\gamma(\QQ_t))^r (1+m_a(\QQ_t))^{1-r}
$$
whence, by the H\"older inequality,
$$
\E[\|\QQ_t\|_{L^p}] \leq C \E[I_\gamma(\QQ_t)]^r \E[1+m_a(\QQ_t)]^{1-r}.
$$
But $a \leq \alpha$. Indeed, just use that $p< [(6+3|\gamma|)/(2+3|\gamma|)]\lor p_2(\gamma,q)$.
But $p\leq (6+3|\gamma|)/(2+3|\gamma|)$ implies 
that $a \leq 2 \leq \alpha$ and $p \leq p_2(\gamma,q)=(3q+3|\gamma|)/(q+3|\gamma|)$ gives us
$a \leq q \leq \alpha$.
Consequently, recalling that $\sup_{[0,T]} \E[m_\alpha(\QQ_t)] <\infty$, we get
$$
\E[\|\QQ_t\|_{L^p}] \leq C \E[I_\gamma(\QQ_t)]^r \leq C(1+\E[I_\gamma(\QQ_t)]).
$$
Integrating in time, we deduce that $\int_0^T \E[\|\QQ_t\|_{L^p}] dt <\infty$, whence finally,
$\int_0^T \|\QQ_t\|_{L^p} dt <\infty$ a.s. We have checked that $\QQ$ satisfies point (b).

\vip

{\it Step 4.} From now on, we consider some fixed $\cF: \PP(C([0,T],\rr^3))\mapsto \rr$ 
as in point (c). We will check that $\cF(\QQ)=0$ a.s.
and this will end the proof.

\vip

{\it Step 4.1.} Here we prove that for all $N\geq 2$,
\begin{equation*}
\E [ |\cF(\QQ^N)| ] \leq C_\cF (N^{-1/2} + \eta_N).
\end{equation*}
To this end, we recall that $\varphi \in C^2_b(\rr^3)$ is fixed
and we define, for $i=1,\dots,N$,
\begin{align*}
O^N_i(t) &:=\varphi(\UU^N_i(t))
-\frac 1N \sum_{j=1}^N \intot \indiq_{\{\UU^N_i(s)\ne \UU^N_j(s)\}}b(\UU^N_i(s)-\UU^N_j(s)) \cdot \nabla\varphi(\UU^N_i(s))
 ds \\
&  - \frac1{2N} \sum_{j=1}^N \intot [a(\UU^N_i(s)-\UU^N_j(s)) \colon \nabla^2 \varphi(\UU^N_i(s))]  ds
- \frac 12 \intot [c^N_{\lfloor s/\tau_0\rfloor}( \bUU^N(s) )]^2 \Delta \varphi(\UU^N_i(s))ds.
\end{align*}
By definition of $\cF$, we have
\begin{align*}
\cF(\QQ^N)=& \frac 1 N \sum_{i=1}^N \varphi_1(\UU^N_i(t_1))\dots \varphi_k(\UU^N_i(t_k)) 
[O^N_i(t)-O^N_i(s)].
\end{align*}
Applying the It\^o formula to compute $\varphi(\UU^N_i(t))$, we realize that, for $i=1,\dots,N$,
$O^N_i(t)=\varphi(\UU^N_i(0))+M^N_i(t)+\Delta^N_i(t)$, where 
\begin{align*}
M^N_i(t)=&  \intot \nabla \varphi(\UU^N_i(s)) \cdot 
\bigl[ \sigma_i^N(\bUU^N(s)) d\BB^i(s) + c_{\lfloor s/\tau_0\rfloor}^N(\bUU^N(s)) d\WW^i(s)\bigr],\\
\Delta^N_i(t)=& \intot \nabla \varphi(\UU^N_i(s)) \cdot 
\bigl[ c_{\lfloor s/\tau_0\rfloor}^N(\bUU^N(s))\nabla_{v_i}c_{\lfloor s/\tau_0\rfloor}^N(\bUU^N(s))\\
&\hskip4cm+ \frac 1N \sum_{j=1}^N \indiq_{\{\UU^N_i(s)\ne\UU^N_j(s) \}}
(b\star \phi_{\eta_N} -b)(\UU^N_i(s)-\UU^N_j(s))\bigr] ds\\
&+ \intot 
\frac 1 {2N} \sum_{j=1}^N [(a\star \phi_{\eta_N} - a)(\UU^N_i(s)-\UU^N_j(s)) \colon \nabla^2 \varphi(\UU^N_i(s))]  ds.
\end{align*}
We used here that $(b\star \phi_{\eta_N})(x)=\indiq_{\{x\ne 0\}}(b\star \phi_{\eta_N})(x)$ since
$(b\star \phi_{\eta_N})(0)=0$ by symmetry.
Performing now some classical stochastic calculus, using that $0<t_1<\dots<t_k<s<t$, that
$\varphi_1,\dots,\varphi_k,\nabla \varphi$ and $c_l$ are uniformly bounded and
that the Brownian motions $\BB^1,\ldots,\BB^N$ and $\WW^1,\ldots \WW^N$ are independent, we easily obtain
\begin{align*}
&\E \Big[ \Big(\frac 1 N \sum_{i=1}^N \varphi_1(\UU^N_i(t_1))\dots \varphi_k(\UU^N_i(t_k))[M^N_i(t)-M^N_i(s)]\Big)^2
\Big] \\
\le& \frac {C_{\cF}}{N^2} \sum_{i=1}^N \int_s^t \E \Bigl[ 1+ \|\sigma^N_i (\bUU^N(s))\|^2 \Bigr] \,ds
\le  \frac {C_{\cF}}{N},
\end{align*}
where we have used the fact that 
$\|\sigma^N_i (\bUU^N(s))\|^2 \leq  C N^{-1} \sum_{j=1}^N (1+|\UU^N_i(s) - \UU^N_j(s)|^{2+\gamma})$,
exchangeability, and the moment estimate of Proposition~\ref{pspwp} (since $\gamma+2 \in (0,2]$).

\vip

Next, we use exchangeability and the boundedness of $\varphi_1,\dots,\varphi_k,\nabla \varphi,\nabla^2\varphi$ 
to write
\begin{align*}
\E \Big[ \Big|\frac 1 N \sum_{i=1}^N \varphi_1(\UU^N_i(t_1))\dots \varphi_k(\UU^N_i(t_k))
[\Delta^N_i(t)-\Delta^N_i(s)]\Big|\Big] \leq & C_{\cF}\E[I^N+J^N+K^N],
\end{align*}
where
\begin{align*}
I^N:=& \intot \Big|c_{\lfloor s/\tau_0\rfloor}^N(\bUU^N(s))\nabla_{v_1}c_{\lfloor s/\tau_0\rfloor}^N(\bUU^N(s))\Big| ds,\\
J^N:=&\intot\Big|(b\star \phi_{\eta_N} -b)(\UU^N_1(s)-\UU^N_2(s))\Big|ds,\\
K^N:=& \intot\Big\|(a\star \phi_{\eta_N} -a)(\UU^N_1(s)-\UU^N_2(s))\Big\|ds.
\end{align*}
Since $c^N_l \nabla_{v_1}c_l$ is bounded by $C/N$ by Proposition \ref{adnoiseprop}-(iii), we immediately
find that $\E[I^N] \leq C/N$. Next, using Lemma \ref{rough}-(ii), 
$$
\E[J^N] \leq C \eta_N \intot \E[|\UU^N_s(s)-\UU^N_2(s)|^\gamma] ds \leq C_t \eta_N
$$
by Lemma \ref{singetmom}-(ii). Finally, using again Lemma \ref{rough}-(ii), 
$$
\E[K^N] \leq C \eta_N^2 \intot \E[|\UU^N_s(s)-\UU^N_2(s)|^{\gamma}] ds \leq C_t \eta_N.
$$

{\it Step 4.2.} We introduce, for $\e\in (0,1)$, a smooth and bounded 
approximation $b_\e:\rr^3\to\rr^3$  satisfying $b_\e (x) = b(x)$ for
$|x| \ge \e$ and $|b_\e(x)|\le |b(x)|= 2|x|^{1+\gamma}$ for all $x$. This smoothing is useless if 
$\gamma \in [-1,0)$,  but we treat all the cases similarly to avoid repetitions.
We also introduce $\cF_\e$ defined as $\cF$ with $b$ replaced by $b_\e$.
The diffusion coefficient $a$ is continuous and so remains unchanged.
For each fixed $\e \in (0,1)$, for every $M>0$,
the map $g\mapsto \cF_\e(g)$ is continuous
and bounded on the set $g\in \PP(C([0,T],\rr^3))$, $ \int (\sup_{[0,T]} |\beta_s|^2) g(d\beta)\leq M$.
This is not hard to check, using that $\varphi_1,\dots,\varphi_k,\varphi,\nabla\varphi,\nabla^2\varphi$
are continuous and bounded, that $b_\e(z)$ and $a(z)$ are continuous and bounded by $C_\e(1+z^2)$, and finally that
$c_l$ is bounded (by $3$) and Lipschitz continuous for the $W_2$ topology by Proposition
\ref{adnoiseprop}-(vi).
Since $\QQ^N$ goes in law to $\QQ$ and since $\sup_N \E[\int (\sup_{[0,T]} |\beta_s|^2) \QQ^N(d\beta)]
=\sup_N \E[\sup_{[0,T]}|\UU^N_1(s)|^2] <\infty$ by exchangeability and Lemma \ref{singetmom}-(i),
we deduce that for any $\e\in(0,1)$,
$$
\E[|\cF_\e(\QQ)|]=\lim_{N} \E[|\cF_\e(\QQ^N)|].
$$

{\it Step 4.3.}  We now prove that for all $N\geq 2$, all $\e\in(0,1)$,
$$
\E \Bigl[ \bigl|\cF(\QQ^N)-\cF_\e(\QQ^N) \bigr| \Bigr] \leq C_\cF  \, \e.
$$
Using that all the functions (including the derivatives) involved in $\cF$ are bounded and that
$|b_\e(x)-b(x)|\leq |x|^{1+ \gamma}\indiq_{\{|x|<\e\}}$, we get
\begin{align}   \nonumber
\bigl|\cF(g)-\cF_\e(g) \bigr| \leq& C_\cF  \iint   \int_0^T 
\indiq_{\{0<|\beta_t -\tg_t|<\e\}} |\beta_t-\tg_t|^{1 + \gamma}
dt g(d\tg) g(d\beta) \\
\leq & C_\cF \,  \e \iint \int_0^T \indiq_{\{\beta_t \ne\tg_t\}} |\beta_t -\tg_t|^{\gamma}
dt g(d\tg) g(d\beta).  \label{tas}
\end{align}
Thus
$$
\bigl|\cF(\QQ^N)-\cF_\e(\QQ^N) \bigr| \leq C_\cF 
\frac {\e} {N^2} \sum_{i\ne j}\int_0^T \bigl|\UU^N_i(t)- \UU^N_j(t) \bigr|^\gamma dt.
$$
It suffices to take expectations, to use exchangeability and then Lemma \ref{singetmom}-(ii).

\vip
{\it Step 4.4.} We next check that a.s.,
$$
\lim_{\e \to 0} \bigl|\cF(\QQ)-\cF_\e(\QQ) \bigr| = 0.
$$
Starting from \eqref{tas}, using \eqref{weakb2} and that $(\QQ_t)_{t\in[0,T]} \in L^1([0,T],L^p(\rd))$ a.s.
for some $p>p_1(\gamma)$ by Step 3, we get
\begin{align*}
\bigl|\cF(\QQ)-\cF_\e(\QQ) \bigr| \leq& C_\cF  \, \e \int_0^T \intrd\intrd |x-y|^{\gamma} \QQ_t(dx)\QQ_t(dy)
\leq C_\cF \,\e \int_0^T (1+\|\QQ_t\|_{L^p})dt,
\end{align*}
whence the conclusion.

\vip

{\it Step 4.5.} We finally conclude: for any $\e \in (0,1)$, we write, using Steps 4.1, 4.2 and 4.3,
\begin{align*}
\E \bigl[|\cF(\QQ)|\land 1 \bigr] \leq & \E \bigl[|\cF_\e(\QQ)| \bigr] + \E \bigl[|\cF(\QQ)-\cF_\e(\QQ)|\land 1 
\bigr] \\
= & \lim_{N \to +\infty}  \E\bigl [|\cF_\e(\QQ^N)| \bigr] + \E \bigl[|\cF(\QQ)-\cF_\e(\QQ)|\land 1 \bigr] \\
\leq & \limsup_{N \to +\infty} \E \bigl[|\cF(\QQ^N)| \bigr] + \limsup_{N \to +\infty} \E\bigl[|\cF(\QQ^N)-\cF_\e(\QQ^N)| 
\bigr] 
+ \E \bigl[|\cF(\QQ)-\cF_\e(\QQ)|\land 1 \bigr] \\
\leq & C_\cF  \, \e +  \E \bigl[|\cF(\QQ)-\cF_\e(\QQ)|\land 1 \bigr].
\end{align*}
We now make tend $\e\to 0$ and use that $\lim_\e \E[|\cF(\QQ)-\cF_\e(\QQ)|\land 1]=0$
thanks to Step 4.4 by dominated convergence. Consequently, 
$\E[|\cF(\QQ)|\land 1]=0$, whence $\cF(\QQ)=0$ a.s. as desired.
\end{proof}

\subsection{Asymptotic annihilation of the perturbation}

We now show the asymptotic equivalence of the perturbed and unperturbed particle systems.
This uses the propagation of chaos for the perturbed system.

\begin{prop}\label{pegalpp}
Recall that $\gamma \in (-2,0)$, assume \eqref{condichaos} and  consider, for each $N\geq 2$, the 
solution $(\bVV^N(t))_{t\geq 0}$ to \eqref{ps} and the solution 
$(\bUU^N(t))_{t\in[0,T]}$ to \eqref{psp},
with the same initial conditions $\UU^N_i(0)=\VV^N_i(0)$, $i=1,\dots N$ and same Brownian motions 
$\BB^i$, $i=1,\dots N$. Then
$$
\lim_{N} \Pr \Big[(\bVV^N(t))_{t\in[0,T]}=(\bUU^N(t))_{t\in[0,T]}  \Big] = 1.
$$
\end{prop}

\begin{proof}
By the strong uniqueness for the particle system \eqref{ps}, see Proposition \ref{wpps},
we see that $(\bVV^N(t))_{t\in[0,T]}=(\bUU^N(t))_{t\in[0,T]}$ as soon as 
$c^N_{\lfloor t/\tau_0\rfloor}(\bUU^N(t)) = 0$ for all $t\in[0,T]$.
Recalling the Definition of $c_l$, see Notation \ref{ddd1}, it thus suffices to prove that 
\begin{equation}\label{obbb}
\lim_{N} \Pr \Big[  \forall \; l=0,\dots,n_0,\;\forall \, t \in [l\tau_0,(l+1)\tau_0], 
\; c^N_{l}(\bUU^N(t)) = 0   \Big] = 1.
\end{equation}
We denote by $d : \PP(C([0, T],\rr^3))  \mapsto \rr^+$ the function defined by
$$
d(g) := \inf_{l=0,\dots,n_0} \; \inf_{t \in [l \tau_0, (l+1) \tau_0]} \;  \inf_{k=1,2,3}  
\intrd h \biggl(\frac{v-x_k^{l}}{\delta_0} \biggr)g_t(dv),
$$ 
where $h : \rr^3 \to [0,1]$ (a smooth function satisfying $\indiq_{\{|v| \le 1\}} \le h \le \indiq_{\{|v| \le 2\}}$), 
the $x_k^l$'s and $\delta_0$ have been introduced in Notation \ref{ddd1}.
Remark first that $d$ is continuous (and bounded) with respect to the weak topology of measure on 
$\PP(C([0, T],\rr^3))$. Consequently, we know from Proposition \ref{cvgcep} that $d(\QQ^N)$
goes in law to $d(g)$, where $\QQ^N=N^{-1}\sum_1^N \delta_{(\UU^N_i(t))_{t\in[0,T]}}$ and where $g =\LL((\UU(t))_{t\in[0,T]})$,
where $(\UU(t))_{t\in[0,T]}$ is the unique solution to the perturbed nonlinear SDE \eqref{nsdep}.
But we also know, from Proposition \ref{nsdepwp} that $g_t=f_t$ for all $t\in[0,T]$, where
$f$ is the unique weak solution to the Landau equation \eqref{HL3D}. We finally recall \eqref{kappa-inf}
$$
\inf_{l=0,\dots,n_0} \; \inf_{t \in [l \tau_0, (l+1) \tau_0]} \; \inf_{k=1,2,3} f_t(B(x_k^l,\delta_0)) \ge \kappa_0,
$$
which implies that $d(g)\geq \kappa_0$. Consequently, it holds that
$$
\lim_{N \rightarrow + \infty} \Pr \Big[  d(\QQ^N) \ge \frac {\kappa_0}2   \Big] = 1,
$$
whence
$$
\lim_{N \rightarrow + \infty} \Pr \Big[  \inf_{l=0,\dots,n_0} \; \inf_{t \in [l \tau_0, (l+1) \tau_0]} \;  \inf_{k=1,2,3} 
\QQ^N_t(B(x_k^{l},2\delta_0)) \ge \frac {\kappa_0}2   \Big] = 1.
$$
Using Proposition \ref{adnoiseprop}-(i) (which implies that $c_l(\mu)=0$ as soon as
$\inf_{k=1,2,3} \mu(B(x_k^l,2\delta_0)) \ge \kappa_0/2$) and that 
$c_{l}^N(\bUU^N_t)=c_l(\QQ^N_t)$ by definition (see Notation \ref{ddd1} again), we conclude that indeed,
\eqref{obbb} holds true.
\end{proof}

\subsection{Conclusion}

We now have all the weapons in hand to give the

\begin{proof}[Proof of Theorem \ref{mr}] 
We know from Proposition \ref{cvgcep} that the perturbed system $(\bm \UU^N(t))_{t\in[0,T]}$ is
$(\UU(t))_{t\in[0,T]}$-chaotic, where $(\UU(t))_{t\in[0,T]}$ is the unique solution to the perturbed
nonlinear SDE \eqref{nsdep}. But Proposition \ref{nsdepwp} 
tells us that  $(\UU(t))_{t\in[0,T]}= (\VV(t))_{t\in[0,T]}$,
where $(\VV(t))_{t\in[0,T]}$ is the unique solution to the (non perturbed) nonlinear SDE \eqref{nsdep},
while Proposition \ref{pegalpp} tells us that $\lim_N \Pr((\bm \VV^N(t))_{t\in[0,T]}=(\bm \UU^N(t))_{t\in[0,T]})=1$.
We immediately conclude that $(\bm \VV^N(t))_{t\in[0,T]}$ is $(\VV(t))_{t\in[0,T]}$-chaotic.
Recalling that $T>0$, which has been fixed at the beginning of the section, can be chosen arbitrarily large,
we deduce that the sequence  $(\bm \VV^N(t))_{t\geq 0}$ is $(\VV(t))_{t\geq 0}$-chaotic.
As already mentioned, see Definition \ref{dfc}, this implies $\QQ^N:=N^{-1}\sum_1^N \delta_{(\VV^N_i(t))_{t\geq 0}}$
goes in probability to $\LL((\VV(t))_{t\geq 0})$ as $N\to \infty$ in $\cP(C([0,\infty),\rd)))$.
But $\LL(V(t))=f_t$ for all $t\geq 0$. It is not hard to conclude, that
for $\mu^N_t:= N^{-1}\sum_1^N \delta_{\VV^N_i(t)}$, the sequence $(\mu^N_t)_{t\geq 0}$ goes in probability
to $(f_t)_{t\geq 0}$ in $C([0,\infty),\cP(\rd) )$.
\end{proof}

%
%
%
%
%
%
%
%

\appendix

\section{A coupling result}\label{sec:a0}
\renewcommand{\theequation}{\Alph{section}.\arabic{equation}}
\setcounter{equation}{0}

We reformulate and extend a result found in \cite[Proposition 2.4]{hm} for the distance $W_1$.
Here $|\,\cdot\,|$ is any fixed norm in $\rr^d$.

\begin{prop}\label{coures}
Let $p>0$ and $d\geq 1$.
For $N\geq 2$, let $F_N$ and $G_N$ be two symmetric probability measures in $\cP_p((\rr^d)^N)$.
There exists $(X_1,\dots,X_N)$ with law $F^N$ and $(Y_1,\dots,Y_N)$ with law $G^N$ enjoying
the following properties.

\vip

(a) The coupling is optimal in the sense that $\E[\sum_1^N |X_i-Y_i|^p]=\inf\{\E[\sum_1^N |U_i-V_i|^p]\}$,
the infimum being taken over all random
vectors $(U_1,\dots,U_N)$ and $(V_1,\dots,V_N)$ with laws $F^N$ and  $G^N$.

\vip

(b) The family $\{(X_i,Y_i)$, $i=1,\dots,N\}$ is exchangeable.

\vip

(c) Almost surely, $W_p^p(N^{-1}\sum_1^N \delta_{X_i},N^{-1}\sum_1^N \delta_{Y_i})=N^{-1} \sum_1^N |X_i-Y_i|^p$.
\end{prop}

\begin{proof} 

We only sketch the proof, since it is very similar to \cite[Proposition 2.4]{hm}.

\vip

We start with a coupling $\tX=(\tX_1,\dots,\tX_N)$, $\tY=(\tY_1,\dots,\tY_N)$ of $F^N$ and $G^N$
satisfying only point (a).
Such an optimal coupling is well-known to exist, see e.g. Villani \cite{v:t}.

\vip

Next, we consider a (uniform) random $\sigma \in \mathfrak{S}_N$, the set of permutations of $\{1,\dots,N\}$,
independent of $\tX, \tY$, and we put 
$\bX_i=\tX_{\sigma(i)}$ and $\bY_i=\tY_{\sigma(i)}$. 
It is straightforward to check that $\bX=(\bX_1,\dots,\bX_N)$, $\bY=(\bY_1,\dots,\bY_N)$ is still a coupling
between $F^N$ and $G^N$ (because these distributions are symmetric), 
still satisfies (a), and now satisfies (b).

\vip

We finally introduce, for $x=(x_1,\dots,x_N)$ and $y=(y_1,\dots,y_N)$ in $(\rr^d)^N$,
$$
S_{x,y}=\Big\{\tau \in \mathfrak{S}_N
\; : \; W_p^p\Big(N^{-1}\sum_1^N \delta_{x_i},N^{-1}\sum_1^N \delta_{y_i}\Big)=N^{-1} \sum_1^N |x_i-y_{\tau(i)}|^p
\Big\}.
$$
Conditionally on $\bX$ and $\bY$, we consider a random permutation $\tau$ uniformly 
chosen in $S_{\bX,\bY}$ and we set $X_i=\bX_i$ and $Y_i=\bY_{\tau(i)}$.
It remains to prove that $X=(X_1,\dots,X_N)$, $Y=(Y_1,\dots,Y_N)$ is a coupling between $F^N$ and
$G^N$ satisfying points (a), (b) and (c).
Point (c) is satisfied because $\tau \in S_{\bX,\bY}$ a.s.
By definition of $\tau$, we see that $\sum_1^N |X_i-Y_i|^p \leq \sum_1^N |\bX_i-\bY_i|^p$ a.s.,
so that (a) is satisfied. And of course, $X$ is $F^N$-distributed, since $X=\bX$. 

\vip

To check point (b), let $\sigma \in \mathfrak{S}_N$ be fixed. 
For $x\in(\rr^d)^N$, we introduce $x_\sigma=(x_{\sigma(1)},\dots,x_{\sigma(N)})$.
We observe that for any $x,y\in(\rr^d)^N$, it holds that
$S_{x_\sigma, y_\sigma} = \sigma^{-1} S_{x,y} \sigma$. Thus conditionally on $\bX$ and $\bY$,
$\sigma^{-1} \circ \tau \circ \sigma$ is uniformly distributed on $S_{\bX_\sigma, \bY_\sigma}$.
Hence by exchangeability of $(\bX,\bY)$, the triple $(\bX,\bY,\tau)$ has the same law as the triple 
$(\bX_{\sigma},\bY_{\sigma},\sigma^{-1} \circ \tau \circ \sigma)$. Thus 
$(\bX,\bY_\tau) \stackrel d = (\bX_{\sigma},\bY_{\tau\circ \sigma})$.
In other words, $(X,Y)$ has the same law as $(X_\sigma,Y_\sigma)$.

\vip

We finally check that  $Y$ is $G^N$-distributed. Consider a bounded measurable
$\varphi: (\rd)^N \mapsto \rr$ and its symmetrization $\tilde\varphi(y_1,\dots,y_N)=(N !)^{-1}
\sum_{\sigma \in \mathfrak{S}_N} \varphi(y_{\sigma(1)},\dots,y_{\sigma(n)})$. Using the exchangeability of 
$(Y_1,\dots,Y_N)$, we can write $\E[\varphi(Y_1,\dots,Y_N)]=\E[\tilde \varphi(Y_1,\dots,Y_N)]$.
But $\tilde \varphi$ being symmetric, we a.s. have that $\tilde \varphi(Y_1,\dots,Y_N)=
\tilde \varphi(\hat Y_1,\dots,\hat Y_N)$, whence 
$\E[\varphi(Y_1,\dots,Y_N)]=\E[\tilde \varphi(\hat Y_1,\dots,\hat Y_N)]$. Using finally the exchangeability
of $(\hat Y_1,\dots,\hat Y_N)$, we conclude that 
$\E[\varphi(Y_1,\dots,Y_N)]=\E[\varphi(\hat Y_1,\dots,\hat Y_N)]$. This ends the proof.
\end{proof}

\section{Solutions to SDEs associated to weak solutions of PDEs.} \label{sec:a1}

\renewcommand{\theequation}{\Alph{section}.\arabic{equation}}
\setcounter{equation}{0}

Here we extend a result of Figalli~\cite[Theorem 2.6]{Fig}.

\begin{prop}\label{pfi}
Let $\alpha:[0,\infty)\times \rd\mapsto \cM_{3\times3}(\rr)$ and $\beta:[0,\infty)\times\rd\mapsto \rd$
be measurable and satisfy $||\alpha(t,x)||\leq \rho(t)(1+|x|^2)$ and $|\beta(t,x)|\leq \rho(t)(1+|x|)$
for some $\rho \in L^1_{loc}([0,\infty))$.
Consider $(\mu_t)_{t\geq 0} \in L^\infty_{loc}([0,\infty),\cP_2(\rd))$,
weak solution to
\begin{equation}\label{edpgen}
\partial_t \mu_t = -\sum_{k=1}^3 \partial_{x_k}(\beta(t) \mu_t) 
+ \frac 12 \sum_{k,l=1}^3 \partial_{x_kx_l}(\alpha_{kl}(t)\mu_t).
\end{equation}
There exists, on some probability space, a $\mu_0$-distributed random variable $X_0$, independent 
of a $d$-dimensional Brownian motion
$(B_t)_{t\geq 0}$, and a solution $(X_t)_{t\geq 0}$ to
\begin{equation} \label{SDEgen}
X_t=X_0+ \intot \beta(s,X_s)ds + \intot (\alpha(s,X_s))^{1/2} d B_s
\end{equation}
which furthermore satisfies that $\LL(X_t)=\mu_t$ for all $t\geq 0$.
\end{prop}

We follow the proof of \cite{Fig}, which concerns the case where $\alpha$ and $\beta$ are bounded.

\begin{proof}
It suffices to prove the result when $\rho\equiv 1$. Indeed, 
consider, in the general case, the time change $h(t)=\intot (1+\rho(s))ds$, its inverse function 
$g(t)=h^{-1}(t)$, and set $\bar \mu_t:= \mu_{g(t)}$.
Then $(\bar \mu_t)_{t\geq 0}$ still belongs to $L^\infty_{loc}([0,\infty),\cP_2(\rd))$
and solves \eqref{edpgen} with $\alpha$ and $\beta$ replaced by 
$$
\bar \alpha(t,x)=(1+\rho(g(t)))^{-1}\alpha(g(t),x) \quad \hbox{and} \quad
\bar \beta(t,x)=(1+\rho(g(t)))^{-1}\beta(g(t),x).
$$
These functions satisfy $||\bar \alpha(t,x)||\leq 1+|x|^2$ and $|\bar \beta(t,x)|\leq 1+|x|$.
Thus if we have proved the result when $\rho\equiv 1$, we can find a solution $(\bar X_t)_{t\geq 0}$ to
$\bar X_t=\bar X_0 + \intot \bar \beta(s,\bar X_s)ds + \intot (\bar \alpha(s,\bar X_s))^{1/2} d B_s$
and such that $\cL(\bar X_t)=\bar \mu_t$. It is then not hard to see that $X_t:=\bar X_{h(t)}$ satisfies
$\cL(X_t)=\mu_t$ and solves 
$X_t= X_0 + \intot \beta(s,X_s)ds + \intot (\alpha(s,X_s))^{1/2} d W_s$, where
$W_t=\int_0^{h(t)}(1+\rho(g(s)))^{-1/2} dB_s$ is still a Brownian motion.

\vip

From now on, we thus assume that $||\alpha(t,x)||\leq 1+|x|^2$ and $|\beta(t,x)|\leq 1+|x|$ and we consider
$(\mu_t)_{t\geq 0}$ as in the statement. We divide the proof in several steps.

\vip

{\sl Step 1.}
We introduce $\phi_t(x) = (2 \pi t)^{-d/2} \, e^{- |x|^2/(2t)}$ and 
$\mu_t^\eps := \mu_t \star \phi_{\eps (1+t)}$. For each $t\geq 0$, $\mu_t^\e$ is a positive smooth function.
Then $(\mu_t^\e)_{t\geq 0}$ solves \eqref{edpgen} with $\alpha$ and $\beta$
replaced by
\begin{align}
\alpha^\eps(t) := \frac{\bigl( \alpha(t) \mu_t  \bigr) \star \phi_{\eps (1+t)}}{\mu^\eps_t} + \eps \, Id
\quad \hbox{and}\quad
\beta^\eps(t) := \frac{\bigl( \beta(t) \mu_t  \bigr) \star \phi_{\eps (1+t)}}{\mu^\eps_t}. \label{eq:FPeps} 
\end{align}
We now check that there is a constant $C$ (not depending on $\e \in (0,1)$) such that
\begin{align*} 
\bigl| \beta^\eps(t,x) \bigr|  \le C \bigl(1+ \sqrt{m_2(\mu_t)} + |x| \bigr) \quad\hbox{and}\quad
\bigl\|  \alpha^\eps(t,x) \bigr\| \le C  \bigl(1+ m_2(\mu_t) + |x|^2 \bigr).
\end{align*}
First, since $\beta(t,x)\leq 1+|x|$,
$$
\bigl| \beta^\eps(t,x) \bigr| \leq  \frac{\int_\rd |\beta(t,y)| \phi_{\eps (1+t)}(x-y)  \mu_t(dy)}
{\int_\rd  \phi_{\eps (1+t)}(x-y)  \mu_t(dy)}
\le 1  +  \frac{\int_\rd |y| \phi_{\eps (1+t)}(x-y)  \mu_t(dy)}{\int_\rd  \phi_{\eps (1+t)}(x-y)  
\mu_t(dy)}.
$$
We next introduce $R := \sqrt{2 m_2(\mu_t)}$, for which $\mu_t \bigl(B(0,R)\bigr) \ge 1/2$,
and write
$$
\bigl| \beta^\eps(t,x) \bigr| \leq   1 + 2  |x| + R  + 
\frac{\int_{|y-x| \ge  |x| + R} |y| \phi_{\eps (1+t)}(x-y)\mu_t(dy)}{\int_{|y-x| \le |x| + R}  \phi_{\eps (1+t)}(x-y)  \mu_t(dy)}.
$$
But $\phi_\eps$ is radial and decreasing, so that
$$
\int_{|y-x| \ge  |x| + R} |y| \phi_{\eps (1+t)}(x-y)\mu_t(dy) \leq  \phi_{\eps (1+t)}(|x| +R)  \sqrt{m_2(\mu_t)}
$$
and (observe that $B(0,R)\subset B(x,|x|+R)$)
$$
\int_{|y-x| \le |x| + R}  \phi_{\eps (1+t)}(x-y)  \mu_t(dy) \geq \phi_{\eps (1+t)}(|x| +R) \mu_t(B(x,|x|+R))
\geq \frac{\phi_{\eps (1+t)}(|x| +R)}2.
$$
This gives
$$
\bigl| \beta^\eps(t,x) \bigr| \leq  1 + 2  |x| + R + 2\sqrt{m_2(\mu_t)} \leq 1+2|x| + 4 \sqrt{m_2(\mu_t)}
$$
as desired. The very same arguments show that
$$
\bigl\| \alpha^\eps(t,x) \bigr\| \leq  1 + (2  |x| + R)^2 + 2m_2(\mu_t) \leq 1+8|x|^2 + 6 m_2(\mu_t).
$$
Finally, it is easy to check that for all $\e \in(0,1)$, all $T>0$, all $R>0$,
$$
\sup_{t \in [0,T]} \sup_{|x| \leq R} \big(|D \alpha^\e(t,x)| + |D \beta^\e(t,x)|  \big)<\infty.
$$

{\it Step 2.} The coefficient $\beta^\e(t,.)$ is locally Lipschitz continuous 
(locally uniformly in time) and has 
at most linear growth, while $\alpha^\e(t,.)$ is locally Lipschitz continuous (locally uniformly in time)
and has at most quadratic growth. Since furthermore $\alpha^\e(t,.)$ is uniformly elliptic (for $\e$ fixed),
we can apply Lemma~\ref{math} and obtain that $(\alpha^\eps(t,.))^{1/2}$ is also locally Lipschitz.
It is thus classical that for a given initial condition 
$X_0^\eps$ with law $\mu_0^\eps$ and a given Brownian motion $(B_t)_{t \ge 0}$, there exists a pathwise 
unique solution to
\begin{equation} \label{eq:Xeps}
X_t^\eps=X_0^\eps+ \intot \beta^\eps(s,X_s^\eps)ds + \intot \bigl(\alpha^\eps(s,X_s^\eps)\bigr)^{1/2} d B_s,
\end{equation}
which furthermore satisfies $\LL (X^\eps_t) = \mu^\eps_t$, by uniqueness of the weak solution
to the PDE \eqref{edpgen} with $\alpha$ and $\beta$ replaced by $\alpha^\e$ and $\beta^\e$.

\vip

{\it Step 3.} Here we check that the family $\{(X^\e_t)_{t\geq 0},\; \e>0\}$ is tight.
Since $\mu_0$ has only a moment of order two, we cannot directly apply the Kolmogorov criterion
and have to use another approximation procedure.
For $R>0$, we consider $(X^{R,\e}_t)_{t\geq 0}$ the pathwise unique solution
to $X_t^{R,\e}=X_0^\e\indiq_{\{|X_0^\e|\leq R\}}
+ \intot \beta^\eps(s,X_s^{R,\eps})ds + \intot \bigl(\alpha^\eps(s,X_s^{R,\eps})\bigr)^{1/2} d B_s$.

\vip

Recall that $\sup_{s\in[0,T],\e\in(0,1)} (|\beta^\e(s,x)|^2+||\alpha^\e(s,x)||) \leq C_T(1+|x|^2)$ by Step 1.
It is thus completely standard to check, using the Kolmogorov criterion, that for each $R>0$, the family
$\{(X^{R,\e}_t)_{t\geq 0},\; \e>0\}$ is tight: for all $A>0$, we can find a compact subset $\cK(R,A)$
of $C([0,\infty),\rd)$ such that $\sup_{\e\in(0,1)} \Pr((X^{R,\e}_t)_{t\geq 0} \notin \cK(R,A)) \leq 1/A$.

\vip

But by pathwise uniqueness, we have that $(X^\e_t)_{t\geq 0}=(X_t^{R,\e})_{t\geq 0}$ on the event $\{|X_0^\e|\leq R\}$.
And it holds that $\Pr(|X_0^\e|> R) \leq m_2(\mu_0^\e)/R^2 \leq (1+m_2(\mu_0))/R^2$. Consequently,
$$
\sup_{\e\in(0,1)} \Pr((X^{\e}_t)_{t\geq 0} \notin \cK(R,A)) \leq \frac 1A + \frac{1+m_2(\mu_0)}{R^{2}}.
$$
Choosing $\cK_B=\cK(\sqrt{2B(1+m_2(\mu_0))},2B)$, we conclude that
$\sup_{\e\in(0,1)} \Pr((X^{\e}_t)_{t\geq 0} \notin \cK_B) \leq 1/B$,
which ends the step.

\vip

{\it Step 4.} We thus can find a sequence $\e_n \searrow 0$ such that 
$(X^{\e_n}_t)_{t\geq 0}$ goes in law (for the uniform topology on compact time intervals) to some $(X_t)_{t\geq 0}$.
For all $t\geq 0$, we have $\cL(X_t)=\mu_t$ because $\cL(X^\en_t)=\mu^\en_t\to \mu_t$. 
It thus only remains to prove that $(X_t)_{t\geq 0}$ solves \eqref{SDEgen}.
By the theory of martingale problems, it suffices to prove that for any 
$0<s_1<...<s_k<s<t$, any $\varphi_1,\dots,\varphi_k \in C_b(\rd)$, any $\varphi \in C^2_c(\rd)$,
we have $\E[\cF(X)]=0$, where $\cF:C([0,\infty),\rd))\mapsto \rr$ is defined by
$$
\cF(\gamma)=\varphi_1(\gamma_{s_1})\dots\varphi_k(\gamma_{s_k}) \Big(\varphi(\gamma_t)-\varphi(\gamma_s)-
\int_s^t \Big[\nabla \varphi (\gamma_u) \cdot \beta(u,\gamma_u) + \frac12\sum_{k,l=1}^3 
\partial_{x_k x_l}\varphi(\gamma_u)\alpha_{kl}(u,\gamma_u) \Big]du\Big).
$$
We also introduce some continuous functions $\tilde \alpha:[0,\infty)\times \rd\mapsto \cM_{3\times3}(\rr)$ and 
$\tilde \beta:[0,\infty)\times\rd\mapsto \rd$ (that will be chosen later very close to $\alpha$ and $\beta$).
We define $\tilde \alpha^\eps$ and $\tilde \beta^\eps$ defined exactly as in~\eqref{eq:FPeps},
but with $\tilde \alpha$ and $\tilde \beta$. 
We finally define $\tilde \cF$ (resp. $\cF_\e$, resp. $\tilde \cF_\e$) as $\cF$ but with $\tilde \alpha$ and 
$\tilde \beta$ (resp. $\alpha^\e$ and $\beta^\e$, resp.  $\tilde\alpha^\e$ and $\tilde\beta^\e$) 
instead of $\alpha$ and $\beta$. 

\vip

We of course start from the identity $\E[\cF_{\e_n}(X^{\e_n})]=0$, and write
\begin{align*}
|\E[\cF(X)]|\leq& |\E[\cF(X)]-\E[\tilde \cF(X)]| +  |\E[\tilde \cF(X)]-\E[\tilde \cF(X^{\e_n})]|\\
& + |\E[\tilde \cF(X^{\e_n})]-\E[\tilde \cF_{\e_n}(X^{\e_n})]|+ |\E[\tilde \cF_{\en}(X^{\e_n})]-\E[\cF_{\e_n}(X^{\e_n})]|.
\end{align*}
First, since $\tilde \alpha$ and $\tilde \beta$ are continuous, it is not hard to deduce that
$\tilde \cF :C([0,\infty),\rd))\mapsto \rr$ is continuous (for the topology of uniform convergence
on compact time intervals) and bounded (recall that $\varphi$ is compactly supported).
Consequently, $\lim_n |\E[\tilde \cF(X)]-\E[\tilde \cF(X^{\e_n})]|=0$. 
Next, there is a constant $C$ such that
\begin{align*}
|\E[\cF(X)]-\E[\tilde \cF(X)]| \leq& C \intot 
\E\big[||\alpha(u,X_u)-\tilde\alpha(u,X_u)|| +|\beta(u,X_u)-\tilde\beta(u,X_u)|\big] du \\
=& C \intot \intrd \big[||\alpha(u,x)-\tilde\alpha(u,x)|| 
+ |\beta(u,x)-\tilde\beta(u,x)|\big] \mu_u(dx)du.
\end{align*}
Similarly,
\begin{align*}
&|\E[\tilde \cF_{\e_n}(X^{\e_n})]-\E[\cF_{\e_n}(X^{\e_n})]|  \\
\leq& C \intot \intrd \big(||\alpha^\en(u,x)-\tilde\alpha^\en(u,x)|| 
+|\beta^{\en}(u,x)-\tilde\beta^\en(u,x)|\big)
\mu_u^\en(dx)du\\
\leq & C \intot \intrd \Big(\frac{[(||\alpha(u) - \tilde\alpha(u)||\mu_u+|\beta(u) - \tilde\beta(u)|\mu_u)\star 
\phi_{\e_n(1+u)}](x)}{\mu^\en_u(x)} + \e_n\Big) \mu_u^\en(dx)du.\\
\leq & C \intot \intrd \intrd (||\alpha(u,x) - \tilde\alpha(u,x)||+|\beta(u,x) - \tilde\beta(u,x)|)
\phi_{\e_n(1+u)}(y-x) \mu_u(dx) dy du + C t \e_n.
\end{align*}
Using that $\phi_{\e_n(1+u)}$ has mass $1$, we conclude that
\begin{align*}
\limsup_n|\E[\tilde \cF_{\e_n}(X^\en)]-\E[\cF_\en(X^\en)]| \leq&
C\!\! \intot \!\!\intrd \!\!\big[||\alpha(u,x)-\tilde\alpha(u,x)||+|\beta(u,x)-\tilde\beta(u,x)|\big] \mu_u(dx)du.
\end{align*}
Finally,
\begin{align*}
&|\E[\tilde \cF(X^{\e_n})]-\E[\tilde \cF_{\e_n}(X^{\e_n})]|  \\
\leq& C \intot \intrd \big(||\tilde \alpha(u,x)-\tilde\alpha^\en(u,x)|| 
+ |\tilde \beta(u,x)-\tilde\beta^\en(u,x)|\big) \mu_u^\en(dx)du\\
\leq& C \intot \intrd \big(||\tilde \alpha(u,x)-\tilde\alpha(u,y)||
+ |\tilde \beta(u,x)-\tilde\beta(u,y)|\big) \phi_{\en(1+u)}(y-x)dy \mu_u(dx)du +C\e_n t.
\end{align*}
Since $\tilde \alpha$ and $\tilde \beta$ are continuous, this clearly tends to $0$ as $n\to \infty$.
All in all, we have checked that 
$$
|\E[\cF(X)]| \leq C \intot \intrd \big[||\alpha(u,x)-\tilde\alpha(u,x)|| 
+ |\beta(u,x)-\tilde\beta(u,x)|\big] \mu_u(dx)du.
$$
But this holds true for any choice of continuous $\tilde \alpha$ and $\tilde \beta$.
And since $\mu_s(dx) ds$ is a radon measure, we can find $\tilde \alpha$ and $\tilde \beta$ continuous and
arbitrarily close to $\alpha$ and $\beta$ in $L^1([0,T]\times\rd, \mu_s(dx) ds)$.
We conclude that $|\E[\cF(X)]|=0$ as desired.
\end{proof}

\section{Entropy and Fisher information} \label{sec:prel}

\renewcommand{\theequation}{\Alph{section}.\arabic{equation}}
\setcounter{equation}{0}

In this section, we present a series of results involving the Boltzmann
entropy $H$ and some fractional (or weighted) Fisher information $I^r$ and $I_\gamma$.
They provide key estimates in order to
exploit the regularity of the objects we deal with.

\subsection{Notation}
For $F \in \PP((\R^3)^N)$ with a density (and a finite moment of positive
order for the entropy), the Boltzmann entropy $H$, the weighted Fisher
information $I_\gamma$ (for $\gamma<0$) and the fractional Fisher
information $I^r$ (for $r \in (1/2,1))$ of $F$ are defined as
\begin{align*} 
H(F) & := \frac1N\int_{(\rr^3)^N} F(\bm v^N)\log F(\bm v^N) \, d\bm v^N,\\
I_\gamma(F)&:= \frac1N\int_{(\rr^3)^N} \frac{| \nabla_\gamma F(\bm v^N)|^2}{F(\bm v^N)} \,
d \bm v^N = \frac1N \int_{(\rr^3)^N} |\nabla_\gamma \log F(\bm v^N) |^2 \,F(d\bm v^N),\\
I^r(F^N) & := \frac1N  \int_{(\rr^3)^N} \frac{|\nabla F(\bm v^N)|^{2r}_{2r}}{F(\bm v^N)^{2r-1}}\,d\bm v^N 
= \frac1N \int_{(\rr^3)^N} |\nabla \log F(\bm v^N) |^{2r}_{2r} \,F(d\bm v^N),
\end{align*}
with the notation $\bm v^N=(v_1,\dots,v_N)$, 
where the differential operator $\nabla_\gamma$ is the {\it weighted} gradient
$$
\nabla_\gamma F(\bm v^N) := \bigl(  \langle v_1 \rangle^{\gamma / 2} \nabla_{v_1} F(\bm v^N), \ldots, 
 \langle v_n \rangle^{\gamma /2} \nabla_{v_n} F(\bm v^N)  \bigr),
$$ 
where $\langle v \rangle = (1+|v|^2)^{1/2}$
and where the norm $|\cdot |_{2r}$ is the $\ell^{2r}$ norm on $(\R^3)^N$ defined by
$$
| \bm v^N |_{2r}^{2r} := \sum_{i=1}^N |v_i|^{2r}.
$$
If $F \in \PP((\R^3)^N)$ has no density, we simply put $H(F) = +\infty$. If $F \in \PP((\R^3)^N)$ has no density
or if its density has no gradient, we put $I_\gamma(F)= I^r(F) = +\infty$. 
The somewhat unusual normalization by $1/N$ is made in order that for any $f \in \PP(\rr^3)$,
$$
H(f^{\otimes N}) = H(f),\quad 
I_\gamma(f^{\otimes N}) = I_\gamma(f)
\quad \text{and}\quad
I^r(f^{\otimes N}) = I^r(f)
$$
Recall finally that the moment of order $q$ is defined, for any $F \in \PP_{sym}\bigl((\rr^{3})^N\bigr)$, by
$$
m_q(F) := \int_{(\rr^3)^N} |v_1|^q  F(d \bm v^N).
$$

\subsection{First properties and estimates}

We make use of the two following properties.

\begin{lem}\label{ieth}

(i) For any $q,\lambda \in (0,\infty)$, there is a constant $C_{q,\lambda} \in \R$ such that 
for any $N\geq 1$, any $F \in \PP_q\bigl((\rr^{3})^N\bigr)$
$$
H (F) \geq  - C_{q,\lambda} - \lambda \, m_q(F).
$$
(ii) Consider the constant $C=1+|C_{2,1}|$ (with $C_{2,1}$ introduced in (i)). For any
$f \in \mathcal P_2(\rr^3)$, any measurable $A\subset \rd$ with Lebesgue measure $|A| < 1$,
\begin{equation*}
f(A) \le \frac{C + H(f) +m_2(f)}{- \log |A|}.
\end{equation*}
In particular, $|A| \le \exp(-4(C+H(f)+m_2(f)))$ implies that $f(A) \le 1/4$ 
\end{lem}

\begin{proof}
The first estimate is classical. See the comments before \cite[Lemma 3.1]{hm} for a proof. 
To prove the second one, we decompose $f$ as
\[
f= f(A) f_1 + (1-f(A)) f_2 \quad \hbox{where} \quad f_1 =\frac {\indiq_A}{f(A)} f \quad \text{and} \quad
 f_2 =\frac {\indiq_{A^c}}{1-f(A)} f.
\]
The entropy of $f$ may be rewritten as
\[
H(f)  = f(A) \log f(A) +(1-f(A)) \log(1-f(A))+ f(A) H(f_1) +  (1- f(A)) \, H(f_2).
\]
Since $m_2(f_2) \le (1-f(A))^{-1} m_2(f)$, the application of (i) (with $q=2$ and $\lambda =1$) to 
$f_2$ leads us to $H(f_2) \ge -C_{2,1} - (1-f(A))^{-1} m_2(f)$. 
Since $x \log x \ge -1/2$ for $x \in [0,1]$, we find
\[
H(f) \ge -1 -(1-f(A)) C_{2,1} - m_2(f) + f(A) H(f_1) \geq -C - m_2(f) + f(A)H(f_1).
\]
But Supp $f_1\subset A$ classically implies, by the Jensen inequality, that $H(f_1) \geq - \log |A|$.
Hence
\[
H(f) \ge  -C - m_2(f) - f(A) \log |A|.
\]
The conclusion follows.
\end{proof}

We now state two useful properties of the fractional and weighted Fisher informations.

\begin{lem} \label{lem:IrIw} Let $\gamma<0$ and $r\in(1/2,1)$.

\vip

(i) The weighted and fractional Fisher informations are super-additive: for all $N \geq 1$, all 
$F \in \PP_{sym}((\rd)^N)$, all $k=1,\dots, N$, denoting by $F_k\in\PP_{sym}((\rd)^k)$ the $k$-marginal of $F$,
$$
I^r(F_k) \le I^r(F) \quad \text{and} \quad 
I_\gamma(F_k) \le I_\gamma(F).
$$

(ii) The fractional Fisher information can be controlled by weighted Fisher information: 
for all $N\geq 1$, all $F \in \PP_{sym}(\rd^N)$,
$$
I^r(F) \le C_q (I_\gamma(F))^r \, (1+m_q(F))^{1-r} \quad 
\text{with } q := \frac{|\gamma| r}{1- r}.
$$
\end{lem}

\begin{proof}
Since $F$ is symmetric, we can write
$$ 
I^r(F) =  \int_{(\rr^3)^N} |\nabla_{v_1} \log F(\bm v^N) |^{2r}\,F(d\bm v^N)
\quad \hbox{and} \quad
I_\gamma(F) =  \int_{(\rr^3)^N} \langle v_1 \rangle^\gamma |\nabla_{v_1} \log F(\bm v^N) |^2\,F(d\bm v^N).
$$
The super-additivity of the fractional Fisher information is a consequence 
of the convexity of the fonction $\Psi_r:\R^+_\ast \times \R^3 \mapsto \R^+$ defined by
$$
\Psi_r(a,b) := \frac{|b|^{2r}}{a^{2r-1}}.
$$
Computing its Hessian matrix, we find
\begin{align*}
\nabla^2 \Psi_r (a,b) =& \frac{2r(2r-1)|b|^{2r-2}}{a^{2r+1}} 
\begin{pmatrix} |b|^2 & -ab^* \\
- ab & \frac{a^2}{2r-1}\big(I_3+(2r-2)\frac{bb^*}{|b|^2}\big)
\end{pmatrix}\\
\geq& \frac{2r(2r-1)|b|^{2r-2}}{a^{2r+1}} 
\begin{pmatrix} |b|^2 & -ab^* \\
- ab & a^2 \frac{bb^*}{|b|^2}
\end{pmatrix}
\end{align*}
which is always non-negative if $2r-1 \ge 0$. 
Then, we have, for $1 \le k \le N$, setting $\bm v^N_k= (v_1, \ldots v_k)$ and $\bm v^N_{N-k}= (v_{k+1}, \ldots v_N)$,
by the Jensen inequality,
\begin{align*}
\frac{|\nabla_{v_1} F_k(\bm v^N_k)|^{2r}}{F_k(\bm v^N_k)^{2r-1}} 
& =   \Psi_r \Bigl(F_k(\bm v^N_k), \nabla_{v_1} F_k(\bm v^N_k)\Bigr) \\
& = \Psi_r \Big(  \int_{(\rr^3)^{N-k}} F(\bm v^N_k, \bm v^N_{N-k}) \,d\bm v^N_{N-k},
\int_{(\rr^3)^N} \nabla_{v_1} F(\bm v^N_k, \bm v^N_{N-k}) \,d\bm v^N_{N-k} \Big) \\
& \le \int_{(\rr^3)^{N-k}} \Psi_r \Big(   F(\bm v^N_k, \bm v^N_{N-k}), \nabla_{v_1} F(\bm v^N_k,
\bm v^N_{N-k}) \Big) \,d\bm v^N_{N-k}\\
& =  \int_{(\rr^3)^{N-k}} \frac{|\nabla_{v_1} F(\bm v^N_k,
\bm v^N_{N-k})|^{2r}}{F(\bm v^N_k, \bm v^N_{N-k})^{2r-1}} \,d\bm v^N_{N-k}.
\end{align*}
Integrating this inequality on $\bm v^N_k$ we obtain $I^r(F_k)\leq I^r(F)$.
Choosing next $r=1$, multiplying the inequality by $\langle v_1 \rangle^\gamma $ and integrating 
on $\bm v^N_k$ we obtain $I_\gamma(F_k)\leq I_\gamma(F)$.

\vip

For the second point, we use the H\"older inequality:
\begin{align*}
I^r(F) & = \int_{(\rd)^N} \Bigl( | \nabla_{v_1} \log F(\bm v^N) |^{2r} F^r(\bm v^N) \langle v_1 \rangle^{\gamma r}\Bigr)  
\Big(F^{1-r}(\bm v^N) \langle v_1 \rangle^{- \gamma r}\Big)d\bm v^N \\
& \le  \bigl[ I_\gamma(F) \bigr] ^r \Big( \int_{(\rd)^N}  \langle v_1
\rangle^{- \gamma r/(1-r)} \,F(d\bm v^N) \Big)^{1-r}.
\end{align*}
Recalling that $q=|\gamma| r/(1-r)$, the conclusion follows.
\end{proof}

We next establish some kind of Gagliardo-Nirenberg-Sobolev inequality in $\rr^3$ involving 
fractional and weighted Fisher informations.

\begin{lem} \label{lem:FishInteg1}
Let $\gamma<0$. For any $r \in (1/2 , 1)$,
for  $p := 3/(3-2r)$ and  $q:= | \gamma| r/(1-r)$, for any $f \in \PP(\R^3)$,
\begin{align*}
\| f \|_{L^p}  \leq C_r\,   I^r(f) \leq C_{r,\gamma} \,   (I_\gamma(f))^r \, (1+m_q(f))^{1-r}
\end{align*}
\end{lem}

\begin{proof} 
Using the H\"older inequality, we can write, for any $p' \in (0,2r)$
$$
\|\nabla f\|_{L^{p'}}^{p'}  = \intrd \Big(\frac{|\nabla f|^{2r}}{ f^{2r-1}}\Big)^{\frac {p'}{2r}} f^{\frac {(2r-1)p'}{2r}}
\leq \Big( \intrd \frac{|\nabla f|^{2r}}{ f^{2r-1}} \Big)^{\frac {p'}{2r}} 
\Big( \intrd  f^{\frac {(2r-1)p'}{2r-p'}} \Big)^{\frac{2r-p'}{2r}},
$$
whence 
$$
\|\nabla f\|_{L^{p'}}  \le ( I^r(f))^{1/2r}  \| f\|_{L^{(2r-1)p'/(2r-p')}}^{1-1/2r}.
$$
Together with the Gagliardo-Nirenberg-Sobolev inequality, it comes, with $1/p = 1/p'-1/3$ 
(which is well-defined since $p'<2$)
$$
\| f \|_{L^{p}} \leq C_{p'} \|\nabla f\|_{L^{p'}} \leq C_{p'} (I^r(f))^{1/2r}  
\| f\|_{L^{(2r-1)p'/(2r-q)}}^{1- 1/2r}.
$$
This inequality becomes really interesting when $p  = (2r-1)p'/(2r-p')$, which leads to the choice 
$p' = 3/[2(2-r)]$. It that case $p = 3/(3-2r)$ and we find
$$
\| f\|_{L^{p}} \leq C_{r} I^r(f).
$$
Using finally Lemma \ref{lem:IrIw}-(ii) completes the proof.
\end{proof}

We deduce that pairs of particles of which the law has a finite
weighted Fisher information are not too close.

\begin{lem} \label{lem:FishInteg}
Let $\gamma<0$. Consider $F \in \PP(\rr^3 \times \rr^3)$ and 
$(\XX_1,\XX_2)$ a $F$-distributed random variable.
For any $\kappa \in (0,2)$ and any $q > |\gamma| \max \{ 1, \kappa/(2-\kappa)\}$, 
there exists  $C_{\kappa,q}$ such that, with $r:= q/(q - \gamma) \in (1/2,1)$,
$$
\E (|\XX_1 - \XX_2|^{-\kappa}) = \int_{\R^3 \times \R^3}\frac{F(dx_1,dx_2)}{|x_1-x_2|^\kappa} \le
C_{\kappa,q} \, [(I_\gamma (F))^r (1+m_q(F))^{1-r} + 1].
$$
\end{lem}

\begin{proof}  We introduce the unitary linear transformation $\Phi:\rd\times\rd \mapsto \rd\times \rd$ defined by
$$
\Phi  (x_1,x_2) =  \frac1{\sqrt 2}  
\bigl(x_1- x_2,x_1+x_2\bigr) =:   (y_1,y_2).
$$
Let $\tilde F := F \circ \Phi^{-1}$ and 
$\tilde f$ the first marginal of $\tilde F$ ($\tilde f$ is the law of $2^{-1/2}(\XX_1- \XX_2)$). 
A simple substitution shows that  $I^r(\tilde F) \le C_r I^r(F)$ for some constant $C_r$.
Next, we cannot apply  Lemma \ref{lem:IrIw}-(i) because $\tilde F$ is not symmetric,
but we obviously have $I^r (\tilde f) \le 2 \,  I^r (\tilde F)$.
We now write
\begin{align*}
\int_{\R^3 \times \R^3}\frac{F(x_1,x_2)}{|x_1-x_2|^\kappa} \, dx_1dx_2
=   2^{\kappa /2} \int_{\R^3 \times \R^3}\frac{\tilde
F(y_1,y_2)}{|y_1|^\kappa} \, dy_1dy_2
= &2^{\kappa /2}  \int_{\R^3}\frac{\tilde f (y)}{|y|^\kappa} \, dy \leq C_{\kappa,p}(1+\|\tilde f\|_{L^p})
\end{align*}
by \eqref{weakb2}, with $p := 3(q-\gamma)/(q-3\gamma)>3/(3-\kappa)$ (thanks to our condition on $q$).
We then use Lemma \ref{lem:FishInteg1} (we have  $p=3/(3-2r)$) to get
$$
\int_{\R^3 \times \R^3}\frac{F(x_1,x_2)}{|x_1-x_2|^\kappa} \, dx_1dx_2
\leq C_{\kappa,q} (1+I^r( \tilde f)) \leq C_{\kappa,q} (1+I^r(F)).
$$
Lemma \ref{lem:IrIw}-(ii) (observe that $q = |\gamma| r/(1-r)$) allows us to conclude.
\end{proof}

\subsection{Many-particle weighted Fisher information} \label{sec:prel2}

We finally need a result showing that if the particle distribution of the $N$-particle system has 
a uniformly bounded weighted Fisher information, then any limit point of the associated empirical
measure has finite expected weighted Fisher information. Such a result is a consequence of some representation 
identities for {\it level-3 functionals} as first observed by Robinson-Ruelle in \cite{RR} for the entropy in a 
somewhat different setting. 
Recently in \cite{hm}, this kind of representation identity has been extended to the 
Fisher information. The proof is mainly based on the
De Finetti-Hewitt-Savage representation theorem~\cite{Hewitt-Savage,deFinetti1937}
(see also \cite{hm}).
Unfortunately, we cannot apply directly the results of \cite{hm}.

\begin{thm}\label{th:levl3tH&tI} 
Let $\gamma<0$.
Consider, for each $N\geq 2$,  a probability measure $F^N \in \Psym((\R^3)^N)$. 
For $j\geq 1$, denote by $F^N_j\in \PP((\R^3)^j)$ the $j$-th marginal of $F^N$.
Assume that there exists a compatible sequence $(\pi_j)$ of symmetric probability
measures on $(\rr^3)^j$ so that $F^N_j \to \pi_j$ in the weak sense of measures in 
$\PP((\rr^3)^j)$. 
Denoting by $\pi \in \PP(\PP(\R^3))$ the probability measure associated to the 
sequence $(\pi_j)$ thanks
to the De Finetti-Hewitt-Savage theorem,  there holds 
$$
\int_{\PP(\R^3)} I_\gamma (f) \, \pi(df)   = \sup_{j \ge 1} I_\gamma(\pi_j)
 \leq \liminf_{N \to \infty} I_\gamma (F^N).
$$
\end{thm}

The De Finetti-Hewitt-Savage theorem asserts that for a 
sequence $(\pi_j)$ of symmetric probability on $E^j$ (for some measurable space $E$), 
compatible in the sense that 
the $k$-marginal of $\pi_j$ is $\pi_k$ for all $1\leq k \leq j$, there exists a unique probability measure
$\pi \in \PP(\PP(E))$ such that $\pi_j= \int_{\PP(E)} f^{\otimes j} \pi(df)$ for all $j\geq 1$. See for instance 
\cite[Theorem 5.1]{hm}.

\begin{cor}\label{tbu}
Let $\gamma<0$. Consider, for each $N\geq 2$,  a probability measure $F^N \in \Psym((\R^3)^N)$,
and $(X_1^N,\dots,X_N^N)$ with law $F^N$. Assume that $\mu_N := N^{-1}\sum_{1}^N \delta_{X_i^N}$
goes in law to some (possibly random) $\mu \in \cP(\rd)$. Then
$$
\E[I_\gamma(\mu)] \leq \liminf_{N \to \infty} I_\gamma (F^N).
$$
\end{cor}

\begin{proof}
Denote by $\pi \in \cP(\cP(\rd))$ the law of $\mu$ and, for $j\geq 1$, by $\pi_j= \int_{\PP(E)} f^{\otimes j} \pi(df)$.
The corollary immediately follows from Theorem \ref{th:levl3tH&tI} once we have checked that 
for all $j\geq 1$, $F^N_j$ goes weakly to $\pi_j$. But this is an easy and classical consequence
of the fact that  $\mu_N$ goes in law to $\mu$, see e.g. Sznitman \cite[I-Proposition 2.2 and Remark 2.3]{SSF}
or \cite[Lemma 2.8]{hm}.
\end{proof}

\vip
Theorem \ref{th:levl3tH&tI} is a consequence of \cite[Lemma 5.6]{hm} and of the 
following series of properties.

\begin{lem}\label{lem:propHIJ} Let $\gamma<0$. The weighted Fisher information
satisfies the following properties.

\vip

(i) For any $j\geq 1$, $I_\gamma : \PP ( (\R^3)^j) \to  \R \cup
\{+\infty \}$ is non-negative,
convex, proper and lower semi-continuous for the weak convergence.

\vip

(ii) For all $j\geq 1$, all $f \in \PP(\R^3)$, $I_\gamma (f^{\otimes j}) = I_\gamma(f) $.

\vip

(iii) For all $F \in \Psym( (\R^3)^j)$, all $\ell,n \geq 1$ with 
$j = \ell + n$ , there holds
$j \, I_\gamma (F) \ge \ell \, I_\gamma(F_\ell) + n \, I_\gamma (F_n)$, where $F_\ell\in \PP( (\R^3)^\ell)$ 
stands for the $\ell$-marginal of $F$.

\vip

(iv) 
The functional $\cI_\gamma : \PP(\PP(\rr^3)) \to \R \cup \{ \infty \}$ 
defined by 
$$
\cI_\gamma (\pi) := \sup_{j \ge 1} I_\gamma(\pi_j) \quad \hbox{where} \quad \pi_j:=\int_{\PP(\rr^3)}f^{\otimes j} \pi(df) \in
\cP((\rd)^j)
$$
is affine in the following sense. For any  $\pi \in
\PP(\PP(\rr^3))$
and any partition of $\PP(\rr^3)$ by some sets $\omega_i$, $1\le i \le M$, such that
$\omega_i$ is an open set in $\rr^3 \backslash (\omega_1 \cup \ldots \cup
\omega_{i-1})$ for any $1\le i \le M-1$ and $\pi(\omega_i) > 0$ for any 
$1\le i \le M$, defining 
$$
\alpha_i := \pi(\omega_i) \quad \hbox{and} \quad \gamma^i := \frac{1}{\alpha_i} \, 
{\bf 1}_{\omega_i} \, \pi \in \PP(\PP(\rr^3))
$$
so that 
$$
\pi = \alpha_1 \, \gamma^1 + ... + \alpha_M \, \gamma^M
\quad \hbox{and} \quad \alpha_1   + ... + \alpha_M =1,
$$
there  holds
$$
\cI_\gamma(\pi) = \alpha_1 \, \cI_\gamma(\gamma^1) + \ldots + \alpha_M \, \cI_\gamma(\gamma^M).
$$
\end{lem}

\begin{proof}[Proof of Lemma~\ref{lem:propHIJ}.] We only sketch the proof,
which is roughly an adaptation to the weighted case of the proof of~\cite[Lemma 5.10]{hm}.

\vip

{\it Step 1. } We first prove point (i).
Let us present an alternative expression of the weighted Fisher information: 
for $F \in \PP(( \R^3)^j)$, it holds that 
$$
I_\gamma (G)  :=  \frac1j\sup_{\psi   \in C_c^1( (\R^3)^j)^{3j}} \Big\{
 - \int_{(\R^3)^j} \diver_V \psi \,F(dV)
 - \sum_{i=1}^j \int \frac{|\psi_i|^2}4 \langle v_i \rangle^{-\gamma} \,F(dV^j) \Big\}.
$$
Again, the RHS term is well defined in  $\R$ because the function $| \psi |^2 / 4  - \hbox{div}_X\, 
\psi $ is continuous and
bounded for any $\psi   \in C_c^1( (\R^3)^j)^{3j}$. 
We immediately  
deduce that $I_\gamma$ is convex, lower semi-continuous 
and proper so that point (i) holds.

\vip

{\it Step 2.} Point (ii) is obvious from the Definition of $I_\gamma$.

\vip

{\it Step 3. } We now prove (iii), reproducing the proof of the same super-additivity property 
established for the usual Fisher information in \cite[Lemma~3.7]{hm}. We define for any $i \le j$
$$
\iota_i  := i \,  I_\gamma (F_i) =    \sup_{\psi   \in C_c^1( (\R^3)^i)^{3i}} 
\Big\{\int_{(\R^{3})^i }   \Bigl( \nabla_V \, F_i \cdot \psi  -
F_i \sum_{j \le i}  \, \frac{|\psi_j|^2}4 \langle v_j \rangle^{-\gamma}  \Bigr)\Big\}
$$
where the sup is taken on all $\psi =(\psi_1,\ldots,\psi_i)$, with all 
$\psi_\ell :  (\R^3)^i \to \R^3$. 
We then write the previous equality for $\iota_j$ and restrict the supremum over all $\psi$ such that for 
some $i \le j$:

$\bullet$ the $i$ first $\psi_\ell$ depend only on $(v_1,\ldots,v_i)$, with the
notation $\psi^i = (\psi_1, \ldots,\psi_i)$,

$\bullet$ the $j-i$ last $\psi_\ell$ depend only on $(v_{i+1},\ldots,v_j)$,
with the notation $\psi^{j-i} = (\psi_{i+1}, \ldots,\psi_j)$.

\noindent
We then have the inequality
\bean
\iota_j& \geq &  
\sup_{\psi^i ,\, \psi^{j-i} } 
\int_{(\R^3)^j }   \Big[\nabla_i F  \cdot \psi^i + \nabla_{j-i} F \cdot
\psi^{j-i} - F \, \frac14 \Big(\sum_{\ell \le i}|\psi_\ell|^2 \langle v_\ell \rangle^{-\gamma}
+ \sum_{\ell >i} |\psi_\ell|^2 \langle v_\ell \rangle^{-\gamma} \Big) \Big] \\
& = &  \sup_{\psi^i   \in C_c^1( ( \R^3)^i)^{3i}} 
\int_{(\R^3)^i }   \Big[\nabla_i F_i  \cdot \psi^i  
-  \frac{F_i}4 \sum_{\ell \le i}|\psi_\ell|^2 \langle v_\ell \rangle^{-\gamma} \Big]
\\ 
&& + 
\sup_{\psi^{j-i}   \in C_c^1( (\R^3)^{j-i})^{3(j-i)}} 
\int_{( \R^3)^{j-i} }   \Big[\nabla_{j-i} G_{j-i}  \cdot \psi^{j-i}  - \frac{F_{j-i}}4  \, 
\sum_{\ell >i} |\psi_\ell|^2 \langle v_\ell \rangle^{-\gamma}  \Big]
\\
& = & \iota_i + \iota_{j-i}.
\eean

\vip

{\it Step 3. } 
We do not prove here the affine caracter of $\cI_\gamma$ and refer to  \cite[Lemma 5.10]{hm} where the 
same property for the usual Fisher information is checked:
the presence of the bounded weight does not raise any difficulty. 
\end{proof}

\end{document}